\documentclass[10pt,a4paper]{article}
\usepackage[english]{babel}
\usepackage[T1]{fontenc}
\usepackage{graphicx, color}
\usepackage{epstopdf}
\usepackage{epsfig,psfrag}
\usepackage{hyperref}
\usepackage{amsmath, amsfonts, amsthm,amssymb, verbatim,mathrsfs}
\newtheorem{theorem}{Theorem}[section]
\newtheorem{definition}[theorem]{Definition}
\newtheorem{lemma}[theorem]{Lemma}

\newtheorem{corollary}[theorem]{Corollary}
\newtheorem{remark}[theorem]{Remark}

\newtheorem{proposition}[theorem]{Proposition}

\numberwithin{equation}{section}
\newcommand{\ii}{{\rm i}}
\newcommand{\Su}{{\rm Supp}}
\def\SF{\langle\FF\rangle}
\def\vv{\mathtt v}
\def\ww{\mathtt w}

\def\pp{\mathtt p}
\def\cc{\mathbf c}

\def\ii{\mathrm i}
\def\be{\mathbf e}
\def\ee{\mathbf e}
\def\Z{\mathbb Z}
\def\ZZ{\mathbb Z}
\def\CC{\mathbb C}

\def\R{\mathbb R}
\def\Q{\mathbb Q}
\def\KKK{\mathbb K}
\def\N{\mathbb N}
\def\T{\mathbb T}

\def\TT{\mathbb T}
\def\RR{\mathcal R}
\def\AA{\mathcal A}
\def\BB{\mathcal B}
\def\SS{\mathcal S}
\def\SSS{\mathcal S}
\def\KK{\mathcal K}
\def\MM{\mathcal M}
\def\WW{\mathcal W}
\def\FF{\mathcal F}

\def\PP{\mathcal P}

\def\HH{\mathcal H}
\def\DD{\mathbb D}
\def\DDD{\mathcal D}
\def\XX{\mathcal X}
\def\GG{\mathcal G}
\def\OO{\mathcal O}
\def\II{\mathcal I}
\def\ZZZ{\mathcal Z}
\def\CCC{\mathcal C}
\def\RRR{\mathcal R}
\def\NNN{\mathcal N}
\def\VV{\mathcal V}
\def\UU{\mathcal U}
\def\EE{\mathcal E}
\def\YY{\mathcal Y}
\def\Ht{h_\SS}

\newcommand{\de}{\delta}
\newcommand{\ga}{\gamma}
\newcommand{\al}{\alpha}
\newcommand{\pa}{\partial}
\newcommand{\la}{\lambda}
\newcommand{\kk}{\kappa}
\newcommand{\rrr}{\varrho}
\newcommand{\rr}{\rho}

\newcommand{\om}{\omega}
\renewcommand{\Re}{\mathrm{Re\, }}
\renewcommand{\Im}{\mathrm{Im\,}}
\newcommand{\wt}{\widetilde}
\newcommand{\wh}{\widehat}

\newcommand{\hyp}{\mathrm{hyp}}
\newcommand{\el}{\mathrm{ell}}
\newcommand{\mix}{\mathrm{mix}}
\newcommand{\adj}{\mathrm{adj}}

\newcommand{\loc}{\mathrm{loc}}
\newcommand{\glob}{\mathrm{glob}}

\newcommand{\inn}{\mathrm{in}}
\newcommand{\out}{\mathrm{out}}
\newcommand{\Id}{\mathrm{Id}}
\newcommand{\ol}{\overline}
\newcommand{\iii}{^{-1}}

\title{Growth of Sobolev norms for the 
analytic NLS on $\T^2$}
\author{M. Guardia\thanks{\tt Departament de Matem\`atiques, 
Universitat Polit\`ecnica de Catalunya, Diagonal 647, 08028
Barcelona, Spain,  marcel.guardia@upc.edu},  E. Haus\thanks{\tt Dipartimento di
Matematica e Applicazioni ``R. Caccioppoli'', Universit\`a degli Studi di Napoli
``Federico II'', Via Cintia, Monte S. Angelo, 80126 Napoli, Italy, 
emanuele.haus@unina.it}, M. Procesi\thanks{\tt Dipartimento di Matematica,
Sapienza Universit\`a di Roma, Piazzale Aldo Moro 5, 00185 Roma, Italy,
mprocesi@mat.uniroma1.it}}

\begin{document}
\maketitle
\abstract{We consider the  completely resonant 
non--linear
Schr\"odinger  equation on the two dimensional torus with any analytic gauge
invariant nonlinearity. Fix $s>1$. We show the existence of solutions of this
equation which achieve arbitrarily large growth of $H^s$ Sobolev norms. We also
give estimates for the time required to attain this growth.}

\tableofcontents
\section{Introduction}
Consider the   completely resonant defocusing non--linear Schr\"odinger 
equation on  the torus $\T^2=(\R/2\pi\Z)^2$ (NLS for brevity),
\begin{equation}\label{main}-\ii u_t+\Delta u= 2d|u|^{2(d-1)}u+2G'(|u|^2)u,\quad
 d\in\N,\; d\geq 2
\end{equation}
where  $G(y)$ is an analytic function (in the unit ball) with a zero of degree
at least $d+1$ (the coefficients $2d$ and 2 are just to have simpler formulas
later on).

It is well known \cite{Bourgain93, BGT} that equation  \eqref{main} is globally
well posed in time in $H^s$ for $s\geq 1$, and  defines  an infinite dimensional
Hamiltonian dynamical
system with respect to the energy functional
\[
H(u)=\int_{\T^2}\left(\frac12|\nabla
u|^2+|u|^{2d}+G(|u|^2)\right)\frac{dx}{(2\pi)^2}.
\]
It has also the following first integrals: the {\em momentum}  
\[
M(u)=\int_{\T^2} \bar u \nabla u\frac{d x}{(2\pi)^2}
\]    and  the {\em mass} 
\[
L(u)= \int_{\T^2} |u|^2\frac{d x}{(2\pi)^2},
\]
which is just the square of the $L^2$ norm.

The purpose of this paper is to study the problem of growth of Sobolev norms for
the equation  \eqref{main}. That is, to obtain orbits whose $s$-Sobolev norm, 
$s> 1$, defined as usual as
\[
 \|u\|_{H^s}=\sum_{k\in \ZZ^k}\langle k\rangle^{2s}|u_k|^2,\,\,\text{ where
}u(x)=\sum_{k\in\ZZ^2}u_k e^{\ii kx}\,\,\text{ and }\langle k\rangle
=\sqrt{1+|k|^2},
\]
grows by an arbitrarily large factor. Note that the $H^1$ norm is almost
constant due to energy conservation. 

The importance of growth of Sobolev norms stems from the fact that it implies
transfer of energy from low to high modes as time grows, a phenomenon related to
the so called {weak turbulence}.

 In \cite{Bourgain00b}, Bourgain posed the following question: are there
solutions of the cubic nonlinear Schr\"odinger equation 
\[
 -\ii u_t+\Delta u= |u|^{2}u
\]
in $\TT^2$ such that $\|u(t)\|_{H^s}\rightarrow+\infty$ as $t\rightarrow
+\infty$?

This question  has been recently positively answered for the cubic
NLS on  $\R\times \T^2$ in \cite{HaniPTV13}. It is believed to be also true in
the original setting $\TT^2$ but the question remains open on any compact
manifold. 

In the past years  there have been a set of results proving the existence of
solutions of the {\em cubic NLS} with arbitrarily large  
{\em finite} growth. The first result, proven in \cite{Kuksin97b}, was for large
data. Namely given a large constant $\KK>0$ there exists a solution whose 
initial  Sobolev norm is large with respect to $\KK$ which after certain time
$T$ attains a Sobolev norm  satisfying $\|u(T)\|_{H^s}\geq \KK\|u(0)\|_{H^s}$.
In the context of small initial data, the breakthrough result was proved in
\cite{CKSTT} for the cubic NLS. 
The authors prove that given two constants $\mu\ll1$ and $\CCC\gg1$, there are
orbits whose Sobolev norms grow from $\mu$ to $\CCC$ after certain time $T>0$.
Estimates for the time needed to attain such growth are given in \cite{GuaKal}.
Note that small initial Sobolev norm implies that the mass and the energy remain
small for all times. 

Growth of Sobolev norms has drawn considerable attention since the 90's not only
for NLS on the two torus but also in more general settings and for other
dispersive PDEs. Let us briefly review  the literature  on the subject. In
\cite{Bourgain96,Staffilani97, CollianderDKS01,Bourgain04,Zhong08,
CatoireW10,Sohinger10a,Sohinger10b, Sohinger11,CollianderKO12}, the authors
obtain polynomial upper bounds for the growth of Sobolev norms.

Arbitrarily large finite growth was first proven in  \cite{Bourgain96},  for the
wave equation with a cubic nonlinearity but with a spectrally
defined Laplacian. As we have already mentioned the same result has been
obtained for the cubic NLS in \cite{Kuksin97b, CKSTT, GuaKal}. The results in
\cite{CKSTT, GuaKal}  have been generalized to the cubic NLS with a convolution
potential in \cite{Guardia14} and the result in \cite{CKSTT} has been
generalized to the quintic NLS in \cite{HausProcesi}. Large finite growth of
Sobolev norms  has also been obtained in \cite{GerardG11,Pocovnicu12} for
certain nonlinear half-wave equations. In \cite{CarlesF12}, the authors obtain
orbits of the cubic NLS which undergo spreading of energy among the modes.
Nevertheless, this spreading does not lead to growth of Sobolev norms. Similar
phenomena were discussed in \cite{GPT,GT,HT}.

Finally, the unbounded growth of Sobolev norms has been recently obtained for the Szeg\H o equation by G\'erard and collaborators following the work initiated in \cite{GerardG10,Pocovnicu11,GerardG15}. Unbounded Sobolev growth, as it has been mentioned before,  has  been also  proven for the cubic NLS in $\R\times \TT^2$ in
\cite{HaniPTV13}. In \cite{Hani11, Hani12} unbounded growth  is shown in  a
pseudo partial differential equation which is a simplification of cubic NLS.

 A dual point of view to instability is to construct quasi-periodic orbits.
These are  solutions which are global in time and whose Sobolev norms
are approximately constant. Among the relevant literature we mention
\cite{Wayne90,Poschel96a,KuksinP96,Bourgain98,BB1,Eliasson10,GYX,BBi10,Wang2,PX,
BCP,PP}.
  Of particular interest  are the recent results obtained through KAM theory
which gives information on linear stability close to the quasi-periodic
solutions.
  In  particular the paper \cite{PP13} proves the existence of  both stable and
unstable tori (of arbitrary finite dimension) for the cubic NLS. In principle
such unstable tori could be used to construct orbits whose Sobolev norm grows,
indeed
 in finite dimensional systems diffusive orbits are usually constructed by
proving that the stable and unstable manifolds of a {\em chain of unstable tori}
intersect. Usually however the intersection of {\em stable/unstable} manifolds
is deduced by  dimensional arguments, by constructing chains of  co-dimension
one tori. 
 In the infinite dimensional case this would mean constructing  almost-periodic
orbits, which is an open problem except for very special cases such as
integrable equations or equations with infinitely many external parameters (see
for instance \cite{ChierchiaP95,Poschel02,Bou}). 
  
  In \cite{CKSTT}, \cite{GuaKal}, \cite{HausProcesi} (and the present paper) this problem is avoided
 by taking advantage of the specific form of the equation. First one reduces to
an approximate equation, i.e. the Hamitonian flow of the {\em first order
Birkhoff
normal form} $H_{\rm Res}$, see \eqref{def:HamAfterNormalForm}. 
  Then for this dynamical system one proves directly the existence of 
    chains of one dimensional unstable tori (periodic orbits) together with
their heteroclinic connections.   Next, one proves the existence of a {\em
slider
solution} which shadows the heteroclinic chain in a finite time. Finally, one
proves the persistence of the slider solution for the full NLS by scaling
arguments. 
   
    The fact that one may construct a heteroclinic chain for the Birkhoff normal
form Hamiltonian \eqref{def:HamAfterNormalForm} relies on the property  that
this Hamiltonian is {\em non-integrable} but has nonetheless many invariant
subspaces on which the dynamics simplifies significantly. More precisely given a
set $\mathcal S\subset \ZZ^2$ we define the subspace
    $$
    U_\mathcal S:= \{u\in L^2(\TT^2):  \quad u(x)= \sum_{j\in \mathcal S} u_j
e^{\ii j\cdot x}\}\,,
    $$
and consider the following definitions.
\begin{definition}[Completeness]\label{completeness}
We say that a set $\SS\subset\Z^2$ is {\em complete} if $U_\SS$ is invariant for
the dynamics of $H_{\rm Res}$.
\end{definition}
\begin{definition}[Action preserving]\label{integrable}
A complete set $\SS\subset\Z^2$ is said to be {\em action preserving} if all the
actions $|u_j|^2$ with $j\in \SS$ are constants of motion for the dynamics of
$H_{\rm Res}$ restricted to $U_\SS$.
\end{definition}
The conditions under which a given set $\SS$ is complete or action preserving
can be rephrased more explicitly by using the structure of $H_{\rm Res}$.
\begin{definition}[Resonance]\label{risuona} Given a $2d$-tuple
$(j_1,\dots,j_{2d})\in (\ZZ^2)^{2d}$ we say that it is a resonance of order $d$
if
$$
\sum_{i=1}^{2d} (-1)^i j_i =0,\qquad\qquad \sum_{i=1}^{2d} (-1)^i |j_i|^2 =0.
$$
\end{definition}
Now $\SS$ is complete if and only if for any $(2d-1)$-tuple
$(j_1,\dots,j_{2d-1})\in \SS^{2d-1}$ there does not  exist any  $k\in
\Z^2\setminus\SS$ such that $(j_1,\dots,j_{2d-1},k)$ is a resonance.
Similarly $\SS$ is action preserving if all resonances $(j_1,\dots,j_{2d})\in
\SS^{2d}$ are {\em trivial}, namely there exists a permutation such that
$(j_1,\dots,j_d)=(j_{d+1},\dots,j_{2d})$.


Now a good strategy is to look for  a finite dimensional set $\SS$ which is
complete but not action preserving,  where we can prove existence of diffusive
orbits. A difficulty stems from the fact that {\em generic} choices of $\SS$ are
action preserving (see \cite{PP}).
As a preliminary step one may study simple sets $\SS$  where the dynamics is
integrable and one can exhibit some growth of Sobolev norms.
In particular one would like to produce a set which has two periodic orbits
linked by a heteroclinic connection, since this is a natural building block for
a heteroclinic chain.
A natural choice is to fix a {\em simple} resonance $\SS=\{j_1,\dots,j_{2k}\}$
of order $k$, namely a resonance  which does not factorize as sum of two
resonances of lower order.
Clearly any set of this form produces non-trivial resonances (of order $d$) in
$\SS^{2d}$ for all $d\geq k$. Sets of this type have been studied for the
quintic NLS, see \cite{GT} and \cite{HausProcesi}, Appendix C. For $d>2$, the
Hamiltonian $H_{\rm Res}$ restricted
to a simple resonance can be explicitly written (with a relatively heavy combinatorics)
and one easily sees that there are in fact two periodic orbits, however we are
not able to give a general statement about the existence of a heteroclinic
connection. Direct  computations show that a single resonance of order
two (i.e. one rectangle)  produces a heteroclinic connection only  if $d\leq 5$,  while {\em no simple resonance of order
$k>2$ } produces a heteroclinic connection for $d\leq6$ (we expect this to be
true for any $d$, but we have not performed the computations in the case
$k>2,d>6$).

A crucial fact is the following: consider a {\em large} set $\SS$ which is the
union of $q\gg d$ rectangles  and such that $\SS$ does not contain any {\em
simple} resonance apart from these rectangles. Then this system {\em always} has
two periodic orbits linked by a heteroclinic connection. Indeed, after some
symplectic reductions, it turns out that $H_{\rm Res}$ is a {\em small}
perturbation of the one obtained for the cubic NLS restricted to a single
rectangle. Note that this procedure works only for rectangles: if $\SS$ is the
union of $q\gg d$ resonances of order $k>2$ and if these are the only simple
resonances in $\SS$, then after the same symplectic reductions one is left with
a small perturbation of an action preserving system, having again two periodic
orbits but no heteroclinic connection between them. While clearly this does not
in any way constitute a proof, it gives some interesting negative evidence about
the possibility of extending these results to the NLS on the circle.

\subsection{Main results}
The purpose of this paper is to generalize the results of \cite{CKSTT} and
\cite{GuaKal} to the nonlinear Schr\"odinger equation  \eqref{main} with any
$d\geq 2$. The case $d=3$ was treated in \cite{HausProcesi} where it is proven
a result analogous to the one in \cite{CKSTT}.

This is the main result of our paper.

\begin{theorem}\label{thm:MainTheorem}
Let $d\geq 2$ and $s>1$. There exists $c>0$ with the following property:
for any  large $\CCC\gg 1$ and small $\mu\ll 1$,  there exists a 
global solution $u(t)=u(t,\cdot)$
of \eqref{main} and a time $T$ satisfying
\[
 T \leq e^{\left(\frac{\CCC}{\mu}\right)^c}
\]
such that
\[
\|u(0)\|_{H^s}\leq\mu \,\,\text{ and }\,\,\|u(T)\|_{H^s}\geq \CCC.
\]
\end{theorem}

\begin{remark}
Theorem \ref{thm:MainTheorem} is still valid or the focusing nonlinear Schr\"odinger equation
\[
-\ii u_t+\Delta u= -2d|u|^{2(d-1)}u+2G'(|u|^2)u
\]
where the lowest order of the nonlinearity has opposite sign, in the following
sense.

In the focusing case, the NLS equation \eqref{main}  is not globally well-posed
on $H^s(\T^2)$, so one can not infer a-priori existence of a solution on the
time interval $[0,T]$. However, one can recover existence (and uniqueness) of
the solution in low Sobolev norm by two independent arguments. First, the proof
of Theorem \ref{thm:approximation} is totally independent of the sign of the
nonlinearity, and it implies that, with our chioce of initial data, a solution
$u(t,x)$ with space Fourier coefficients in $\ell^1(\ZZ)$ exists for all
$t\in[0,T]$. Second, for small $L^2(\T^2)$ solutions (as is the case in the
present paper), the focusing NLS Hamiltonian controls the $H^1$-norm, thanks to
the Gagliardo-Nirenberg inequality: thus, the conservation of the Hamiltonian,
together with the local well-posedness of equation \eqref{main} proved in
\cite{Bourgain93}, gives global-in-time existence of a solution in $H^1(\T^2)$.
The solution might very well blow up in $H^s$ before the time $T$, but our proof
still 
implies that the (possibly infinite) quantity $\| u(T) \|_{H^s}$ satisfies
$\|u(T)\|_{H^s}\geq \CCC$.
\end{remark}


Theorem \ref{thm:MainTheorem} generalizes the results in \cite{GuaKal}
for the cubic NLS.
In \cite{GuaKal} the authors give two results. In the first
result (Theorem 1 in \cite{GuaKal}), they only measure growth of Sobolev
norms and do not assume that the initial Sobolev norm is small. Then,  they
obtain {\em polynomial} time estimates 
with respect to the growth. In the second result (Theorem 7 in \cite{GuaKal})
they impose small initial Sobolev norm and large final Sobolev norm and obtain
slower time estimates. 

In the present setting we cannot get improved estimates as in Theorem 1 of
\cite{GuaKal} by assuming that only the $L^2$ norm of the initial datum is small.
 The reason is that the higher the degree of the nonlinearity the
more interactions between modes exist. Certainly, more interactions should imply
more paths to obtain growth of Sobolev norm and therefore similar or faster time
estimates. Nevertheless, they also make the problem harder to handle. The proof
of Theorem \ref{thm:MainTheorem} follows the approach developed by \cite{CKSTT}
and analyzes very particular orbits which are essentially supported on a finite
number of modes (see Section \ref{sec:descriptionproof}). Thus, one needs to
keep track of the large number of interactions between modes so that the energy
is not spread to a larger and larger number of modes as time evolves. This is
more difficult for equation \eqref{main} with $d\geq 3$ than for the cubic NLS.
To avoid this spreading, we have to choose slower orbits.

It is reasonable to expect that polynomial time estimates are still true for equation
\eqref{main} with $d\geq 3$ but one needs some new ideas in the analysis of the finite
set of modes on which the orbit attaining growth of Sobolev norms is supported.
This is explained in Remark \ref{rmk:NoSpreading}.


Theorem \ref{thm:MainTheorem}
is proven in
Section \ref{sec:descriptionproof}. Then, Sections
\ref{sec:GenerationSetsCombinatorics}-\ref{sec:Approximation} contain the proofs
of the partial results needed in Section \ref{sec:descriptionproof}.

{\it Acknowledgements:}
The authors thank Zaher Hani and the anonymous referees for their helpful
suggestions.
The first  author is partially supported by the  Spanish
MINECO-FEDER Grants MTM2012-31714 and   MTM2015-65715 and the Catalan Grant 
2014SGR504; the second
author is supported by ERC under FP7-ERC Project 306414, HamPDEs and Programme
STAR, funded by UniNA and Compagnia di San Paolo; the third author is supported
by ERC under FP7-ERC Project 306414, HamPDEs.

\section{Structure of the proof}\label{sec:descriptionproof}
\subsection{Basic notations}
We write the differential equation for the Fourier modes
\begin{equation}\label{FR}
u(t,x):= \sum_{k\in \Z^2} u_k(t) e^{\ii (k, x)}, 
\end{equation} 
associated to \eqref{main}. It is of the form
\begin{equation}\label{def:NLS:ode}
 \dot u_k=2 \ii \pa_{\ \ol{u_k}}\ \HH(u, \ol u).
\end{equation}
Thus, it is Hamiltonian with respect to  the symplectic form $\Omega= \frac\ii
2\sum_{k\in\ZZ^2} d u_k\wedge d \bar u_k$ and the Hamiltonian
\begin{equation}\label{def:HamForFourier}
\HH (u, \ol u)=\DDD (u, \ol u)+\GG  (u, \ol u)
\end{equation}
with
\begin{align*}
 \DDD (u,\ol u)&=\frac12\sum_{k\in\ZZ^2}|k|^2|u_k|^2\\
\GG  (u,\ol u)&=\sum_{k_i\in \Z^2: \sum_{i=1}^{2d}(-1)^i k_i=0}\hskip-30pt
u_{k_1}\bar u_{k_2}u_{k_3}\bar u_{k_4}\ldots u_{k_{2d-1}}\bar u_{k_{2d}}\\
&+\int_{\T^2} G(|u|^2)\frac{dx}{(2\pi)^2}. \\
\end{align*}
 We may write, for any $r\in\N$,
\begin{equation}\label{albet}
\begin{split}
 [u]^{2r}:=&\frac{1}{(2\pi)^2}\int_{\T^2} |u|^{2r}dx\\
=&\sum_{k_i\in \Z^2,\ \sum_i(-1)^ik_i=0} u_{k_1}\bar u_{k_2}u_{k_3}\bar
u_{k_4}\ldots u_{k_{2r-1}}\bar u_{k_{2r}}\\
=& \sum_{\alpha,\beta\in (\N)^{\Z^2}: \atop
{|\alpha|=|\beta|=r},\sum(\alpha_k-\beta_k)k=0}
\binom{r}{\alpha}\binom{r}{\beta}u^\alpha\bar u^\beta,
\end{split}
\end{equation}
where $\alpha:k\mapsto \alpha_k\in\N$,
$|\al|=\sum_{k\in\N}|\al_k|$,  $u^\alpha=\prod_k
u_k^{\alpha_k}$ and 
\[
 \binom{r}{\alpha}=\frac{r!}{\Pi_{k\in\N}\alpha_k!},
\]
where, since $|\al|=r$, only a finite number of $\al_k$ are different from
zero.

With this notation one clearly has 
\[
\GG(u,\ol u)= \sum_{r\geq d} c_{r} [u]^{2r}\,,\quad c_d=1 \,,\quad \sum_{r\geq
d} |c_r|
<\infty.
\]
\begin{remark}\label{rmk:gauge}
Since by hypothesis the mass is preserved, we may perform the  trivial phase
shifts $u_j \to e^{-2\ii f(L) t} u_j$. In this way the Hamiltonian becomes
$\DDD+\GG_1$ with $\GG_1= \GG - F(L)$ where $F$ is a primitive of $f$.
\end{remark}
 
 In the course of the paper we will need the following definition 
 \begin{definition}\label{integra}
 Given a set of complex symplectic variables $(z_k,\bar z_k)$ with the
symplectic form $\frac\ii 2 d z \wedge d \bar z$, we say that a monomial is {\em
action preserving} if it depends only on the {\em actions } $|z_k|^2$. This
naturally defines a projection on the subspace of action preserving polynomials
which we denote by $\Pi_I$.
 \end{definition}
\subsection{Birkhoff Normal Form}
We perform one step of Birkhoff normal form to reduce the size of the
non-resonant terms. We perform it in the $\ell^1$ space, which is defined, as
usual, by
\[
 \ell^1=\left\{u:\ZZ^2\rightarrow \CC: \|u\|_{\ell^1}=
 \sum_{k\in\ZZ^2}|u_k|<\infty\right\}.
\]
Recall that $\ell^1$ is a Banach algebra with respect to the convolution
product.  We consider a small ball centered at the origin,
\[
 B(\eta)=\left\{u\in\ell^1:\|u\|_{\ell^1}\leq\eta\right\}.
\]

\begin{theorem}\label{thm:NormalForm}
There exists $\eta>0$ small enough such that  there exists a symplectic
change of coordinates $\Gamma: B(\eta)\rightarrow B(2\eta)\subset\ell^1$, 
$u=\Gamma(a)$, 
which takes the Hamiltonian $\HH$ in \eqref{def:HamForFourier} into its Birkhoff
normal form up to order $2d$,
that is,
\begin{equation}\label{def:HamAfterNormalForm}
 \HH\circ\Gamma=\DDD+H_{\rm Res}+\RRR,
\end{equation}
where $H_{\rm Res}$ only contains  resonant terms, namely
\begin{equation}\label{Ham2}
H_{\rm Res}= \hskip-5pt\sum_{k_i\in \Z^2,\atop {\sum_i(-1)^ik_i=0\atop \sum_i
(-1)^i |k_i|^2=0}}\!\!\!\! a_{k_1}\bar a_{k_2}a_{k_3}\bar a_{k_4}\ldots
a_{k_{2d-1}}\bar a_{k_{2d}}= \hskip-30pt\sum_{\alpha,\beta\in (\N)^{\Z^2}:
|\alpha|=|\beta|=d\atop {\sum_k (\alpha_k-\beta_k)k=0\,,\;\sum_k
(\alpha_k-\beta_k)|k|^2=0}}
\hskip-10pt\binom{d}{\alpha}\binom{d}{\beta}a^\alpha\bar a^\beta\,.
\end{equation}
The vector field $X_\RRR$, associated to
the Hamiltonian $\RRR$, satisfies
\[
 \|X_\RRR\|_{\ell^1}\leq \OO\left(\|a\|_{\ell^1}^{2d+1}\right).
\]
Moreover, the change of variables $\Gamma$ satisfies
\[
 \left\|\Gamma-\mathrm{Id}\right\|_{\ell^1}\leq
\OO\left(\|a\|_{\ell^1}^{2d-1}\right).
\]
\end{theorem}
The proof of this theorem follows the same lines as the proof of Theorem 2 in
\cite{GuaKal}.

To study the Hamiltonian $ \HH\circ\Gamma$, we change to rotating coordinates to
remove the quadratic part of the Hamiltonian. We take
\begin{equation}\label{def:rotatingcoord}
 a_k=r_k{e^{\ii |k|^2 t}}.
\end{equation}
Then, $r$ satisfies the equation associated to the Hamiltonian
\begin{equation}\label{def:HamRotating}
 \HH'=H_{\rm Res}+\RRR',
\end{equation}
where
\begin{equation}\label{def:RemainderInRotating}
 \RRR'\left(\{r_k\}_{k\in\ZZ^2},t\right)=\RRR\left(\{r_k{e^{\ii |k|^2
t}}\}_{k\in\ZZ^2}\right).
\end{equation}
As a first step we study the dynamics of the truncated Hamiltonian $H_{\rm
Res}$. The associated equation is given by 
\begin{equation}\label{def:VFTruncated}
 -\ii\dot r=\EE(r)
\end{equation}
where
\begin{equation}\label{eq:HRes}
\EE_k(r)= 2d\sum_{k_i\in \Z^2,\atop {\sum_{i=1}^{2d-1} (-1)^ik_i=k\atop
\sum_{i=1}^{2d-1} (-1)^i |k_i|^2=|k|^2}}\!\!\!\! r_{k_1}\bar r_{k_2}r_{k_3}\bar
r_{k_4}\ldots r_{k_{2d-1}}.
\end{equation}
The Hamiltonian $H_{\rm Res}$ and the associated equation are scaling invariant
with respect to
\begin{equation}\label{def:rescaling}
 r^\rrr(t)= \rrr^{-1}r( \rrr^{-(2d-2)}t),\qquad  \rrr\in\R\setminus\{0\}.
\end{equation}

\subsection{The reduction to the Toy Model}
Following \cite{CKSTT}, we look for a finite set of modes which interact in a
very particular and symmetric way. This set was constructed for the cubic case
in \cite{CKSTT}  and in the quintic case in
\cite{HausProcesi}. The higher the degree of the nonlinearity, the more
complicated the interaction between the modes is. Here we follow the approach
developed in \cite{HausProcesi}. 
We start by defining an {\em acceptable} frequency set as follows.

\begin{definition}\label{acceptable}
Fix $N\gg1$, $s>1$. Then  $\SS\equiv\SS(N)\subset\ZZ^2$ is acceptable if the
following holds:
\begin{enumerate}
\item $\SS$ is the disjoint union of $N$ generations $\SS=\cup_{i=1}^N\SS_i$,
each of them having cardinality $n:=2^{N-1}$
\item $\SS$ satisfies the {\em norm explosion} property:
\begin{equation}\label{Normexp}
\frac{\sum_{k\in \SS_{N-2}}|k|^{2s} }{\sum_{k\in \SS_{3}}|k|^{2s}}> 
2^{(N-6)(s-1)}.
\end{equation}
\item The N dimensional subspace
\begin{equation}\label{christmas}
U_\SS:=\{ r\in \CC^{\Z^2}:\; r_k=0 \;\forall k \notin \SS\,,\quad r_{l}=r_{j}:=
b_i\;\; \forall i=1,\dots,N\,,\;\forall l,j\in \SS_i\}.
\end{equation}
is invariant under the flow of the Hamiltonian  $H_{\rm Res}$ defined in
\eqref{Ham2}.
\item The  flow of $H_{\rm Res}$ restricted to $U_\SS$  is Hamiltonian with
respect to  the symplectic form
$\frac{i}{2}\sum_j db_j\wedge d\bar b_j$ with:
\begin{equation}\label{def:toymodel:0}
\begin{split}
&h_\SS(b)= d! n^{d-1}\left(\sum_{i=1}^N|b_i|^2\right)^{d}+\\
&
+n^{d-2}d!d(d-1)\left\{\left(\sum_{i=1}^N|b_i|^2\right)^{d-2}\left[-\frac14\sum_
{i=1}^N|b_i|^4+ \sum_{i=1}^{N-1}\Re(b_i^2\bar
b_{i+1}^2)\right]+\frac{1}{n}\mathcal P\left(b,\bar b, \frac1 n\right)\right\}.
\end{split}
\end{equation}
where $n=2^{N-1}$ and  $\PP$ satisfies the following
properties:
\begin{enumerate}
\item $\mathcal P$ is a real coefficients polynomial in all its variables and 
it is homogeneous of degree $2d$ in $(b,\bar b)$.
\item  $\mathcal P$ is real, namely $\mathcal P(b,\bar b, \frac1 n)=\mathcal 
P(\bar b, b, \frac1 n)$.
\item $\mathcal P$ is Gauge preserving, i.e. it Poisson commutes with 
$J=\sum_{i=1}^N|b_i|^2$.
\item All the monomials in $\PP$ are of even degree in each $(b_i,\bar b_i)$.
This implies that for all $i=1,\dots, N$ the subspace $\{b_i=0\}$ is invariant
for
the flow of $h_\SS$. 
\item For $j=1,\ldots,N-1$, the subspace
$$
U_\SS^{j}:=\{ b\in \CC^N :\; b_i=0\,,\quad i\neq j,j+1\}
$$ is invariant with respect to the flow of $H_{\rm Res}$. Moreover, the
pullback of this Hamiltonian into $U_\SS^{j}$ is $j$-independent (up to an index
translation) and, as a function of  $(b_j,\bar b_j),(b_{j+1},\bar b_{j+1})$, is
symmetric with respect to the exchange $j\longleftrightarrow j+1$.
\item  
Given $i\neq j$, consider the monomials in $\PP$ which depend only on $(b_i,\bar
b_i),(b_j,\bar b_j)$  and are exactly of degree two in $(b_i,\bar b_i)$. Then if
$|i-j|\neq 1$ such monomials are {\em action preserving} namely of the form
$\chi_{ij}|b_j|^{2d-2}|b_i|^2$ (for some suitable coefficient $\chi_{ij}$).
Otherwise, if $|i-j|=1$ then they are either of the form
$\chi_{ij}|b_j|^{2d-2}|b_i|^2$ or of the form
$\rho_{ij}|b_j|^{2d-4}\Re(b_i^2\bar b_j^2)$. Moreover, $\chi_{ij}\equiv\chi$ is
independent of $i$ and $j$ and $\rho_{i,i+1}\equiv\rho$ is independent of $i$.
\end{enumerate}

\end{enumerate}
\end{definition}
\begin{theorem}\label{thm.comb}
For each $N$ sufficiently large there exist infinitely many acceptable sets
$\SS(N)$. 
\end{theorem}

This theorem is combinatoric in nature and proved in Section
\ref{sec:GenerationSetsCombinatorics}.
The set of modes $\SSS\subset\ZZ^2$ is a generalization of the set of modes
constructed in \cite{CKSTT}. In \cite{CKSTT} there is only one possible resonant
interaction given by the conditions $|k_1|^2+|k_3|^2=|k_2|^2+|k_4|^2$ and
$k_1+k_3=k_2+k_4$. Geometrically corresponds to four modes forming a rectangle
in $\ZZ^2$. Now, since $d\geq 3$, there are more possibilities of resonant
interactions. Nevertheless, as it is explained in \cite{HausProcesi}, the
interactions more suitable to achieve growth of Sobolev norms are still the ones
which form rectangles. The interaction through more modes not built upon
rectangles seem to be more stable. Therefore we consider analogs of the resonant
interactions constructed in \cite{CKSTT}. Nevertheless, in the case $d\geq 3$,
the rectangles construction presents the following obvious difficulties. On the
one hand, linear combinations of rectangular resonance conditions generate new
unavoidable resonant relations which make the toy model more difficult to
analyze. On the 
other hand, one needs to construct the set $\SSS$ such that all the resonant
relations which are not constructed upon rectangles are avoided. The higher the
degree $d$ the larger amount of such new resonant relations. 

\begin{remark}
In order to obtain the time estimates we also need  a quantitative version of
Theorem \ref{thm.comb}, i.e. a bound on the size of the modes in $\SS$. This is
done in Lemma \ref{parapa} and Corollary \ref{size.and.explosion}.
\end{remark}

\begin{remark}\label{rmk:NoSpreading}
In \cite{GuaKal} an extra condition to the set $\SSS$ is added. This condition,
called by the authors \emph{no spreading condition} says the following. Take
$k\in\ZZ^2 \setminus\SSS$, then there exist at most four rectangles which have
$k$ as a vertex, two vertices in $\SSS$ and the fourth does not belong to
$\SSS$. This implies that in the Hamiltonian $H_\mathrm{Res}$, among the
monomials which depend on $a_k$ there are only four which depend also on two
modes in $\SSS$. This implies that when one considers the full Hamiltonian
\eqref{def:HamAfterNormalForm} one has slow spreading of energy from the modes
of $\SSS$ to the modes not belonging to $\SSS$ since essentially $a_k$ only
``receives energy'' through these four monomials.

This condition is not true in the present setting. Nevertheless we expect that
a slightly weaker  condition holds, replacing four rectangles by a fixed number of rectangles
which depends on the degree $d$ but not on the number of generations $N$.
Unfortunately,  such no spreading condition is considerably more
involved  since
resonant interactions occur for all choices of $\mathcal S$ and classifying them seems a complicated task requiring some new ideas. 
Note that if this {\em weak no spreading condition}   were proved to be true we
would obtain polynomial time estimates as in Theorem 1 of \cite{GuaKal}. 
\end{remark}

The toy model \eqref{def:toymodel:0} is gauge invariant by  condition
4(c) of Definition
\ref{acceptable}. Thus, as explained in Remark \ref{rmk:gauge},
the first term 
$d! n^{d-1}\left(\sum_{i=1}^N|b_i|^2\right)^{d}$, which is  a function of the
mass $J=\sum_{i=1}^N|b_i|^2$, can be eliminated by a change of coordinates which
does not modify the modulus of the components $b_i$'s. Thus, we can consider the
toy model with this term subtracted. Now, we rescale time to have a first order
independent of $n$ ($n$ has been introduced in Definition
\ref{acceptable}). We consider the new time $\tau$ defined by
\begin{equation}\label{def:timerescaling}
 t=\frac{\tau}{n^{d-2}d!d(d-1)}.
\end{equation}
We obtain then, the Hamiltonian
\begin{equation}\label{def:toymodel}
h(b)=\left(\sum_{i=1}^N|b_i|^2\right)^{d-2}\left[-\frac14\sum_{i=1}^N|b_i|^4+
\sum_{i=1}^{N-1}\Re(b_i^2\bar b_{i+1}^2)\right]+\frac{1}{n}\mathcal
P\left(b,\bar b,
\frac1 n\right).
\end{equation}
This toy model is a perturbation from the one obtained in \cite{CKSTT} and
studied also in \cite{GuaKal}. Nevertheless, note that it is a perturbation in
terms of $n^{-1}$. Since we want to study the dynamics of this model for rather
long time, classical perturbative methods do not apply. This implies that we
need to redo and adapt the study done in \cite{GuaKal} for the toy model for the
cubic NLS.

The key point is that  the properties of the toy model obtained in Theorem
\ref{thm.comb} (see Definition \eqref{acceptable}) imply that the
toy model \eqref{def:toymodel} presents the same dynamical features as the toy
model in \cite{CKSTT}. Even if \eqref{def:toymodel} may have very complicated
dynamics, it has certain invariant subspaces where the dynamics is easy to
analyze. Fix mass $J=\sum_{i=1}^N|b_i|^2=1$. Then, by property 4(d) of
Definition
\ref{acceptable}, the toy model  \eqref{def:toymodel} has the periodic orbits
$\TT_j=\{|b_j|=1\,\text{ and }\,b_i=0\,\text{ for all }i\neq j\}$. One can also
consider the invariant subspaces $U_\SSS^j$ where two modes are non zero (see
property 4(e) in Definition \ref{acceptable}). This subspace contains $\TT_j$
and $\TT_{j+1}$. Furthermore, as it is explained in \cite{GuaKal}, in this
subspace the Hamiltonian $h(b)$ becomes a two degrees of freedom Hamiltonian
which is integrable since $h$ itself and the mass $J$ are first integrals in
involution. Then, 
one can see that in $U_\SSS^j$ the unstable manifold of $\TT_{j}$ coincides with
the stable manifold of $\TT_{j+1}$. Thus, we have a sequence of periodic orbits
$\{\TT_j\}_{j=1}^N$ which are connected by heteroclinic orbits. Orbits shadowing
such structure  provide growth of Sobolev norms.

Next theorem shows the existence of orbits with such dynamics.
\begin{theorem}\label{thm:ToyModelOrbit}
Fix large $\ga>1$. Then, for any $N$ large enough and $\de=e^{-\ga N}$ and for
any acceptable set $\SS$ , there exists an orbit $b(\tau)$
of equations \eqref{def:toymodel}, constants $\KKK>0$ and $\nu>0$,
independent of $N$ and $\de$, and $T_0>0$ satisfying
\[
 T_0\leq \KKK N\ln(1/\de),
\]
such that
\[
\begin{split}
 |b_3(0)|&>1-\de^\nu\\
|b_j(0)|&< \de^\nu\qquad\text{ for }j\neq 3
\end{split}
\qquad \text{ and }\qquad
\begin{split}
 |b_{N-2}(T_0)|&>1-\de^\nu\\
|b_j(T_0)|&<\de^\nu \qquad\text{ for }j\neq N-2
\end{split}
\]
Moreover, there exist times $\tau_j\in [0,T_0]$, $j=3,\ldots, N-2$, satisfying
$\tau_{j+1}-\tau_j\leq \KKK\ln (1/\de)$,
such that for any $\tau\in[\tau_j, \tau_{j+1}]$ and  $k\neq j-1,j,j+1$,
\[
 |b_k(\tau)|\leq \de^{\nu}.
\]
\end{theorem}
This theorem is proved in Section \ref{sec:LocalAndGlobalMaps}. 

Note that this theorem can be also stated in terms of the original time $t$, then the time needed to have such evolution is given by
\begin{equation}\label{def:T0prime}
 T_0'\leq\KKK' n^{-(d-2)} N\ln(1/\de)\,,\qquad \KKK'= d!d(d-1)\KKK
\end{equation}

\subsection{Proof of Theorem \ref{thm:MainTheorem}}
Once we have analyzed certain orbits of the toy model, we show that they
are a  good first order for certain orbits of the original partial differential
equation. We use the invariance rescaling \eqref{def:rescaling}. Consider 
\[b^\rrr(t)= \rrr^{-1}b\left( \rrr^{-2(d-1)}n^{d-2}d!d(d-1)t\right)\] 
where
$b(\tau)=b(n^{d-2}d!d(d-1)t)$ is the trajectory given in
Theorem \ref{thm:ToyModelOrbit}. 
Then, the trajectory
\begin{equation}\label{def:SolutionTruncatedSystem}
 \begin{split}
 \mathtt r_k^ \rrr(t)&=b_j^ \rrr(t)  \,\,\,\text{for any }k\in \SSS_j\\
 \mathtt r_k^ \rrr(t)&=0\,\,\,\,\text{for }k\not\in \SSS
 \end{split}
\end{equation}
is a solution of the Hamiltonian $H_\mathrm{Res}$ given in \eqref{Ham2}. Due to
the rescaling, now we study such trajectory in the time range $[0,T]$ with 
\begin{equation}\label{def:Time:Rescaled}
T= \rrr^{2(d-1)}T_0',
\end{equation}
where $T_0'$ is the time introduced in \eqref{def:T0prime}.

We show that for large enough $ \rrr$, \eqref{def:SolutionTruncatedSystem}
is the first order of a true solution of the nonlinear Schr\"odinger equation
\eqref{main}.

\begin{theorem}\label{thm:approximation}
Fix $N\gg1$ and $\rrr_0=e^{C2^{dN}N^2}$ for some large $C$. Let  $\mathtt r_k^ \rrr$ be \eqref{def:SolutionTruncatedSystem}, $T$ be
\eqref{def:Time:Rescaled}. Then, for all $\rrr\geq\rrr_0$ and for any solution $r(t)$  of
\eqref{def:HamAfterNormalForm} with initial condition $r(0)\in\ell^1$ satisfying
$\|r(0)-\mathtt  r^\rrr(0)\|_{\ell^1}\leq \rrr^{-5/2}$, one has that
\[
 \|r(t)-\mathtt  r^\rrr(t)\|_{\ell^1}\leq \rrr^{-3/2}
\]
for $0\leq t\leq T$.
\end{theorem}
This theorem is proven in Section \ref{sec:Approximation}. 

To prove Theorem \ref{thm:MainTheorem}, it only remains to show that a well
chosen trajectory $r(t)$ among those obtained in Theorem
\ref{thm:approximation} undergoes growth of Sobolev norms. The proof of this
fact is done analogously as in \cite{GuaKal}. We reproduce here the reasoning
for completeness.

\begin{proof}[Proof of Theorem \ref{thm:MainTheorem}]
We start by choosing the trajectory which undergoes the growth of Sobolev norms.
We
consider a solution $u(t)$ of \eqref{def:NLS:ode} satisfying $u(0)=\mathtt  r^\rrr(0)$,
where $\mathtt  r^\rrr(t)$ has been defined in \eqref{def:SolutionTruncatedSystem}. 

We define
\[
 \mathfrak S_j=\sum_{k\in\SSS_j}|k|^{2s}\,\,\,\text{ for }j=1,\ldots, N.
\]
We obtain a bound of the final Sobolev norm $\|u(T)\|_{H^s}$ in terms of
$\mathfrak S_{N-2}$ as
\[
\left \|u(T)\right\|^2_{H^s}\geq \sum_{k\in\SSS_{N-2}}|k|^{2s}
\left|u_k(T)\right|^2\geq \mathfrak
S_{N-2}\inf_{k\in\SSS_{N-2}}\left|u_k(T)\right|^2.
\]
Now we obtain a lower bound for $\left|u_k(T)\right|$, $k\in\SSS_{N-2}$. To this
end, we need to show that we can apply Theorem \ref{thm:approximation} to the
solution $u$. Using the change $\Gamma$ obtained in Theorem \ref{thm:NormalForm}
and the change of variables \eqref{def:rotatingcoord}, we can  write $u(t)$ as 
\[
u(t)=\Gamma \left(\left\{r_k(t) e^{\ii|k|^2t}\right\}\right),
\]
where $r(t)$ is a solution of system \eqref{def:VFTruncated}. Note that, since
$u(0)=\mathtt  r^\rrr(0)$, by Theorem \ref{thm:NormalForm},
\[
\begin{split}
 \left\|r(0)-\mathtt  r^\rrr(0)\right\|_{\ell^1}&=\left\|r(0)-u(0)\right\|_{\ell^1}\\
&=\left\|r(0)-\Gamma\left(r\right)(0)\right\|_{\ell^1}\\
&\lesssim \left\|r(0)\right\|^3_{\ell^1}.
\end{split}
\]
We compute the $\ell^1$ norm of $u(0)=\mathtt  r^\rrr(0)$. From the definition of $\mathtt r^
\rrr(0)$ in \eqref{def:SolutionTruncatedSystem} and Theorem
\ref{thm:ToyModelOrbit}, we know that $\|\mathtt  r^\rrr(0)\|_{\ell^\infty}\leq
\rrr^{-1}$. Moreover, $|\mathrm{supp}\ \mathtt  r^\rrr(0)|=|\SSS|=N2^{N-1}$. Thus,
\[
\|u(0)\|_{\ell^1}= \|\mathtt  r^\rrr(0)\|_{\ell^1}\leq  \rrr^{-1} N2^{N-1}.
\]
Theorem \ref{thm:NormalForm} implies that $\Gamma$ is invertible and that
$\Gamma^{-1}$ satisfies $\left\|\Gamma^{-1} (u)-u\right\|_{\ell^1}\leq
\OO\left(\|u\|_{\ell^1}^3\right)$. Therefore, 
\[
\|r(0)\|_{\ell^1}\leq \left\|\Gamma^{-1}(u(0))\right\|_{\ell^1}\lesssim
\|u(0)\|_{\ell^1}\lesssim  \rrr^{-1} N2^{N-1},
\]
which implies, using the definition of $ \rrr_0$ and
taking $N$ large enough,
\[ 
  \left\|r(0)-\mathtt  r^\rrr(0)\right\|_{\ell^1}\lesssim  \rrr^{-3} N^32^{3(N-1)}\leq 
\rrr^{-5/2}.
\]
This estimate implies that $r(0)$ satisfies the hypothesis of Theorem
\ref{thm:approximation}.  We use this fact to estimate the Sobolev norm of
$r(T)$.  Using also  Theorem \ref{thm:NormalForm}, we split $
\left|u_k(T)\right|$ as
\begin{equation}\label{eq:ComputationFinalSobolev}
\begin{split}
 \left|u_k(T)\right|\geq & \left|r_k(T)\right|-\left|\Gamma_k
\left(\left\{r_k(T) e^{\ii|k|^2T}\right\}\right)(T)-r_k(T)
e^{\ii|k|^2T}\right|\\
\geq &\left|\mathtt  r^\rrr_k(T) \right|- \left|r_k(T)-\mathtt  r^\rrr_k(T) \right|\\
&-\left|\Gamma_k \left(\left\{r_k(T) e^{\ii|k|^2T}\right\}\right)(T)-r_k(T)
e^{\ii|k|^2T}\right|.
\end{split}
\end{equation}
We need  a lower bound for the first term of the right hand side and upper
bounds for the second and third ones. Using the definition of $\mathtt  r^\rrr$ in 
\eqref{def:SolutionTruncatedSystem}, the relation between $T$ and $T_0$
established in \eqref{def:Time:Rescaled} and the results in Theorem
\ref{thm:ToyModelOrbit}, we have that for $k\in\SSS_{N-2}$,
\[
\left|\mathtt r_k^ \rrr(T) \right|=    \rrr^{-1}\left| b_{N-1}(T_0)\right| \geq
\frac{3}{4} \rrr^{-1}.
\]
For the second term in the right hand side of
\eqref{eq:ComputationFinalSobolev}, it is enough to use Theorem
\ref{thm:approximation} to obtain,
\[
\left|r_k(T)-\mathtt  r^\rrr_k(T) \right|\leq \left(\sum_{k\in\ZZ^2}\left|r_k(T)-\mathtt r^
\rrr_k(T) \right|\right) \leq \frac{ \rrr^{-1}}{8}.
\]
For the  lower bound of the third term, we use the bound for
$\Gamma-\mathrm{Id}$ given in Theorem \ref{thm:NormalForm}. Then,
\[
\begin{split}
\Big|\Gamma_k &\left(\left\{r_k(T) e^{\ii|k|^2T}\right\}\right)(T)-r_k(T)
e^{\ii|k|^2T}\Big|\\ 
&\leq  \left\|\Gamma_k \left(\left\{r_k(T)
e^{\ii|k|^2T}\right\}\right)(T)-r_k(T) e^{\ii|k|^2T}\right\|_{\ell^1}
\leq  \frac{ \rrr^{-1}}{8}.
\end{split}
\]
Thus, we can conclude that
\begin{equation}\label{eq:GrowthSobolev:Final}
 \left \|u(T)\right\|^2_{H^s}\geq \frac{ \rrr^{-2}}{4}\mathfrak S_{N-2}.
\end{equation}
Now we prove that
\begin{equation}\label{eq:GrowthSobolev:Initial}
 \left \|u(0)\right\|^2_{H^s}\lesssim  \rrr^{-2}\mathfrak S_3.
\end{equation}
Let us recall that $u(0)= \mathtt r^\rrr (0)$ and then  $\mathrm{supp }\ u(0)=\SSS$.
Therefore,
\[
\left\|u(0)\right\|^2_{H^s}=\sum_{k\in\SSS}|k|^{2s}
\left|u_k(0)\right|^2=\sum_{k\in\SSS}|k|^{2s} \left|\mathtt  r^\rrr_k(0)\right|^2.
\]
Then, recalling the definition of $\mathtt  r^\rrr$ in
\eqref{def:SolutionTruncatedSystem} and the results in Theorem
\ref{thm:ToyModelOrbit},
\[
\begin{split}
\sum_{k\in\SSS}|k|^{2s}\left|\mathtt  r^\rrr_k(0)\right|^2&\leq  \rrr^{-2} \mathfrak
S_3+ \rrr^{-2}\de^{2\nu}\sum_{j\neq3}\mathfrak S_j\\
&\leq  \rrr^{-2}\mathfrak S_3\left(1+\de^{2\nu}\sum_{j\neq3}\frac{\mathfrak
S_j}{\mathfrak S_3}\right).
\end{split}
\]
From Theorem \ref{thm.comb} (see Lemma \ref{parapa}) we know that for $j\neq
3$, $\mathfrak S_j/\mathfrak S_3\lesssim e^{sN}$.
Therefore, to bound these terms we use the definition of  $\de$ from Theorem
\ref{thm:ToyModelOrbit} taking $\ga=\wt \ga (s-1)$.
Since $s-1>0$ is fixed, we can choose such $\wt \ga \gg 1$.
Then, we have that
\[
\|u(0)\|_{H^s}=\sum_{k\in\SSS}|k|^{2s}\left|\mathtt  r^\rrr_k(0)\right|^2\sim
\rrr^{-2}\mathfrak S_3.
\]
Using inequalities  \eqref{eq:GrowthSobolev:Final} and
\eqref{eq:GrowthSobolev:Initial},
we have that
\[
\frac{\left\|u(T)\right\|^2_{H^s}}{\left\|u(0)\right\|^2_{H^s}}\gtrsim
\frac{\mathfrak S_{N-2}}{\mathfrak S_3},
\]
and then, applying Theorem \ref{thm.comb}, we obtain
\[
\frac{\|u(T)\|_{H^s}^2}{\|u(0)\|_{H^s}^2}\gtrsim 2^{(s-1)(N-6)}
\ge \left(\frac{\CCC}{\mu}\right)^2.
\]
The last bound is obtained by taking $N$ appropriately large.

Now we have to ensure that $\|u(0)\|_{H^s}\sim \rrr^{-2}\mathfrak S_3\sim\mu$ so
that the final
norm satisfies $\|u(T)\|_{H^s}\gtrsim \CCC$. 
By Corollary \ref{size.and.explosion} we know that the modes $k\in\mathcal S_3$
satisfy $|k|\sim e^{\eta
2^{8dN}N^{8d+1}}$ for some $\eta>0$. Let us also note that if
Definition \ref{acceptable} is satisfied by the set $\SSS\subset\ZZ^2$, it is
also satisfied by the set 
\[
 \SSS'=\{qk:\,k\in\SSS\}
\]
for any given $q\in\N$. Call $u^\SSS$ and $u^{\SSS'}$ the orbits obtained by
reducing into the toy model in the sets $\SSS$ and $\SSS'$ respectively. Then,
$\|u^{\SSS'}(0)\|_{H^s}\sim q^s\|u^{\SSS}(0)\|_{H^s}$. 
Taking
\begin{equation}\label{def:lambda}
 \rrr=e^{\kk 2^{8dN}N^{8d+1}},\,\,\,\kk\gg 1
\end{equation} and adjusting the
parameters $q$ and $\kk$,
one can impose that $
\mu/2\leq\|u^{\SSS'}(0)\|_{H^s}\leq \mu$. 

Finally, it only remains to estimate the diffusion time $T$. We have chosen
$N$ such that $2^{(s-1)(N-6)}
\sim (\CCC/\mu)^2$. Then, using the
definition of $T$ in \eqref{def:Time:Rescaled} and $ \rrr$ in \eqref{def:lambda}
and choosing properly $c$, we obtain
\[
|T|\lesssim   \rrr^{2(d-1)} N^2 \leq e^{(\CCC/\mu)^c}
\]
for some $c>0$.
This completes the proof of Theorem \ref{thm:MainTheorem}.
\end{proof}

\section{Generation sets and
combinatorics}\label{sec:GenerationSetsCombinatorics}

We now discuss the combinatorial part of the paper, namely we prove Theorem
\ref{thm.comb}. As explained in the introduction, we need to choose some
frequency set $\SS\subset\ZZ^2$ which is complete (see Definition
\ref{completeness}) under the flow of \eqref{Ham2} and not action preserving
(namely
a certain number of resonances occur). Moreover, we need enough resonances to be
able to attain the desired energy transfer. Remember that a resonance is a
relation of the form
\begin{equation}\label{reson}
\sum_{\ell=1}^{2d} (-1)^{\ell} k_\ell=0 \qquad \sum_{\ell=1}^{2d} (-1)^{\ell}
|k_\ell|^2=0\ .
\end{equation}
In the case of the cubic NLS (i.e. $d=2$) the only non-trivial resonances are
given by non-degenerate rectangles. For $d>2$ we have many more options in the
choice of the resonant sets. However, as discussed in the introduction, we
are still going to use rectangles as building blocks for the construction of the
set $\SS$ with the same generation structure as in the cubic case, so that every
point of $\SS$ (except for the first and the last generation) belongs to exactly
two rectangles. In the cubic case this means that each mode contributes only to
two resonant monomials. Clearly this is false already in the quintic case, as
one can see as follows. Assume that
$$
k_1-k_2+k_3-k_4=0 \qquad |k_1|^2-|k_2|^2+|k_3|^2-|k_4|^2=0
$$
$$
k_4-k_5+k_6-k_7=0 \qquad |k_4|^2-|k_5|^2+|k_6|^2-|k_7|^2=0
$$
are two rectangles with a common vertex. Then, these two relations give the
resonant sextuple
$$
k_1-k_2+k_3-k_5+k_6-k_7=0 \qquad
|k_1|^2-|k_2|^2+|k_3|^2-|k_5|^2+|k_6|^2-|k_7|^2=0\ .
$$
As the degree of the NLS increases, the combinatorics of the resonances that
appear as a consequence of the rectangle relations becomes more and more
complicated, so we need some formal bookkeeping in order to handle this complex
structure.

It will be convenient to work in the space $\Z^m$ with $m= N 2^{N-1}= |\mathcal
S|$. We denote by $\{\be_j\}_{j=1}^m$ the canonical basis of  $\Z^m$ and divide
the basis elements  in $N$  disjoint {\em abstract generations}  $\mathcal A_i$
(each containing $2^{N-1}$ elements) using the convention that $\be_j\in
\mathcal A_i$ if and only if $(i-1)2^{N-1}+1\leq j \leq i 2^{N-1}$ . Following
\cite{PP}  given $\SS=\{\vv_1,\dots,\vv_m\}\in (\R^{2})^m$ we define the linear
maps
\begin{equation}\label{def:LinearMapsPlacement}
\pi_\SS: \; \Z^m \to \R^2\,,\quad \pi_\SS(\be_i )= \mathtt v_i\,,\quad
\pi^{(2)}_\SS: \; \Z^m \to \R,\quad \pi^{(2)}_\SS(\be_i)= |\vv_i|^2
\end{equation}
so that $\pi(\AA_i)=\SS_i$. By convention we denote $\cup_i \AA_i=\AA$.

\begin{definition}[Abstract Family]\label{def.fam}
An abstract family (of generation number $i\in\{1,\ldots,N-1\}$) is a vector
$$f= \be_{j_1}+\be_{j_2}-\be_{j_3}-\be_{j_4}\,, \quad \text{with}\,\quad
\be_{j_1},\be_{j_2}\in \AA_i\,,\quad\be_{j_3},\be_{j_4}\in \AA_{i+1},$$
and $\; j_1\neq j_2$, $ j_3\neq j_4$

We say that $\be_{j_1},\be_{j_2}$ are the parents of $\be_{j_3},\be_{j_4}$ and
that $\be_{j_3},\be_{j_4}$ are the children of $\be_{j_1},\be_{j_2}$. Moreover,
we say that $\be_{j_1}$ is the spouse of $\be_{j_2}$ (and vice versa) and that
$\be_{j_3}$ is the sibling of $\be_{j_4}$ (and vice versa). 
\end{definition}
\begin{definition}[Genealogical tree]\label{def.fam2}
A set $\mathcal F$ of abstract families is called a {\em genealogical tree}
provided that:
\begin{enumerate}
\item For all $i\in\{1,\ldots,N-1\}$, every $\be_j\in\AA_i$ is a member of one
and only one abstract  family of generation number $i$.
\item For all $i\in\{2,\ldots,N\}$, every $\be_j\in\AA_i$ is a member of one and
only one abstract  family of generation number $i-1$. 
\item For all $i\in\{2,\ldots,N-1\}$ and for all  $\be_j\in\AA_i$ we have that
the {\em sibling} of $\be_j$ and the {\em spouse} of $\be_j$ do not coincide.
\item For all $i\in\{1,\ldots,N\}$ and for all $\be_{j_1},\be_{j_2}\in\AA_i$,
there exists a linear isomorphism
$$
g_{j_1j_2}=g:\ZZ^m\to\ZZ^m
$$
with the following properties:
\begin{enumerate}
\item basis elements are mapped to basis elements, namely for all
$\be_{k_1}\in\AA$ there exists $\be_{k_2}\in\AA$ s.t. $g(\be_{k_1})=\be_{k_2}$;
\item for all $\ell\in\{1,\ldots,N\}$, one has $g(\AA_\ell)=\AA_\ell$;
\item $g(\be_{j_1})=\be_{j_2}$;
\item $g(\FF)=\FF$.
\end{enumerate}
\end{enumerate}
\end{definition}

\begin{definition}
Given $\la=\sum_{j}\la_j \be_j\in \R^m$  we denote by $\Su(\la):=\{j=1,\dots,m:
\la_j\neq 0\}$ its support.
\end{definition}
\begin{remark}\label{no.int}
If $f_1,f_2$ are abstract families of the same generation number, then their
supports have empty intersection.
\end{remark}

\begin{remark}
Item 4 in Definition \ref{def.fam2} is a symmetry property of the genealogical
tree that will be used in order to prove that the {\em intra-generational
equality}
$$
\quad r_{l}=r_{j}:= b_i\;\; \forall i=1,\dots,N\,,\;\forall l,j\in \SS_i
$$
(see the definition of $U_\SS$ in \eqref{christmas}) is preserved by the flow of
the truncated resonant Hamiltonian \eqref{Ham2}. In the case of cubic and
quintic NLS, items 1, 2, 3 of Definition \ref{def.fam2} are enough to ensure
this; however, starting from degree 7, the structure of resonances gets more
complicated and some additional symmetry property is needed.
\end{remark}

\begin{definition}[Generation set] Consider a set $\SS=\{\vv_1,\dots,\vv_m\}\in
(\R^{2})^m$ and the linear maps $\pi_\SS$ and  $\pi^{(2)}_\SS$ defined in
\eqref{def:LinearMapsPlacement}. We say that the set  $\SS $ is an $N$ generation set
if
\begin{equation}\label{famiglie}
\pi_\SS(f)=0\,,\quad \pi^{(2)}_\SS(f)=0\,,\quad \forall f\in \FF\,.
\end{equation}
\end{definition}

We want to use the same genealogical tree as in \cite{CKSTT}. Namely, we
identify
our abstract generations $\AA_i$  with the $\Sigma_i$'s defined in Section 4 of
\cite{CKSTT} and  consider the genealogical tree $\FF$ corresponding to the set
of {\em combinatorial nuclear families connecting generations}
$\Sigma_i,\Sigma_{i+1}$ defined in \cite{CKSTT}.

Let us give a brief overview of the notations of \cite{CKSTT}. Let
$$
S_1=\{1,\ii \}\,,\qquad S_2=\{0,\ii+1 \}\,,
$$
then the $2^{N-1}$ elements of the $k$-th generation are identified with
\begin{equation}
\label{lafame1}
(z_1,\ldots, z_{k-1},z_k,\ldots,z_{N-1})\in S_2^{k-1}\times S_1^{N-k}:=\Sigma_k
\end{equation}
The union of the $\Sigma_k$ is denoted by $\Sigma$.
We order $\Sigma$ by identifying  it  with the ordered set $\AA$ in such a way
that each $\Sigma_i $ is identified with $\AA_i$. 
Now, for all $1\leq k\leq N-1$, a combinatorial nuclear family of generation
number $k$ is a quadruple
\begin{equation}
\label{lafame2}
(z_1,\ldots, z_{k-1},w,z_{k+1},\ldots,z_{N-1})\,\quad w\in S_1\cup S_2
\end{equation}
where all the $z_j $ with $j\neq k$ are fixed with $z_j\in S_2$ if $1\leq
j\leq k-1$ and $z_j\in S_1$ if $k+1\leq j\leq N-1$. The parents correspond to
$w\in S_1$ and the children to $w\in S_2$.

Each combinatorial nuclear family identifies a quadruple in $\AA$ formed by
$\be_{j_1},\be_{j_2}\in\AA_k$ and $\be_{j_3},\be_{j_4}\in\AA_{k+1}$ and hence
an
abstract family according to Definition \ref{def.fam}. This fixes a set $\FF$.
\begin{lemma}\label{gigi no}
The set $\FF$ defined by the combinatorial nuclear families is a genealogical
tree according to Definition \ref{def.fam2}.
\end{lemma}
\begin{proof}
Properties $1,2,3$  follow directly from the definition (see also \cite{CKSTT}).
As for property $4$ we proceed as follows.
Let $\sigma$ be the permutation of the elements of $S_1,S_2$ defined by
$$
\sigma(0)=\ii +1\,,\quad  \sigma(\ii
+1)=0\,,\quad\sigma(1)=\ii\,,\quad\sigma(\ii)=1\,.
$$
For all $\ell=1,\dots, N-1$  we define the map $\mathfrak f_\ell: \Sigma
\to\Sigma $ as
$$
(z_1,\ldots,z_{\ell-1}, z_{\ell},z_{\ell+1},\ldots,z_{N-1})\mapsto (z_1,\ldots,
z_{\ell-1}, \sigma(z_{\ell}),z_{\ell+1},\ldots,z_{N-1}).
$$
All the $\mathfrak f_\ell$ preserve the sets $\Sigma_i$ and the combinatorial
nuclear
families. Moreover they  commute with each other. 
Given any two elements $\be_{j_1},\be_{j_2}\in \AA_i$ we consider the
corresponding two elements $ \mathtt e_{j_1},\mathtt e_{j_2}$ in $\Sigma_i$.
Then there exists a map $\mathfrak g_{j_1,j_2}$, composition of a finite number
of $\mathfrak{f}_\ell$, which maps $ \mathtt e_{j_1}$ to $\mathtt e_{j_2}$. By
construction these maps preserves the $\Sigma_i$ and  the combinatorial nuclear
families. We pull back $\mathfrak g_{j_1j_2}$ to $\AA$ and then extend it to
$\ZZ^m$ by linearity. This is the required map $g_{j_1j_2}$.
\end{proof}

\subsection{Some geometry}\label{sec:SomeGeometry}
Now we want to prove the existence of sets $\SS=\{\vv_1, \ldots \vv_m\}$ which satisfy
all the properties of Definition \ref{acceptable}. We take advantage of the abstract combinatorial setting which we have
defined in the previous section and we use the maps $\pi_\SSS$ and $\pi_\SSS^{(2)}$ given in \eqref{def:LinearMapsPlacement}. It is helpful to think of $\SS$ as a vector in $\R^{2m}$.

Fix any genealogical tree $\mathcal F$ according to Definition \ref{def.fam2}.
For all practical purposes, we can assume that $\mathcal F$ is the one in Lemma
\ref{gigi no}.  We
claim that the resonance relations \eqref{famiglie} define a
real algebraic  manifold $\mathcal M$ as
\begin{equation}\label{mmmm}
\mathcal M:=\left\{\SS\in\R^{2m}\ :\ \forall f\in\mathcal F \quad \pi_\SS(f)=0,\
\pi^{(2)}_\SS(f)=0\right\}\ .
\end{equation}
Indeed  by imposing the linear equations we reduce to a $(N+1)2^{N-1}$ subspace
which we denote by
 \begin{equation}\label{lll}
 \mathcal L:=\left\{\SS\in\R^{2m}\ :\ \forall f\in\mathcal F \quad
\pi_\SS(f)=0\,\right\}\ .
 \end{equation}
  Then by imposing the quadratic constraints we further reduce the dimension. 
We can proceed by induction. Let us suppose that we have enforced all the linear
and quadratic constraints for the first $i$ generations (i.e for all abstract
families $f$ of generation number $\leq i-1$)  and for the first  $h<2^{N-2}$
families of generation number $i$. Then given a {\em parental couple}, (which
for simplicity of notation we denote) $\vv_1,\vv_2$ in the $i$-th generation  we
 have to fix the corresponding children which we denote by $\ww_1,\ww_2$ in the
generation $i+1$.  We have the two equations
$$\ww_2=-\ww_1+\vv_1+\vv_2\,,\quad (\vv_1-\ww_1,\vv_2-\ww_1)=0.$$
so that $\ww_2$ is fixed in terms of $\ww_1$ which in turn lies on the circle
with diameter the segment joining $\vv_1,\vv_2$. Hence provided that $\vv_1\neq
\vv_2$ both children
$\ww_2\neq \ww_1$ are fixed by one angle. Finally (by excluding at most a finite
number of points) we can ensure that $\ww_1,\ww_2$ do not coincide with any of
the previously fixed tangential sites.

  In conclusion  we have  $2\cdot 2^{N-1}$ degrees of freedom from the first
generation and then $2^{N-2}$ angles for each subsequent generation, hence a
manifold  of dimension $(N+3)2^{N-2}$ with singularities all contained  in 
 the proper submanifold  $  \mathcal B:=  \cup_{i\neq j}\{\vv_i-\vv_j=0\}\cap
\mathcal M$.
 Moreover  $\Q^{2m}\cap \mathcal M$ is dense on $\mathcal M$.
Now a resonance as in formula \eqref{reson} defines a codimension 3 algebraic
variety in $\R^{2m}$ as follows.
\begin{definition}
Given $k\in\mathbb N$, we denote by $\mathcal R_k$ the set of vectors
$\lambda\in \Z^m$ with
$$
\sum_i \lambda_i=0\,,\quad \sum_i |\lambda_i|\leq 2k\,.
$$
We say that $\lambda\in\mathcal R_d$ is resonant within $\SS$ if  
$$
\pi_\SS(\lambda)=0\,,\quad \pi^{(2)}_\SS(\lambda)=0\, .
$$
\end{definition}
Note that any resonance within $\SS$ given by equation \eqref{reson} can be
written in this form.  Some resonances cannot be avoided: they are the ones
whose associated algebraic variety contains $\mathcal M$.
 
%

\begin{remark}\label{pipi2}
Since both $\pi_\SS$ and $\pi^{(2)}_\SS$ are linear maps then 
$$
\pi_\SS(\lambda)=0\,,\quad \pi^{(2)}_\SS(\lambda)=0\,,\quad \forall \lambda\in
{\rm Span}(f\in \FF; \Q)\cap\ZZ^{m} .
$$
\end{remark}

All the elements of ${\rm Span}(f\in \FF; \Q)\cap\ZZ^{m}\cap\mathcal R_d$ correspond to
resonances that cannot be avoided, since they are obtained as linear combination
of the relations defining family rectangles.

\begin{definition}
We denote by 
$$
\langle \FF \rangle= {\rm Span}(f\in \FF; \Q)\cap\ZZ^{m}\ .
$$
\end{definition}

\begin{remark}
Note that in general,  given a set $\mathcal G\subset\ZZ^m$, one has ${\rm
Span}(g\in \mathcal G; \ZZ) \subseteq {\rm Span}(g\in \mathcal G; \Q)\cap\ZZ^m$,
but the two need not coincide. However, because of the special structure of
$\FF$, it turns out that $\langle \FF \rangle = {\rm Span}(f\in \FF; \ZZ)$, see
Lemma \ref{gentree} (iii).
\end{remark}

The next lemma gives properties of the unavoidable resonances.
\begin{lemma}\label{gentree}
The following statements hold:
\begin{itemize}
\item[(i)] a genealogical tree $\FF$ is a set of linearly independent abstract
families;
\item[(ii)] all nonzero vectors $\lambda\in \langle \FF \rangle$ have support
$|\Su(\lambda)|\geq 4$ and $|\Su(\lambda)|=4$ if and only if $\lambda$ is a
multiple of an abstract family.
\item[(iii)] we have that $\SF= {\rm Span}(f\in \FF; \ZZ)$;
\item[(iv)] assume that $\lambda\in\SF$ is such that all the elements of
$\Su(\lambda)$ except at most two belong to the same generation: then $\lambda$
is a multiple of an abstract family;
\item[(v)] let $\lambda\in\SF$ and let $v=\sum_{j\in\AA} \la_j\be_j$ be its
decomposition on the basis $\{\be_j\}_j$. Then for all $1\leq i\leq N$ one has 
$$
\sum_{j\in\AA_i} |\lambda_j| \in 2 \N\,,
$$
namely the $\ell^1$-norm of the projection of $\lambda$ on the $i$-th generation
is an even number.
\end{itemize}
\end{lemma}
The proof of this lemma is delayed to Section \ref{sec:ProofLemmaGentree}.

Now we  prove that all the other resonances can be avoided.

\begin{definition}[Non-degeneracy]\label{defi:NonDegeneracy} We say that a
generation set $\SS $ is {\em non-degenerate} if
\begin{itemize}
\item[(i)] For all $\lambda\in  \mathcal R_{2d}\setminus \langle\FF\rangle$ one
has $\pi_\SS(\lambda)\neq 0$.
\item[(ii)] For all $\mu\in \Z^m$ such that $\sum_i \mu_i=1$ and $\sum_i
|\mu_i|\leq 2d-1$ one has that either
$$
K_\SS(\mu):=|\pi_\SS(\mu)|^2-\pi^{(2)}_\SS(\mu)\neq 0
$$
or there exists $1\leq j\leq m$ such that $\mu-\be_j \in \langle \FF\rangle$.
\end{itemize}
\end{definition}
\begin{remark}
Note that the non-degeneracy condition $(ii)$ implies the completeness of $\SS$;
if $d=2$ (i.e. the cubic NLS) actually this is all that is needed (and indeed
only this condition is imposed in \cite{CKSTT}). Condition $(i)$ is a
faithfulness condition (namely it ensures that all $\lambda\in\mathcal
R_{2d}\setminus\langle \FF\rangle$ are not resonant within $\SS$) and could
probably be weakened. 
\end{remark}

The fact that there exist generation sets $\SS$ has been proved in \cite{CKSTT}
together with a weaker non-degeneracy condition in the case $d=2$. In this
section we prove that one can construct non-degenerate generation sets for the
NLS of any degree.

\begin{theorem}\label{thm:FiniteResonantSet}
Consider the manifold $\mathcal M$ introduced in \eqref{mmmm}. Then, there
exists a proper algebraic manifold $\mathcal D\subset \mathcal M$ (of
codimension one in $\mathcal M$) such that  all $\SS\in \mathcal M\setminus
\mathcal D$ are  non-degenerate generation sets.
\end{theorem}
The proof of this Theorem is delayed to Section
\ref{sec:ProofThmFiniteResonantSet}. Now we are ready to prove Theorem
\ref{thm.comb}.

\subsection{Proof of Theorem \ref{thm.comb}}
We need to prove the existence of a set $\SS\subset\ZZ^2$, $|\SS|=m=N2^{N-1}$,
satisfying all the
properties in Definition \ref{acceptable}. It is convenient to consider $\SS$ as
a point in $\ZZ^{2m}$. 

\begin{lemma}\label{lemma:Condition1S}
Consider $\SS$ belonging to $(\MM\setminus \mathcal D)\cap \ZZ^{2m}$, where $\MM$ is the
variety defined in \eqref{mmmm} and $\mathcal D$ is the subvariety given by
Theorem \ref{thm:FiniteResonantSet}. Then,  $\SS$ satisfies  condition 1 in
Definition \ref{acceptable}.
\end{lemma}
\begin{proof}
Condition 1 is equivalent to saying $\vv_i-\vv_j \neq0$ for all $i\neq j$. This
can
be also written as $\pi_\SS(\ee_i-\ee_j)\neq 0$. Item 2 in Lemma \ref{gentree}
implies that $\ee_i-\ee_j\not\in\langle\FF\rangle$. Moreover,
$\ee_i-\ee_j\in\mathcal R_{2d}$. Then, condition 1 of Definition
\ref{acceptable}  follows from item (i) in Definition \ref{defi:NonDegeneracy}.
\end{proof}

\begin{lemma}\label{lemma:Condition3S}
Consider $\SS$ belonging to $(\MM\setminus \mathcal D)\cap \ZZ^{2m}$, where $\MM$ is the
variety defined in \eqref{mmmm} and $\mathcal D$ is the subvariety given by
Theorem \ref{thm:FiniteResonantSet}. Then,  $\SS$ satisfies  condition 3 in
Definition \ref{acceptable}.
\end{lemma}
\begin{proof}
We prove the fulfillment of condition 3 in two steps. First,  we show that the
larger subspace
\[
 V_\SS=\left\{r\in\CC^{\ZZ^2}: r_k=0,\,\forall k\not\in\SS\right\}
\]
is invariant. This fact follows from item (ii) of Definition
\ref{defi:NonDegeneracy}. Indeed, consider a resonance as in \eqref{reson} where
$k=k_1\notin \SS$ and $k_2,\ldots,k_{2d}\in\SS$. Then, by construction there
exists $\mu\in\ZZ^m$, $|\mu|\leq 2d-1$ and $\sum \mu_i=1$ such that 
\[
 k=\sum \mu_i\vv_i,\qquad |k|^2=\sum \mu_i |\vv_i|^2.
\]
Substituting the linear equation in the quadratic one, we obtain $K_\SS(\mu)=0$.
This contradicts item (ii) of Definition \ref{defi:NonDegeneracy}.
Thanks to item (i) of Definition \ref{defi:NonDegeneracy},   the Hamiltonian
$H_{\rm Res}$  defined in \eqref{Ham2}
restricted to $V_\SS$ is
\begin{equation}\label{ristretta}
H_\SS:= \sum_{ \alpha-\beta\in \langle\FF\rangle\cap\RR_d\atop
\alpha_j,\beta_j\geq 0\,,\;|\alpha|_1=|\beta|_1=d
}\binom{d}{\alpha}\binom{d}{\beta}r^\alpha\bar r^\beta\,,
\end{equation} 
where $r^\al=\Pi r_{\vv_i}^{\al_i}$. Indeed  the relations
\[
\pi_\SS(\alpha-\beta)=\sum_{j=1}^m(\alpha_j-\beta_j)\vv_j=0\,,\quad 
\pi^{(2)}_\SS(\alpha-\beta)=\sum_{j=1}^m(\alpha_j-\beta_j)|\vv_j|^2=0
\]
hold if and only if $\alpha-\beta\in \langle\FF\rangle$.

It still remains to show that $U_\SS\subset V_\SS$ defined in \eqref{christmas}
is also invariant.
Fix any $j_1,j_2 \in \AA_i$, we need to prove that 
\begin{equation}\label{nunsescappa}
\partial_{\ol{r}_{j_1}}H_\SS\Big\vert_{U_\SS}=\partial_{\ol{r}_{j_2}}
H_\SS\Big\vert_{U_\SS}.
\end{equation}
Now consider the map $g:=g_{j_1j_2}$ of Definition \ref{def.fam2} item 4. We can
extend this map to monomials (and, by linearity, to polynomials) by setting
$$
\forall \al,\beta\in \N^{m}\,,\quad   g (r^\alpha\bar r^\beta) = r^{g(\al)} \bar
r^{g(\beta)}.
$$
where we recall that $$g(\alpha):=\sum_j \alpha_{j}g(\be_j)=\sum_j
\alpha_{g^{-1}(j)}\be_{j}$$
(same for $\beta$). In particular, we have $(g(\alpha))_{j_2}=\alpha_{j_1}$ and
$(g(\beta))_{j_2}=\beta_{j_1}$. It is also easy to see that for all $\ell$,
setting
$$
\mathfrak a_\ell:= \sum_{k\in \AA_\ell}\al_k\,,\qquad \mathfrak b_\ell:=
\sum_{k\in \AA_\ell}\beta_k\,,
$$
one has $g(\mathfrak a_\ell)= \mathfrak a_\ell$, same for $\mathfrak b_\ell$.

For each monomial $\mathfrak m:=\mathfrak m_{\al,\beta}= r^\alpha\bar r^\beta $
one has  
\begin{equation}\label{monomig}
\partial_{\ol{r}_{j_1}}\mathfrak m \big\vert_{U_\SS}=\frac{\beta_{j_1}}{\bar
b_i}\prod_{\ell=1}^N b_\ell^{\mathfrak a_\ell}\bar b_\ell^{\mathfrak b_\ell}=
\frac{(g(\beta))_{j_2}}{\bar b_i}\prod_{\ell=1}^N b_\ell^{g(\mathfrak
a_\ell)}\bar b_\ell^{g(\mathfrak b_\ell)}=
\partial_{\ol{r}_{j_2}}g(\mathfrak m)\big\vert_{U_\SS}\ .
\end{equation}
Note that $\beta_{j_1}\neq0$ implies $\mathfrak b_\ell>0$ so that all the
expressions in \eqref{monomig} are monomials.

Moreover, $g$ preserves the Hamiltonian $H_\SS$ i.e.
\begin{equation}\label{hgh}
g(H_\SS)=\sum_{ \alpha-\beta\in \langle\FF\rangle\cap\RR_d\atop
\alpha_j,\beta_j\geq 0\,,\;|\alpha|_1=|\beta|_1=d
}\binom{d}{\alpha}\binom{d}{\beta}r^{g(\alpha)}\bar r^{g(\beta)}=\sum_{
 \alpha-\beta\in \langle\FF\rangle\cap\RR_d\atop \alpha_j,\beta_j\geq
0\,,\;|\alpha|_1=|\beta|_1=d
}\binom{d}{\alpha}\binom{d}{\beta}r^\alpha\bar r^\beta=H_\SS
\end{equation}
since
$$
g(\alpha)-g(\beta)\in \langle\FF\rangle\cap\RR_d \Longleftrightarrow
\alpha-\beta\in \langle\FF\rangle\cap\RR_d\ ,
$$
$$
\binom{d}{\alpha}=\binom{d}{g(\alpha)}\ , \qquad
\binom{d}{\beta}=\binom{d}{g(\beta)}\ .
$$
Finally, we use \eqref{monomig} and \eqref{hgh} in order to prove
\eqref{nunsescappa}:
\begin{equation*}
\partial_{\ol{r}_{j_1}}H_\SS\Big\vert_{U_\SS}=\partial_{\ol{r}_{j_2}}
g(H_\SS)\Big\vert_{U_\SS}=\partial_{\ol{r}_{j_2}}H_\SS\Big\vert_{U_\SS}\ .
\end{equation*}
\end{proof}

 In order to  prove condition 4 we first analyze the Hamiltonian $H_{\rm Res}$.
To this end, we define
\begin{equation}\label{def:Lj}
L^{(2j)}=\|r\|_{\ell^{2j}}^{2j}= \sum_{k\in
\Z^2}|r_k|^{2j}.
\end{equation} Note that $L^{(2)}$ is the conserved quantity
$\|r\|_{L^2}^2$.
We have the following lemma.

 \begin{lemma}
 The Hamiltonian $H_{\rm Res}$ has the following form
 \begin{equation}\label{parto}
 H_\mathrm{Res}- d!(L^{(2)})^d= d!d(d-1)(L^{(2)})^{d-2} \Big[
-\frac{1}{4}\sum_{k\in \Z^2}|r_k|^4 + \sum_{k_1,k_2,k_3,k_4\in \Z^2 \atop
{k_1\neq k_3,k_4\,,\; k_1+k_2=k_3+k_4 \atop{|k_1|^2+|k_2|^2=|k_3|^2+|k_4|^2} }
}r_{k_1}r_{k_2}\bar r_{k_r}\bar r_{k_4}\Big] +R
 \end{equation}
 where the term in square brackets is the {\em cubic} NLS while $R$ contains
only terms of the following types:
 \begin{itemize}
 \item Action preserving terms (see Definition \ref{integra}) of the form  
$|r|^{2\alpha}$  with $\alpha!:=\prod_{k\in \Z^2}\alpha_k!>2$.
  \item  Non-action preserving  terms whose degree in the actions is less than
$d-2$.
 \item Non-action preserving terms $r^\alpha\bar r^\beta$ of degree $d-2$ in the
actions and such that $\alpha!\beta!>1$.
 \end{itemize} 
 \end{lemma} 
 \begin{proof}

First we note that a resonance is action preserving if (up to permutations of
the even and the odd indexes between themselves) one has
$$
\{k_1,k_1,k_2,k_2,\dots,k_d,k_d\}.
$$
We can evidence an {\em integrable part} of the Hamiltonian $H_\mathrm{Res}$ as
$$H_{\rm{Int}}:=H_{\rm{Int}}(|r_k|^2)=\hskip-10pt\sum_{\alpha\in (\N)^{\Z^2}:
\atop |\alpha|=d} {\binom{d}{\alpha}}^2 |r|^{2\alpha}\,, $$ which contains all
the terms in \eqref{Ham2}  with $\alpha=\beta$. Note that $H_{\rm{Int}}$ is a
symmetric function of the actions
$\{|r_k|^2\}_{k\in \Z^2}$.
It is well known that the functions $L^{(2j)}$ defined in \eqref{def:Lj}
generate the symmetric
polynomials in the actions. Hence we can express the integrable Hamiltonian as a
polynomial in the $L^{(2j)}$. Since $L^{(2)}$ is a constant of motion (and we
will perform a symmetry reduction with respect to it) it  will be convenient to
evidence the  two terms of highest degree in $L^{(2)}$.
$$
(L^{(2)})^m= \sum_{ |\alpha|=m}\binom{m}{\alpha}|r|^{2\alpha} = m! \sum_{
|\alpha|=m\atop \alpha!=1}|r|^{2\alpha} + \sum_{ |\alpha|=m\atop
\alpha!>1}\binom{m}{\alpha}|r|^{2\alpha}.
$$
By direct computation,  
\[
\begin{split}
  H_{\text {Int}}-  d! (L^{(2)})^d &= \sum_{ \alpha!> 1} {\binom{d}{\alpha}}^2
|r|^{2\alpha}(1-\alpha!)\\
&=  -\frac{(d!)^2}4 \sum_{|\alpha|= d,\atop \alpha!=2}  |r|^{2\alpha} + \sum_{
\alpha!> 2} {\binom{d}{\alpha}}^2 |r|^{2\alpha}(1-\alpha!).
\end{split}
\]
Note that
\[
 \sum_{  |\alpha|= d,\atop \alpha!=2}  |r|^{2\alpha} = \sum_ {k\in \Z^2} |r_
k|^4\sum_{|\beta|= d-2\atop
 (\beta+2\ee_ k)!=2} |r|^{2\beta}\,,
\]
where $\ee_k\in (\N)^{\Z^2}$ is the $k$'th basis vector. We compare the above
expression with
\[
\begin{split}
(L^{(2)})^{d-2}L^{(4)}&= \sum_ k\sum_{ |\beta|=d-2}\binom{d-2}{\beta}|r|^{2\beta
+4 \ee_ k}\\
&= (d-2)! \sum_{ |\beta|=d-2\atop(\beta+2\ee_ k)!=2}|r|^{2\beta+4 \ee_ k} +
\sum_{ |\beta|=d-2\atop (\beta+2\ee_ k)!>2}\binom{d-2}{\beta}|r|
^{2\beta+4 \ee_ k}.
\end{split}
\]
Thus,
\begin{equation}\label{porro}
H_{\text {Int}}-  d! (L^{(2)})^d= -\frac{d! d
(d-1)}{4}L^{(4)}\left(L^{(2)}\right)^{d-2} + \sum_{ \alpha!> 2} c_\alpha
|r|^{2\alpha}
\end{equation}
We could continue our computation. However we only need that $\sum_{ \alpha!> 2}
c_\alpha |r|^{2\alpha}$, as polynomial in the $L^{(i)}$, is 
 of degree at most $d-3$ in $L^{(2)}$.

We can perform a similar procedure for the non-integrable part of the
Hamiltonian 
$$H_\mathrm{Res}-H_{\rm Int}= \sum_{\alpha\neq \beta\in (\N)^{\Z^2}:
|\alpha|=|\beta|=d\atop {\sum_k (\alpha_k-\beta_k)k=0\,,\;\sum_k
(\alpha_k-\beta_k)|k|^2=0}}
\hskip-10pt\binom{d}{\alpha}\binom{d}{\beta}r^\alpha\bar r^\beta. $$
We first evidence the terms of higher degree in the action variables which are
clearly:
\[
\sum_{k_1,k_2,k_3,k_4\in \Z^2 \atop {k_1\neq k_3,k_4\,,\; k_1+k_2=k_3+k_4
\atop{|k_1|^2+|k_2|^2=|k_3|^2+|k_4|^2} } }
\sum_{ |\alpha|=
d-2}\binom{d}{\alpha+\ee_{k_1}+\ee_{k_2}}\binom{d}{\alpha+\ee_{k_3}+\ee_{k_4}}
|r|^{2\alpha} r_{k_1}r_{k_2}\bar r_{k_3}\bar r_{k_4}=
\]
$$
=\sum_{k_1,k_2,k_3,k_4\in \Z^2 \atop {k_1\neq k_3,k_4\,,\; k_1+k_2=k_3+k_4
\atop{|k_1|^2+|k_2|^2=|k_3|^2+|k_4|^2} } }
\Bigg[(d!)^2\sum_{ |\alpha|= d-2\atop
{(\alpha+\ee_{k_1}+\ee_{k_2})!=1\atop{
(\alpha+\ee_{k_3}+\ee_{k_4})!=1}}}|r|^{2\alpha} r_{k_1}r_{k_2}\bar r_{k_3}\bar
r_{k_4}+
$$
$$
 +\sum_{ |\alpha|= d-2\atop
(\alpha+\ee_{k_1}+\ee_{k_2})!(\alpha+\ee_{k_3}+\ee_{k_4})!>1}\binom{d}{
\alpha+\ee_{k_1}+\ee_{k_2}}\binom{d}{\alpha+\ee_{k_3}+\ee_{k_4}}|r|^{2\alpha}
r_{k_1}r_{k_2}\bar r_{k_3}\bar r_{k_4}\Bigg] 
$$
We proceed as for the integrable terms evidencing the highest order term in
$L^{(2)}$, we have 
$$
(L^{(2)})^{d-2}\!\!\!\!\!\!\!\sum_{k_1,k_2,k_3,k_4\in \Z^2 \atop {k_1\neq
k_3,k_4\,,\; k_1+k_2=k_3+k_4 \atop{|k_1|^2+|k_2|^2=|k_3|^2+|k_4|^2} }
}\!\!\!\!\!r_{k_1}r_{k_2}\bar r_{k_3}\bar r_{k_4}=
$$
$$
\sum_{k_1,k_2,k_3,k_4\in \Z^2 \atop {k_1\neq k_3,k_4\,,\; k_1+k_2=k_3+k_4
\atop{|k_1|^2+|k_2|^2=|k_3|^2+|k_4|^2} } }\Bigg[ (d-2)!\sum_{ |\alpha|= d-2\atop
\alpha!=1}  |r|^{2\alpha}r_{k_1}r_{k_2}\bar r_{k_3}\bar r_{k_4}
+\sum_{
|\alpha|= d-2\atop \alpha!>1} \binom{d-2}{\alpha}
|r|^{2\alpha}r_{k_1}r_{k_2}\bar r_{k_3}\bar r_{k_4}\Bigg]
$$
We have proved our thesis, in formul{\ae}  the remainder $R$ is given by
$$
R= \!\!\!\!\!\!\sum_{|\alpha|=|\beta|=d\,,\;
|\alpha-\beta|>4\atop{\sum_k (\alpha_k-\beta_k)k=0\,,\;\sum_k
(\alpha_k-\beta_k)|k|^2=0} }\!\!\!\!\!\!R_{\alpha,\beta}r^\alpha\bar r^\beta+
\!\!\!\!\!\!\sum_{|\alpha|=|\beta|=d\,,\;
|\alpha-\beta|=4\,,\alpha!\beta!>1\atop{\sum_k (\alpha_k-\beta_k)k=0\,,\;\sum_k
(\alpha_k-\beta_k)|k|^2=0} }\!\!\!\!\!\!\!\!\!R_{\alpha,\beta}r^\alpha\bar
r^\beta+
\sum_{|\alpha|=d\,,\;\alpha!>2}R_{\alpha}|r|^{2\alpha}
$$
\end{proof}
\begin{lemma}\label{lemma:Condition4S}
Consider $\SS$ belonging to $(\MM\setminus \mathcal D)\cap \ZZ^{2m}$, where $\MM$ is the
variety defined in \eqref{mmmm} and $\mathcal D$ is the subvariety given by
Theorem \ref{thm:FiniteResonantSet}. Then,  $\SS$ satisfies  condition 4 in
Definition \ref{acceptable}.
\end{lemma}
\begin{proof}
We consider the Hamiltonian $H_{\rm Res}$ of formula \eqref{parto} restricted to
the subspace $U_\SS$. The new Hamiltonian $h_\SS$ is defined as $h_\SS=n^{-1}
H_{\rm Res}\vert_{U_\SS}$  where  $n=2^{N-1}$.  Note that the factor $n^{-1}$ is
not a time rescaling. It needs to be added in order to obtain the symplectic
form $\frac{\ii}{2}\sum_j db_j\wedge d\bar b_j$. Note that $h_\SS$ is
homogeneous of degree $2d$ in $(b,\bar b)$.

One can analyze explicitly the toy-model Hamiltonian $\Ht$: 
\begin{equation}\label{toy}
\Ht(b)= \frac{1}{n}\sum_{\mathfrak a,\mathfrak b\in \N_0^N \atop \sum_i
\mathfrak a_i=\sum_i\mathfrak b_i=d} C_{\mathfrak a,\mathfrak b} b^\mathfrak a
\bar b^\mathfrak b\,,\quad C_{\mathfrak a,\mathfrak b}:=\sum_{\alpha,\beta\in
\N_0^m:  \; \alpha-\beta\in \langle\FF\rangle\atop {\sum_{j\in \mathcal A_i
}\alpha_j= \mathfrak a_i\,\atop
\sum_{j\in \mathcal A_i}
\beta_j=\mathfrak b_i }}\binom{d}{\alpha}\binom{d}{\beta}\,,
\end{equation}
where  by an abuse of notation with $j\in \mathcal A_i$ we mean  $\ee_j\in
\mathcal A_i$ and hence $(i-1)n+1\leq j \leq i n$.
The important fact is that $\Ht$ is a polynomial in $n=2^{N-1}$ (the number of
elements in each generation) and, since $n$ is very large, we only need to
compute the leading orders.
The degree of $\Ht$ in $n$ is at most $d-1$ (the coefficients $C_{\mathfrak
a,\mathfrak b}$'s have degree at most $d$). The terms that we will need to
compute explicitly are the coefficients of $n^{d-1}$ and of $n^{d-2}$ in $\Ht$
(which amounts to computing the coefficients of $n^d$ and of $n^{d-1}$ in
$C_{\mathfrak a,\mathfrak b}$). The crucial remark, informally stated, is that
the degree in $n$ is lower for terms whose combinatorics imposes more
constraints; this happens by two mechanisms:
\begin{itemize}
\item by decreasing $\Su(\alpha)$ i.e. the cardinality (which is at most $d$) of
the $\{\alpha_i\neq0\}$ (or, symmetrically, $\Su(\beta)$);
\item by increasing $\alpha-\beta\in\langle\FF\rangle$, indeed if we know that 
the indexes $k_i$ satisfy some  family relations then by fixing the family we
fix four of the indexes.
\end{itemize}
We know that $H_{\mathrm{Res}}$ Hamiltonian has the expression \eqref{parto} and
moreover it is easy to see that all the terms in $R$ contribute at most
$n^{d-3}$. Then we have
\begin{eqnarray*}
h_\SS(b)= & & d! n^{d-1}\left(\sum_{i=1}^N|b_i|^2\right)^{d}+\\
&+&
n^{d-2}d!d(d-1)\left(\sum_{i=1}^N|b_i|^2\right)^{d-2}\left[-\frac14\sum_{i=1}
^N|b_i|^4+ \sum_{i=1}^{N-1}\Re(b_i^2\bar b_{i+1}^2)\right]+\OO(n^{d-3}).
\end{eqnarray*}
We still have to analyze the terms contained in $\OO(n^{d-3})$ in order to check
that the polynomial $\mathcal P$ of Definition \ref{acceptable} satisfies the
properties 4(a)-4(f). Properties 4(a), 4(b), 4{(}c) are completely
straightforward, while 4(d) follows directly from Lemma \ref{gentree}, item (v).

As for property 4(e), the fact that $U_\SS^{j}$ is invariant follows from 4(d).
Note that the only elements of $\SF$ entirely supported on the generations
$\AA_j,\AA_{j+1}$ are of the form $\sum_k\la_k f_k$ with $\la_k\in\ZZ$, where
the $f_k$'s are the vectors representing the $2^{N-2}$ families of generation
number $j$ (see Definition \ref{def.fam}). Note that the $f_k$'s have disjoint
support. Then, using \eqref{ristretta}, it is immediate to see that the
expression of the Hamiltonian as a function of $(b_j,\bar b_j)$ and
$(b_{j+1},\bar b_{j+1})$ relies on a purely combinatorial computation,
independent of $j$ (up to an index translation). This combinatorial structure is
left invariant if one exchanges parents with children in all the families $f_k$:
this gives the symmetry with respect to the exchange $j\longleftrightarrow j+1$.

To conclude, we prove 4(f). Given $i\neq j$, we consider monomials in $h_\SS$
which depend only on $(b_i,\bar b_i),(b_j,\bar b_j)$ and are exactly of degree
two in $(b_i,\bar b_i)$. Monomials of this form can only come from monomials in
\eqref{ristretta} such that $\alpha-\beta\in\SF$ is supported entirely on the
$j$-th generation except for at most two elements. Therefore we can apply Lemma
\ref{gentree}, (iv) and deduce that $\alpha-\beta$ is either zero or (up to the
sign) an abstract family. The case $\alpha-\beta=0$ corresponds to action
preserving monomials of \eqref{ristretta} and produces monomials of the form
$\chi_{ij}|b_j|^{2d-2}|b_i|^2$ (for some suitable coefficient $\chi_{ij}$),
while the case $\alpha-\beta=\pm f$ with $f\in\FF$ is possible only if $|i-j|=1$
and produces terms of the form $\rho_{ij}|b_j|^{2d-4}\Re(b_i^2\bar b_j^2)$ (for
some suitable coefficient $\rho_{ij}$). The fact that $H_{\rm Int}$ is a
symmetric polynomial in the variables $|r_k|^2$ implies that
$\chi_{ij}\equiv\chi$ is 
independent of $i$ and $j$. Finally, the fact that $\rho_{i,i+1}\equiv\rho$ is
independent of $i$ follows from 4(e).
\end{proof}

In conclusion, any $\SS\in(\MM\setminus \mathcal D)\cap \ZZ^{2m}$ satisfies conditions
1,3,4 of Definition \ref{acceptable}. The fact that $(\MM\setminus \mathcal
D)\cap \ZZ^{2m}$ is non-empty follows from the density of $\MM\cap\Q^{2m}$ on
$\MM$ and from the fact that $\MM$ and $\mathcal D$ are homogeneous (if $v$
belongs to the manifold, $tv$ also belongs to the manifold for all $t\in\RR$).
 The existence of sets $\SS$ satisfying
also item 2 follows the same reasoning as in \cite{CKSTT}. In order to give
quantitative estimates for the norm of the points in $\SS$, we denote by 
$\{\mathtt j_1,\dots,\mathtt j_m\}$ the prototype embedding obtained by mapping 
 each $\mathtt e_i=(z_1,\ldots, z_{k-1},z_k,\ldots,z_{N-1})\in\Sigma$ (notation
as in \eqref{lafame1}) to $\mathtt j_i\in\ZZ^2$ via
$$\mathtt e_i\mapsto \mathtt j_i= ({\rm Re} \prod_i z_i, {\rm Im} \prod_i
z_i)\in \ZZ^2\qquad\qquad i=1,\ldots,m$$  (note that this in this list the
vectors $\mathtt j_i$ are NOT distinct but have high multiplicity). 
\begin{lemma}\label{parapa}
There exists $R<(N 2^N)^{16dN(N 2^N)^{8d}}$, such that one may choose  a
non-degenerate generation set $$\SS=\{\mathtt v_1,\ldots,\mathtt v_m\}\in
(\MM\setminus \mathcal D)\cap \ZZ^{2m}$$ satisfying
\begin{equation}\label{proximity}
|\mathtt v_i- R \mathtt j_i| \leq 3^{-N} R \qquad\qquad \forall i=1,\ldots,m\ .
\end{equation}
\end{lemma}

\begin{remark}\label{quellola}
For $N\gg1$, the condition \eqref{proximity} implies that the norm explosion property \eqref{Normexp} is satisfied.
\end{remark}

\begin{proof}
We preliminarily notice that the resonance relations are a set of at most
quadratic  equations in the variables $\mathtt v_i$'s. Thus, if we assume to
have fixed (with the inductive procedure described in Section
\ref{sec:SomeGeometry}) the first $k$ variables so that the non-degeneracy
conditions given by Definition \ref{defi:NonDegeneracy}, then in adding the
$k+1$-th variable, in order to enforce the non-degeneracy conditions, we must
verify that it does not satisfy $K\leq k^{7d}$ at most quadratic relations.
Moreover by definition $\mathcal M$ is a homogeneous manifold, namely if $v\in
\MM$ then $t v\in \MM$ for all $t\in\mathbb R$. Finally we notice that also the
resonance relations are homogeneous, hence if $v\in\MM\setminus \mathcal D$ then
also $t v\in\MM\setminus \mathcal D$.

We start by considering a neighborhood of radius $10^{-N}$ of each $\mathtt j_{i}$ with $i\in \mathcal A_1$.  
Then rescaling by $R_1=10^{2dN}$ we can ensure that in each neighborhood there are more than $2^{8d N}$ integer points so that we can surely choose in these neighborhoods
integer points $\mathtt w^{(1)}_i$  such that
 $$ |\mathtt w^{(1)}_i- R_1 \mathtt j_i| \leq 10^{-N} R_1\quad \forall i\in \mathcal A_1$$
and  the $\mathtt w^{(1)}_i$ satisfy the non-degeneracy conditions.

We proceed by induction.  At each generation $j\geq 2$ we have  $\mathtt w^{(j-1)}_i\in\ZZ^2$ with $i\in \cup_{h=1}^{j-1}\mathcal A_h$ so that
 \begin{enumerate}
\item[$1_{j-1}$]  $ |\mathtt w^{(j-1)}_i- R_{j-1} \mathtt j_i| \leq 3^{j-2} \cdot 10^{-N} R_{j-1}$,
\item[$2_{j-1}$] $\{\mathtt w^{(j-1)}_i\}_{i\in \cup_{h=1}^{j-1}\mathcal A_h}$ is a non-degenerate generation set with $j-1$ generations.
\end{enumerate} 
Then we claim that we can choose $\mathtt w^{(j-1)}_i\in\mathbb{Q}^2$ for $i\in
\mathcal A_j$ so that
 \begin{enumerate}
\item[(i)]  $ |\mathtt w^{(j-1)}_i- R_{j-1} \mathtt j_i| \leq 3^{j-1} \cdot 10^{-N} R_{j-1}$,
\item [(ii)] setting $K=(N 2^N)^{16d(N 2^N)^{8d}}$, we have that  $ K \mathtt
w^{(j-1)}_i\in \ZZ^2$. 
\item[(iii)] $\{\mathtt w^{(j-1)}_i\}_{i\in \cup_{h=1}^{j}\mathcal A_h}$ is a non-degenerate generation set with $j$ generations.
\end{enumerate} 
If our claim holds true, we set  $R_j=KR_{j-1}$ and  $\mathtt w^{(j)}_i= K \mathtt w^{(j-1)}_i$ for $i\in \cup_{h=1}^{j}\mathcal A_h$. By construction  items $1_j$ and $2_j$ hold .
We conclude our proof by fixing $R= R_N$ and $\mathtt v_i= \mathtt w^{(N)}_i$ for all $i\in \mathcal A$.

It remains to prove our claim.  To pass from a $j-1$ generation set to one with
$j$  generations we have to use the family relations and for each couple of
parents produce the corresponding two children. Let us fix two parents $\mathtt
w^{(j-1)}_{i_1}\rightsquigarrow p_1$ and $\mathtt
w^{(j-1)}_{i_2}\rightsquigarrow p_2$. 
This means fixing two opposite points on the circle $(v- p_1,v-p_2)=0$. By
construction, if we choose as children $c_1,c_2$  the two opposite points such
that $c_1-c_2$ is  orthogonal to $p_1-p_2$ then  
$$  |c_k- R_{j-1} \mathtt j_{\ell_k}| \leq 2 \cdot 3^{j-1} \cdot 10^{-N} R_{j-1}\qquad k=1,2$$
(where $\mathtt j_{\ell_1},\mathtt j_{\ell_2}$ are the two corresponding children in the prototype embedding), however we cannot guarantee that these two points satisfy item (iii). We can write the rational points on the circle as
$$ P_t:= p_1 - \frac{(p_1-p_2,t) t}{|t|^2}\,,\quad t=(m_1,m_2)\in\ZZ^2\ .$$ 
Noting that
$$
c_k= \frac{p_1+p_2}{2}\pm O \frac{p_1-p_2}{2} \,,\quad  O= \begin{pmatrix}0&1\\-1 &0 \end{pmatrix}
$$
we can compute the $\tau$ corresponding to one of the two children, say
$$
P_\tau:=c_1= \frac{p_1+p_2}{2}+ O \frac{p_1-p_2}{2}
$$
and we get
$$
{(p_1-p_2,\tau) }= {( p_1-p_2, O \tau)}  \longrightarrow  \tau:=  (O- \mathbb I)(p_1-p_2)
$$
(note that in this way the $x,y$ coordinates of $\tau$ are NOT coprime!).
We now consider the points $P_{\tau_k}$ for $\pm k=D,\dots, 2D$ defined by
$$\tau_k= (k\mathbb I+O)\, \tau$$
so that as $k$ varies, $\tau_k$ identifies different points on the circle.
By direct computation one has that
$$
P_{\tau_k}=p_1+\frac{k+1}{2(k^2+1)}\tau_k
$$
and
$$
{\rm dist}(P_{\tau_k},P_\tau) =\frac{1}{\sqrt{k^2+1}}|p_1-p_2|
$$
If we fix
$ D > 20^N $ we are sure that each point $P_{\tau_k}$ satisfies item (i). Now
each  non-degeneracy condition removes at most two points on the circle and we
have at most $(N 2^N)^{7d}$ conditions, hence we can ensure the existence of
non-degenerate $P_{\tau_k}$ by fixing $D:= (N 2^N)^{8d}>20^N$. Finally since
$K \geq (D!)^2 \gtrsim {\rm lcm}(k^2+1)_{k=D}^{2D}$, where $\rm
lcm$ denotes the least common multiple, item (ii) is also satsified. Therefore,
the thesis follows by noting that $R_N=R_1K^{N-1}<K^N$ and that
$3^{N-1}\cdot10^{-N}<3^{-N}$.
\smallskip

\begin{figure}[ht]\centering
\begin{minipage}[t]{12cm}
{\centering
{\psfrag{c}{$ \mathtt j_3$}
\psfrag{b}{$\mathtt j_1$}
\psfrag{d}{$\mathtt j_2$}
\psfrag{a}{$\mathtt j_4$}
\psfrag{e}{$p_1$}
\psfrag{f}{$p_2$}
\psfrag{g}{$c_1$}
\psfrag{h}[c]{$c_2$}
\psfrag{i}{$\tau$}
\psfrag{l}{$\tau_{k_1}$}
\psfrag{m}{$\tau_{k_2}$}
\psfrag{n}{$\tau_{k_3}$}
\includegraphics[width=8cm]{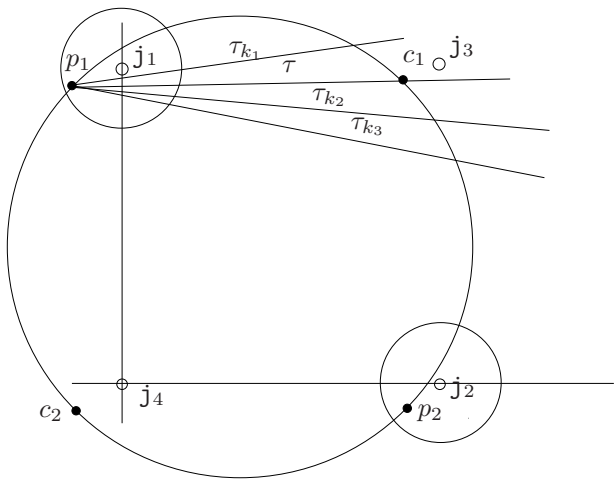}
}}
\end{minipage}
\caption{Our procedure for finding rational points with bounded denominators, here we wish to place a point in the second generation {\em close} to $\mathtt j_3$.}
\end{figure}
%


\end{proof}

\begin{corollary}\label{size.and.explosion}
For $N\gg1$ there exists an acceptable (see Definition \ref{acceptable}) generation set $\SS=\SS(N)$ such that
\begin{equation}\label{size.estimate}
|\mathtt v|<(N 2^N)^{16dN(N 2^N)^{8d}+1} \qquad\qquad \forall \mathtt v\in\SS\ .
\end{equation}
\end{corollary}

\begin{proof}
It follows directly from Lemma \ref{parapa} and Remark \ref{quellola}.
\end{proof}

\subsection{Proof of Lemma \ref{gentree}}\label{sec:ProofLemmaGentree}
We first prove (i), (ii). Assume by contradiction that the abstract families of
$\FF$ are linearly dependent. This means that there exist some
$\alpha_k\in\Q,f_k\in\FF$ such that
$$\mathscr L:=\sum_k\alpha_k f_k=0\ ,\quad \exists  k_0 \,: \alpha_{ k_0}\neq
0\,. $$
Now, let $i_<$ and $i_>$ be respectively the minimal and the maximal generation
numbers of the families appearing in the linear combination $\mathscr L$ (with
nonzero coefficient). It follows from Remark \ref{no.int} that the support of
$\mathscr L$ contains at least two elements of the generation $i_<$ and two
elements of the generation $i_>+1$, which implies $\mathscr L\neq0$, which is
absurd.
Now, since the abstract families of $\FF$ are linearly independent, they form a
basis of $\langle\FF\rangle$. Therefore each $\lambda\in\langle\FF\rangle$ can
be written in a unique way as a linear combination
\begin{equation}\label{lin.comb}
\lambda=\sum_k\alpha_k f_k\ .
\end{equation}
Then, $|\Su(\lambda)|\geq4$ since as above it contains at least two elements in
$\AA_{i_<}$ and two elements in $\AA_{i_>+1}$.

Now, suppose $|\Su(\lambda)|=4$. This means that we have exactly two elements in
the generation $i_<$ and two in the generation $i_>+1$, and no elements in the
possible intermediate generations. We claim that $i_<=i_>$. In order to prove
our claim, we first notice that, if for some $i$ the linear combination
\eqref{lin.comb} contains $h$ different families (namely, there are $h$
different families with nonzero coefficient $\alpha_k$) of generation number $i$
and $k\neq h$ different families of generation number $i+1$, then $\Su(\lambda)$
contains at least 2 elements of the generation $i+1$. Then we notice that, if
for some $i$ the expression \eqref{lin.comb} contains exactly one family of
generation number $i$ and exactly one family of generation number $i+1$, then
$\Su(\lambda)$ contains at least 2 elements of the generation $i+1$ (since
sibling and spouse cannot coincide). It follows that, since $\Su(\lambda)$ does
not contain elements from the intermediate generations, there cannot be
intermediate 
generations, i.e. $i_<=i_>=\bar \imath$. 
Finally, thanks to Remark \ref{no.int}, in order to have 
only two elements of generation number $\bar \imath$ and two elements of
generation number $\bar \imath+1$, there must be exactly one family.

We now prove (iii). Consider $\lambda\in \SF$. Setting $\mathfrak{g}(f_k)$ to be
the generation number of the family $f_k$, we can  write in a unique way
$$ 
\lambda=\sum_k\alpha_k f_k = \sum_{\mathfrak{g}(f_k)=
i_>}\alpha_kf_k+\sum_{\mathfrak{g}(f_k)\neq i_>}\alpha_kf_k.
$$
Now, by definition of $\SF$, we have $\lambda= \sum_j \lambda_j \ee_j$ with
$\lambda_j\in \ZZ$. One easily sees, by Remark \ref{no.int}, that for all $j\in
\AA_{i_>+1}$ one has
$\lambda_j=-\alpha_k$ for one (and only one) $f_k$ of generation number $i_>$.
Then
$$
\lambda-\sum_{\mathfrak{g}(f_k)= i_>}\alpha_kf_k \in \SF
$$ and the claim follows by recursion on the maximal age.

Then, we prove (iv). Assume $\lambda\neq0$, otherwise the thesis is obvious. We
have observed that $\Su(\lambda)$ must contain at least two elements of the
generation $i_<$ and at least two elements of the generation $i_>+1$. The
assumption in (iv) implies that one of these two generations contains exactly
two elements of $\Su(\lambda)$ and that moreover, for all $j\neq i_<,i_>+1$,
$\Su(\lambda)$ contains no elements of the generation $j$. We assume that the
generation $i_<$ contains exactly two elements of $\Su(\lambda)$ which means
that the linear combination defining $\lambda$ contains one and only one family
$f_{k_0}$ of generation number $i_<$ (the case with generation $i_>+1$ is
symmetric), appearing with the coefficient $\alpha_{k_0}\neq0$. Then we
distinguish two cases: either $i_<=i_>$ or $i_<\neq i_>$. If $i_<=i_>$, then the
thesis follows easily by Remark \ref{no.int}. If $i_<\neq i_>$, then the
generation $i_<+1$ contains no element of $\Su(\lambda)$. But this means that
the two children in $f_{k_
0}$ (call them $\cc_1,\cc_2$) must be 
canceled out, which implies that the two families $f_{k_1},f_{k_2}$ in which
$\cc_1,\cc_2$ appear as parents ($k_1\neq k_2$ since siblings do not marry each
other) have the same (non-zero) coefficient
$\alpha_{k_1}=\alpha_{k_2}=\alpha_{k_0}$. But then also the two spouses of
$\cc_1,\cc_2$ must cancel out: consider for instance the spouse of $\cc_1$ and
call it $\mathbf{s}_1$. We have that $\mathbf{s}_1$ appears as a child in a
family of generation number $i_<$: we call this family $f_{k_3}$. Note that
$k_3\neq k_0$ since $\mathbf{s}_1\notin\{\cc_1,\cc_2\}$. The fact that
$\mathbf{s}_1$ is canceled out implies that $\alpha_{k_3}=\alpha_{k_1}\neq0$,
but this is absurd since $f_{k_0}$ is the only family of generation number $i_<$
to appear in the linear combination. This completes the proof of (iv).

Finally, the property (v) is a simple remark when $\lambda$ is a single family
vector and trivially generalizes to the case $\lambda\in\SF$. This completes the
proof of the lemma.

\subsection{Proof of Theorem
\ref{thm:FiniteResonantSet}}\label{sec:ProofThmFiniteResonantSet}
Now we prove the existence of non degenerate generation sets, according to
Definition \ref{defi:NonDegeneracy}. A preliminary but very important step is to show
that the linear and quadratic relations defining $\mathcal M$, see \eqref{mmmm},
do not imply any linear relation except those given by $\pi_\SS(\lambda)=0$ for
all $\la\in \SF$. It is clear that if this were not true one could not impose
condition (i) of Definition \ref{defi:NonDegeneracy}.
Specifically we prove
\begin{lemma}\label{fifi}
Consider a codimension one subspace $\Sigma\subset \R^{2m}$ and the sets
$\mathcal M$  and $\mathcal L$ defined respectively in \eqref{mmmm} and  \eqref{lll}.
Then,  $\mathcal M\subset \Sigma$ implies $\mathcal L \subset \Sigma$.
\end{lemma}
The strategy of this proof relies on the choice of a good set of variables for
$\mathcal L$ (and consequently $\mathcal M$) as explained in Section \ref{sec:SomeGeometry}.
Namely, up to a reordering of the vectors $\be_j$, the matrix whose rows are the 
abstract families (see Definition \ref{def.fam})  is in {\em row echelon form}, see \eqref{mammete}. This gives
a recursive rule to fix the dependent variables (the {\em pivots}) as well as
the circles for the remainig variables. Then we write the relation defining
$\Sigma$ in the independent variables of $\mathcal L$ and the condition
$\mathcal L \not\subset \Sigma$ means that the coefficients are not all zero.  
Then  with respect to the {\em youngest} variable $\Sigma$ defines a line, while the
quadratic relations a {\em non-degenerate} circle, which obviously cannot be
contained in a line. 
\begin{proof}[Proof of Lemma \ref{fifi}]
We will prove that if $\mathcal L\not\subset\Sigma$, then $\mathcal M\not\subset
\Sigma$. Namely, we will prove that, for any given codimension one subspace
$\Sigma$ which does not contain $\mathcal L$, we can choose $\SS\in\mathcal
M\setminus\Sigma$.

If we denote by $\vv_j^{(1)},\vv_j^{(2)}$ the two components of $\vv_j\in\R^2$,
a codimension one subspace $\Sigma\subset\R^{2m}$ is defined by an equation of
the form
\begin{equation}\label{sigma}
\sum_{j=1}^m\sum_{k=1}^2\lambda_{jk}\vv_j^{(k)}=0\ .
\end{equation}
Now, for simplicity of notation and without loss of generality, we reorder the
basis $\{\ee_j\}_{j=1}^m$ of $\Z^m$  so that two siblings belonging to the same
abstract family always have consecutive subindices. In matrix notation, the
condition of $\SS$ being a generation set can be denoted
$$
\pi_\SS(F^T)=0\,,\quad \pi^{(2)}_\SS(F^T)=0\ ,
$$
where $F$ is a matrix whose rows are given by the abstract families and $F^T$ denotes its transpose. We choose
to order the rows of $F$ so that the matrix is in {\em lower row echelon} form
(see figure).
\begin{equation}\label{mammete}
\begin{array}{c  c c c c c c c c c c c } \ww_1&\ww_2&\ww_3&\ww_4&\ww_5&\pp_5&
\ww_6 &\pp_6 &\ww_7 &\pp_7 &\ww_8&\pp_8\\ \\ 1 &1&0&0&-1&-1&\!\! \!\!\!\vline
\quad0&0&0&0  &0&0 \\ \cline{7-8}
0&0&1&1&0&0&-1&-1 &\!\! \!\!\!\vline \quad 0&0  &0&0\\ \cline{9-10}
0&0&0&0&1&0&1&0&-1&-1& \!\! \!\!\!\vline \quad 0&0 \\
\cline{11-12}0&0&0&0&0&1&0&1&0&0  &-1&-1 \end{array}
\end{equation}
Each row of a matrix in lower row echelon form has a {\em pivot}, i.e. the first
nonzero coefficient of the row starting from the right. Being in lower row
echelon form means that the pivot of a row is always strictly to the right of
the pivot of the row above it. In the matrix $F$ defined by the abstract
families, the pivots are all equal to $-1$ and they correspond to one and only
one child from each family. In order to exploit this fact, we accordingly rename
the elements of the generation set by writing
$\SS=(\pp,\ww)\in\R^{2a}\times\R^{2b}$, with $a=(N-1)2^{N-2}$,
$b=m-a=(N+1)2^{N-2}$, where the $\pp_j\in\R^2$ are the elements of the
generation set corresponding to the pivots and the $\ww_{\ell}\in\R^2$ are all
the others, i.e. all the elements of the first generation and one and only one
child (the non-pivot) from each family. Here, the index $\ell$ ranges from $1$
to $b$, while the index $j$ ranges from $2^{N-1}+1$ to $b$ (note that
$a+2^{N-1}=b$), so that a couple $(\pp_j,\ww_{\ell})$ corresponds to 
a couple of siblings if and only if $j=\ell$. Then, the linear relations
$\pi_\SS(F^T)=0$ can be used to write each $\pp_j$ as a linear combination
containing only  the $\ww_{\ell}$'s with $\ell\leq j$:
\begin{equation}\label{pivot}
\pp_j=\sum_{\ell\leq j}\mu_{\ell}\ww_{\ell}, \quad \mu_{\ell}\in\Q\ .
\end{equation}
Finally, the quadratic relations $\pi_\SS^{(2)}(F^T)=0$ constrain each
$\ww_{\ell}$ with $\ell>2^{N-1}$ (i.e. not in the first generation) to a circle
depending on the $\ww_j$ with $j<\ell$; note that this circle has positive
radius provided that the parents of $\ww_{\ell}$ are distinct.
Then, equation \eqref{pivot} together with Lemma \ref{gentree} (i) implies that the
left hand side of equation~\eqref{sigma} can be rewritten in a unique way as a linear
combination of the $\ww_{\ell}$'s only. Thus, we have
\begin{equation}\label{sigma2}
\sum_{\ell=1}^b\sum_{k=1}^2\eta_{\ell, k}\ww_{\ell}^{(k)}=0\ .
\end{equation}
Hence, the assumption that $\mathcal L\not\subset\Sigma$ is equivalent to the
fact that $\eta\in\R^{2b}$ does not vanish. Let $$\bar\ell:=\max\left\{\ell\ \middle|\
(\eta_{\ell,1},\eta_{\ell,2})\neq(0,0)\right\}\ .$$
If $\bar\ell\leq 2^{N-1}$, then $\ww_{\bar\ell}$ is in the first generation.
Since there are no restrictions (either linear or quadratic) on the first
generation, the statement is trivial. Hence assume $\bar\ell>2^{N-1}$. As we
have discussed previously, we can assume (by removing from $\mathcal M$ a proper
submanifold of codimension one) that $\vv_h\neq\vv_k$ for all $h\neq k$. Then
the quadratic constraint on $\ww_{\bar\ell}\in\R^2$ gives a circle of positive
radius. Since \eqref{sigma2} defines a line in the variable $\ww_{\bar \ell}$ we
can ensure that the relation \eqref{sigma2} is not fulfilled by excluding at
most two points on this circle.  Thus we are able to construct $\SS\in\mathcal
M\setminus\Sigma$.

\end{proof}
\begin{remark}\label{acquacalda}
By Lemma \ref{fifi}, if a linear equation $\sum_j \lambda_j \vv_j\equiv 0$
identically on $\mathcal M$ then it must be $\lambda\in \langle \FF\rangle$.
\end{remark} 

Now we are ready to prove Theorem \ref{thm:FiniteResonantSet}. 

\begin{proof}[Proof of Theorem \ref{thm:FiniteResonantSet}] Set
$$\mathcal D_0= \cup_{\lambda \in \mathcal R_{2d}\setminus \langle\FF\rangle}
\{\SS \in \mathcal M:\;\; \pi_\SS(\lambda)=0\}. $$
By Remark \ref{acquacalda},  this is an algebraic manifold of codimension one in
$\mathcal M$. Moreover, by definition, in $\MM\setminus \mathcal D_0$ the condition (i) of Definition \ref{defi:NonDegeneracy} is satisfied.

To impose the second condition we proceed by induction. As in our
geometric construction of $\mathcal M$ (see Section \ref{sec:SomeGeometry}), we  suppose to have  fixed $i$
generations and $0\leq h<2^{N-2}$ families with children in the $i+1$-th
generation. This means that we have fixed $\pi_\SS(\be_j)$ for all $j\leq
2^{N-1} i$ and for some subset of cardinality $2h$ of $\be_j$ in the  $i+1$-th
generation. Let us denote by $A$ the set of indexes $j$ such that 
$\pi_\SS(\be_j)$ has been fixed.
 Our inductive hypothesis is that  all the non-degeneracy conditions with
support  contained in $A$ are satisfied. In particular, this implies that all the
$\vv_j$ with $j\in A$ are distinct.
  Let us denote by $\cc_1,\cc_2\in \AA_{i+1}$ the next children we wish to
generate and by    $\be_{j_1},\be_{j_2}\in \AA_i$ their  {\em parental couple}.
We wish to fix
  $\ww_1=\pi_\SS(\cc_1)$ and  $\ww_2=\pi_\SS(\cc_2)$ so that the non-degeneracy
conditions hold. Due to the linear relations
$\ww_2=-\ww_1+\vv_{j_1}+\vv_{j_2}$ while the quadratic relations read
$(\vv_{j_1}-\ww_1,\vv_{j_2}-\ww_1)=0$.
Let us consider $\mu\in \Z^m$ of the form
\begin{equation}\label{inaxs}
\mu= \sum_{j\in A}\xi_j \be_j + a \cc_1+b \cc_2\,, \quad \sum_{j\in A}\xi_j 
+a+b=1\,, \quad \sum_{j\in A}|\xi_j|  +|a|+|b|\leq 2d-1
\end{equation}
and study $K_\SS(\mu)$.

Recall that $\ww_2=-\ww_1+\vv_{j_1}+\vv_{j_2}$  and
$|\ww_2|^2=-|\ww_1|^2+|\vv_{j_1}|^2+|\vv_{j_2}|^2$. 
We have
$$
\left\{\begin{aligned}
 \pi_\SS(\mu) &=  \sum_{j\in A}\xi_j \vv_j + (a-b) \ww_1+
b(\vv_{j_1}+\vv_{j_2})\\
\pi^{(2)}_\SS(\mu)&=  \sum_{j\in A}\xi_j |\vv_j|^2 + (a-b) |\ww_1|^2+
b(|\vv_{j_1}|^2+|\vv_{j_2}|^2). 
\end{aligned}\right.
$$
We set $\alpha:= a-b$ and $$\lambda:= \sum_{j\in A}\lambda_j \ee_j= \sum_{j\in
A}\xi_j \ee_j+b(\ee_{j_1}+\ee_{j_2})$$
so that 
$$
K_\SS(\mu)= 
|\pi_\SS(\lambda) +\alpha \ww_1|^2-\pi^{(2)}_\SS(\lambda) - \alpha |\ww_1|^2.
$$

If $\alpha=0$, then $a=b$ and therefore
$\lambda-\mu=a(\ee_{j_1}+\ee_{j_2}-\cc_1-\cc_2)\in\langle\FF\rangle$ and
$K_\SS(\mu)=K_\SS(\la)$. Moreover, we have that $\sum_j \lambda_j=1$ and
$\sum_j |\lambda_j|\leq 2d-1$. Since the support of $\lambda$ is contained
in $A$, the non-degeneracy condition (ii) of Definition \ref{defi:NonDegeneracy} for the vector $\mu$ follows from the inductive hypothesis.

Otherwise, assume  $\al\neq 0$. Then, since $ |\ww_1|^2=(\vv_{j_1}+\vv_{j_2},\ww_1)-(\vv_{j_1},\vv_{j_2})$, we  get
\[
\begin{split}
K_\SS(\mu)=&\, |\pi_\SS(\lambda)|^2 +\alpha^2| \ww_1|^2 +2\alpha
(\pi_\SS(\lambda),\ww_1)-\pi^{(2)}_\SS(\lambda) - \alpha |\ww_1|^2\\
=&\,K_\SS(\lambda) +2\alpha (\pi_\SS(\lambda),\ww_1)+ \alpha(\alpha-1)
[(\vv_{j_1}+\vv_{j_2},\ww_1)-(\vv_{j_1},\vv_{j_2})]\\
=&\,K_\SS(\lambda) +\alpha
\Big(2\pi_\SS(\lambda)+(\alpha-1)(\vv_{j_1}+\vv_{j_2}),\ww_1\Big)-
\alpha(\alpha-1) (\vv_{j_1},\vv_{j_2}).
\end{split}
\]

If $2\pi_\SS(\lambda)+(\alpha-1)(\vv_{j_1}+\vv_{j_2})\neq 0$  then $K_\SS(\mu)=
0$ defines a line in the plane $\ww_1\in \R^2$. Then the non-degeneracy
condition amounts to fixing $\ww_1$ so that 
$K_\SS(\mu)\neq 0$ i.e. by excluding at most two points on the circle
$(\ww_1-\vv_{j_1}, \ww_1-\vv_{j_2})=0$.

Suppose now that $2\pi_\SS(\lambda)+(\alpha-1)(\vv_{j_1}+\vv_{j_2})= 0$, then
$K_\SS(\mu)$ does not depend on the choice of $\ww_1$. We have to show that
either $K_\SS(\mu)\neq 0$ for all $\SS\in \mathcal M\setminus\mathcal D_0$ or we
get  the special case allowed in Definition \ref{defi:NonDegeneracy} (ii).    We
claim that\footnote{This motivates our choice of $\RR_{2d}$ in Definition
\ref{defi:NonDegeneracy}.}
$$
\eta:= 2\lambda+(\alpha-1)(\ee_{j_1}+\ee_{j_2})\in \mathcal R_{2d}.
$$
Indeed, since $\sum_j\lambda_j+\alpha=1$ then $\sum_j \eta_j=0$. Moreover $\eta=
2\sum_{j\in A}\xi_j\be_j +(a+b-1)(\be_{j_1}+\be_{j_2})$ which,  by \eqref{inaxs},
implies
$$
\sum_j |\eta_j| \le 2 (\sum_j|\xi_j|+|a|+|b|+1) \le 4d.
$$ 
Now by definition for all $\SS\in \mathcal M\setminus\mathcal D_0$, we have
$\pi_\SS(\eta)=0$ if and only if $\eta\in \langle \FF\rangle$. This in turn
implies that not only
$$\pi_\SS(\eta)=2\pi_\SS(\lambda)+(\alpha-1)(\vv_{j_1}+\vv_{j_2})= 0$$ but also
(see Remark \ref{pipi2})
$$\pi^{(2)}_\SS(\eta)=2\pi^{(2)}_\SS(\lambda)+(\alpha-1)(|\vv_{j_1}|^2+|\vv_{j_2
}|^2)= 0.$$ Hence
$$ K_\SS(\lambda)= \frac{(\alpha-1)^2}{4}|\vv_{j_1}+\vv_{j_2}|^2+
\frac{\alpha-1}{2}(|\vv_{j_1}|^2+|\vv_{j_2}|^2)
$$ and in conclusion
$$K_\SS(\mu)= \frac{(\alpha-1)(\alpha+1)}{4}|\vv_{j_1}-\vv_{j_2}|^2. $$
We have that   $\be_{j_1}-\be_{j_2}\in \RR_{2d}$ and Lemma \ref{gentree} (ii) implies 
$\be_{j_1}-\be_{j_2}\not\in \langle \mathcal F\rangle$. 
Therefore, for $\SS\in \mathcal M\setminus
\mathcal D_0$, we have $\vv_{j_1}\neq \vv_{j_2}$ (see Remark \ref{acquacalda}). Then $K_\SS(\mu)$   vanishes on
$\mathcal M\setminus\mathcal D_0$ only if $\alpha=\pm 1$.

If $\alpha=1$ then $\la\in \RR_{2d}$ and $\pi_{\SS}(\lambda)=0$, which holds
true in $\mathcal M\setminus \mathcal D_0$ if only if $\lambda\in \langle
\FF\rangle$; then $$\mu-\cc_1=\lambda-b(\ee_{j_1}+\ee_{j_2}-\cc_1-\cc_2)\in
\langle \FF\rangle$$ and the non-degeneracy condition (ii) in Definition
\ref{defi:NonDegeneracy} holds.

If $\alpha=-1$, then one symmetrically defines $$\tilde\lambda:= \sum_{j\in
A}\tilde\lambda_j \ee_j= \sum_{j\in A}\xi_j \ee_j+a(\ee_{j_1}+\ee_{j_2})$$ and
proceeding as above one obtains $\tilde\la\in \RR_{2d}$ and
$\pi_{\SS}(\tilde\lambda)=0$, which implies $\tilde\lambda\in \langle
\FF\rangle$; finally
$$\mu-\cc_2=\tilde\lambda-a(\ee_{j_1}+\ee_{j_2}-\cc_1-\cc_2)\in \langle
\FF\rangle$$ which again ensures that the non-degeneracy condition (ii) holds.
\end{proof}

\section{Dynamics of the toy model}
\subsection{Invariant subspaces for one and two generations}
We now study the invariant subspaces of $H_\SS$ where $\SS$ is an acceptable
set, see Definition \ref{acceptable}. The simplest non trivial orbits are those
where we fix $j= 1,\dots,N$ and set  $b_i=0$ for all $i\neq j$. This is
an invariant subspace by Definition \ref{acceptable} item 4(d).
By gauge invariance and reality (resp. items 4(c) and 4(b) of Definition
\ref{acceptable}) we have that the Hamiltonian restricted to this subspace is a
single monomial $|b_j|^{2d}$ with a real non-zero coefficient. Moreover, this
coefficient does not depend on $j$ (it follows, for instance, by Definition
\ref{acceptable} item 4(f)). We have proved, for any fixed surface level of the
mass, the existence of $N$ periodic orbits all with the same frequency. We
denote by $\TT_j$ the corresponding periodic orbit with $|b_j|=1$ and $b_i=0$
for $i\neq j$.

We can now suppose that all the $b_j$'s are zero except two consecutive ones. 
By Definition \ref{acceptable} item 4(e) we can restrict ourselves to the case
when these two generations are the first and the second.
Thus, we get the Hamiltonian
$$
h(b_1,\bar b_1,b_2,\bar b_2)=
(|b_1|^2+|b_2|^2)^{d-2}\Big(-\frac14(|b_1|^4+|b_2|^4)+ \Re(b_1^2\bar
b_2^2)\Big)+\frac1n\PP(b_1,\bar b_1,b_2,\bar b_2,\frac1n).
$$
with the constant of motion $J= |b_1|^2+|b_2|^2$.
We know by Definition \ref{acceptable} item 4 that $h$ is symmetric with respect
to the  exchange of $b_1$ and $b_2$, real and gauge invariant. Moreover it has
even degree in its variables. Thus we  can have only a finite number of possible
fundamental {\em building blocks} which appear through sums and products:
\begin{enumerate}
\item Integrable terms  $L^{(j)}:= |b_1|^{2j}+ |b_2|^{2j}$;
\item Non-integrable terms which are multiples of a family relation $\Re
[(b_1\bar b_2)^{2 k}]$.
\end{enumerate}
We remark that all the integrable terms $L^{(j)}$ can be written in terms of
$|b_1|^2|b_2|^2$ and $J$ in the same way the non integrable terms are
polynomials in $|b_1|^2|b_2|^2$  and $\Re [(b_1\bar b_2)^{2}]$. We now reduce
the degrees of freedom passing to one complex variable $c$, one (cyclic) angle
$\vartheta$ and the conserved quantity $J$.
Explicitly we have
\begin{equation}\label{symp1}
J= |b_1|^2+|b_2|^2, \; b_1= \sqrt{J-|c|^2} e^{\ii \vartheta}\,,\; b_2= c e^{\ii
\vartheta}.
\end{equation}
This change of variables is symplectic and the new symplectic form is $\frac12
(dJ\wedge d\vartheta +\ii  dc \wedge d \bar c )$. Note that, writing both
$|b_1|^2|b_2|^2$  and $\Re [(b_1\bar b_2)^{2}]$ in terms of the new variables,
the term $J-|c|^2$ always factors out. This means that, subtracting the constant
terms depending only on $J$, we get the Hamiltonian
$$
h(J,c,\bar c)=(J-|c|^2)\Big( J^{d-2} \big(\frac12|c|^2 + \Re(c^2)\big)+
\frac{1}{n}Q(J,|c|^2,\Re(c^2))\Big),
$$
where  the dependence of $Q$ on its arguments is polynomial and homogeneous of
degree $d-1$ and we have $Q(J,0,0)=0$. Now we may extract the linear terms in
$|c|^2,\Re(c^2)$ from $Q$ (note that for the quintic NLS $d=3$ these are the
only possible terms) and restricting to the surface level $J=1$ we get an
expression of the form
\begin{equation}\label{hccbar}
h(c,\bar c)= \kappa_n (1-|c|^2)\Big( a_n|c|^2 + \Re(c^2)+ \frac{1}{n}\mathcal
Q(|c|^2,\Re(c^2))\Big)\ .
\end{equation}
where
\[
 \mathcal Q(|c|^2,\Re(c^2))=
Q(1,|c|^2,\Re(c^2))-\pa_{2}Q(1,0,0)|c|^2-\pa_{3}Q(1,0,0)\Re(c^2).
\]
Note that $\mathcal Q$ has a zero of order two in its variables $\kappa_n= 1
+\OO(1/n)$ and $a_n= (\frac12+\OO(1/n))$. It is natural to pass the quadratic
part of the Hamiltonian in hyperbolic normal form by defining
\begin{equation}\label{def:anANDomega}
\Re(\omega^2):= -a_n\,,\quad\text{namely}\quad \omega= e^{\ii \theta} \;\text{
with} \quad  \theta = \frac12 \arccos(-a_n)=\frac{\pi}{3}+\OO(n^{-1})
\end{equation}
and setting
\begin{eqnarray}\label{ctopq}
c&=&\frac{1}{\sqrt{\Im(\omega^2)}}(\omega q+\bar\omega p)\\
\nonumber \bar{c}&=&\frac{1}{\sqrt{\Im(\omega^2)}}(\bar\omega q+\omega p).
\end{eqnarray}

\begin{lemma}
The change of variables given by \eqref{ctopq} is symplectic i.e.
$\frac{\ii}{2}dc\wedge d\bar c=dp\wedge dq$ and the Hamiltonian in the new
variables is given by
\begin{equation}\label{2genpq}
h(p,q)=\kappa_n\Big(1-\frac{1}{{\Im(\omega^2)}} (p^2+ q^2+ 2 \Re(\omega^2) p  q
)\Big) \Big(2\Im(\omega^2) pq+ \frac{1}{n}P(pq,p^2+q^2)\Big)
\end{equation}
 with $P$ having a zero of degree at least two in its arguments.
\end{lemma}
\begin{proof}
We have:
\[
\begin{split}
|c|^2&= \frac{1}{{\Im(\omega^2)}} (p^2+ q^2+ 2 \Re(\omega^2) p  q )\\
\Re(c^2)&=  \frac{1}{{\Im(\omega^2)}}(\Re( \omega^2) (p ^2 +   q ^2) + 2p  q )
\end{split}
\]
which imply
$$
-\Re(\omega^2) |c|^2+ \Re(c^2)= \frac{2}{{\Im(\omega^2)}}  (1-\Re(\omega^2)^2 )
p  q = 2\Im(\omega^2) p  q 
$$
and hence substituting into \eqref{hccbar} we get the thesis.
\end{proof}

The flow generated by the Hamiltonian \eqref{2genpq} leaves invariant the
ellipse $\mathcal E$ with equation $p^2+ q^2+ 2 \Re(\omega^2) p 
q=\Im(\omega^2)$, which corresponds to the periodic orbit $\TT_2$, while the
hyperbolic fixed point $(0,0)$ corresponds to the periodic orbit $\TT_1$. We are
going to prove the existence of heteroclinic connections linking $(0,0)$ to a
point in $\mathcal E$, i.e. sliding from the periodic orbit $\TT_1$ to the
periodic orbit $\TT_2$.


Neglecting the term $(1/n)P(pq,p^2+q^2)$ in \eqref{2genpq}, one easily sees that
there is a heteroclinic connection lying on $q=0$, flowing from the point
$(0,0)$ as $t\to -\infty$ to the hyperbolic critical point
$(p,q)=(\Im(\omega^2),0)$ on $\mathcal E$ as $t\to +\infty$.
%
Now, we can deal with the full system using perturbative methods.

\begin{lemma}\label{lemma:SeparatrixAsGraph}
 The Hamiltonian system given by \eqref{2genpq} has a hyperbolic critical point
$(p^*,q^*)=(\sqrt{3}/2,0)+\OO(n^{-1})$, which belongs to $\mathcal E$, and a
heteroclinic connection which tends to this point in forward time and to the
point $(p,q)=(0,0)$ in backward time. Moreover, this connection can be written
as a graph
\[
 q=\xi (p),\,\,p\in [0,p^*]
\]
and it satisfies $\sup_{p\in [0,p^*]}|\xi(p)|=\OO(n^{-1})$.
\end{lemma}
The proof of this lemma is straightforward. 


\begin{remark}
Since the Hamiltonian \eqref{2genpq} is symmetric in $(p,q)$, then there is also
the hyperbolic critical point $(p,q)=(q^*,p^*)=(0,\sqrt{3}/2)+\OO(n^{-1})$,
which belongs to $\mathcal E$, and a heteroclinic connection which tends to this
point in backward time and to the point $(p,q)=(0,0)$ in forward time. Such
heteroclinic connection can be written as a graph
\[
 p=\xi (q),\,\,q\in [0,p^*].
\]
\end{remark}

\subsection{Adapted coordinates for the $j$-th periodic orbit}
In this section we study the dynamics of the toy models
\[
h(b)=\left(\sum_{i=1}^N|b_i|^2\right)^{d-2}\left[-\frac14\sum_{i=1}^N|b_i|^4+
\sum_{i=1}^{N-1}\Re(b_i^2\bar b_{i+1}^2)\right]+\frac{1}{n}\PP\left(b,\ol
b,\frac1n\right).
\]
Following \cite{CKSTT}, we will take advantatge of the mass conservation to make
a symplectic reduction. This will allow us to obtain certain good systems of
coordinates. 

To make the symplectic reduction we fix the mass $\MM(b)=1$. Note that the toy
model is invariant by certain rescaling and time reparameterization. So, from
orbits in $\MM(b)=1$ we can obtain orbits for any mass. Now, we perform the
change of coordinates close to the $j$ periodic orbit
\[
\left(b_1,\ol b_1,\ldots, b_N, \ol b_N\right)\mapsto \left(c_1^{(j)},\ol
c_1^{(j)},\ldots, J, \theta^{(j)} ,\ldots, c_N^{(j)},\ol c_N^{(j)}\right)
\]
defined by 
\begin{equation}\label{def:SaddleAdaptedCoordinates}
 b_j=\sqrt{J-\sum_{k\neq j}\left|c_k^{(j)}\right|^2}e^{\ii\theta^{(j)}},\,\,
b_k=c_k^{(j)}e^{\ii\theta^{(j)}} \ \text{ for all }\ k\ne j,
\end{equation}
where $\theta_j^{(j)}$ is the angular variable over the periodic orbit and
$J=\sum_{k=1}^N|b_k|^2$ is the mass. It can be checked that this change of
coordinates is symplectic. From now we omit the superscript $(j)$ when it is
clear in the neighborhood of which saddle we are dealing with. The new
Hamiltonian is independent of $\theta$ since the mass $J$ is a first integral.
Fixing $J=1$, the system for the variables $c=(c_1,\ldots,
c_{j-1},c_{j+1},\ldots c_N)$  is Hamiltonian with respect to
\begin{equation}\label{def:HamAfterSymplecticReduction}
\begin{split}
 H^{(j)}(c)=&-\frac{1}{4}\sum_{k\neq j}|c_k|^4-\frac{1}{4}\left(1-\sum_{k\neq
j}|c_k|^2\right)^2+\sum_{k\neq j,j+1}\Re(c_k^2\ol
c_{k-1}^2)\\
&+\left(1-\sum_{k\neq
j}|c_k|^2\right)\Re\left(c_{j-1}^2+c_{j+1}^2\right)+\frac{1}{n}\wt\PP\left(c,\ol
c, \frac 1n\right)
\end{split}
\end{equation}
and the symplectic form $\Omega=\sum_{k\neq j}\frac{\ii}{2}dc_k\wedge d\ol c_k$,
where $\wt\PP$ is the polynomial $\PP$ introduced in Theorem
\ref{thm:FiniteResonantSet} expressed in the new variables $c$. The Hamiltonian
system can be split as 
\[
 H^{(j)}(c)=H_2^{(j)}(c)+H_4^{(j)}(c)
\]
where $H_2^{(j)}(c)$ contains the quadratic monomials and $H_4^{(j)}(c)$
contains the higher order terms, that is monomials of even degree from 4 to
$2d$. Statements 4(e) and 4(f) of Definition \ref{acceptable} imply the
following lemma. 
\begin{lemma}
 The Hamiltonian $H_2^{(j)}(c)$ is of the form 
\[
 H_2^{(j)}(c)=a_n\sum_{k\neq j}|c_k|^2+a_n\kk_n
\Re\left(c_{j-1}^2+c_{j+1}^2\right)
\]
where $a_n=1/2+\OO(n^{-1})$, $\kk_n=1+\OO(n^{-1})$.
\end{lemma}
Being close to $\TT_j$ corresponds to $c\sim 0$. To analyze this local behavior,
we diagonalize the linear part at the critical point. Note that for $c_k$,
$k\neq j-1,j+1$ it is already diagonalized so we apply the change of variables
\eqref{ctopq} to the adjacent modes $c_{j\pm 1}$. We  obtain the new quadratic
part of the Hamiltonian
\begin{equation}\label{quad.cpq}
 H^{(j)}_2(p,q,c)=a_n\sum_{k\in
\PP_j}|c_k|^2+\la_n\left(p_1q_1+p_2q_2\right),\quad \la_n= 2\text{
Im}(\omega^2)=2\sqrt{1-a_n^2}= \sqrt{3}+\OO(n^{-1}).
\end{equation}
Note that this change of coordinates transform $\Omega$ into the symplectic form
\[
 \Omega=\sum_{k\neq j-1,j,j+1}\frac{\ii}{2}dc_k\wedge d\ol c_k+dp_1\wedge
dq_1+dp_2\wedge dq_2.
\]
To study the Hamiltonian expressed in the new variables we introduce
\[
\PP_j=\{1\leq k\leq N; k \neq j-1,j,j+1\},
\]
which is the set of subindexes of the elliptic modes. From now on we will denote
by $q$ and $p$ all the stable and unstable coordinates $q=(q_1,q_2)$ and
$p=(p_1,p_2)$ respectively and by  $c$ all the elliptic modes, namely $c_k$ with
$k\in \PP_j$.

\begin{lemma}\label{lemma:Diagonalization}
 The change \eqref{ctopq} transforms the Hamiltonian
\eqref{def:HamAfterSymplecticReduction}  into the Hamiltonian
\begin{equation}\label{def:Ham:Diagonal}
  H^{(j)}(p,q,c)= H^{(j)}_2(p,q,c)+ H^j_4(p,q,c)
\end{equation}
with homogeneous polynomials $H^{(j)}_2(p,q,c)$ given by \eqref{quad.cpq} and
\[
  H^{(j)}_4(p,q,c)=  H^{(j)}_{\hyp}\left(p,q\right)+ H^{(j)}_{\el}(c)+
H^{(j)}_{\mix}\left(p,q,c\right)
\]
where
\begin{align}
H^{(j)}_\hyp(p,q)=&-2p_1q_1 (p_1^2+q_1^2-p_1q_1)-2p_2q_2
(p_2^2+q_2^2-p_2q_2)\notag\\
&+\sum_{k,\ell=0}^2\nu_{k\ell}
p_1^kq_1^{2-k}p_2^\ell
q_2^{2-\ell}+\OO\left(\frac{1}{n}(p_1+q_1+p_2+q_2)^4\right)\notag\\
H^{(j)}_{\el}\left(c\right)=&-\frac{1}{4}\sum_{k\in\PP_j}|c_k|^4-\frac{1}{4}
\left(\sum_{k\in\PP_j}|c_k|^2\right)^2 \notag\\
&+\sum_{k\in\PP_j\setminus\{j+2\}}\Re(c_k^2\ol{c_{k-1}}^2)+\OO\left(\frac{1}{n}
\sum_{k,k'\in\PP_j}|c_k|^2|c_{k'}|^2\right)\notag\\
H^{(j)}_{\mix}(p,q,c)=&-\sqrt{3}\sum_{k\in
\PP_j}|c_k|^2\left(q_1p_1+q_2p_2\right)\notag\\
&+\frac{2\sqrt{3}}{3}\Re\Big(\left(\omega_0 p_1+\ol\omega_0 q_1\right)^2 c_{j-
2}^2\Big)+\frac{2\sqrt{3}}{3}\Re\Big(\left(\omega_0 p_2+\ol\omega_0 q_2\right)^2
c_{j+
2}^2\Big)\notag\\
&+\OO\left(\frac{1}{n}\sum_{k\in \PP_j}|c_k|^2(p_1+q_1+p_2+q_2)^2\right)\notag
\end{align}
for some constants $\nu_{k\ell}\in\R$ and $\om_0=e^{\ii\frac{\pi}{3}}$. It can
be easily checked that all $\nu_{k\ell}$ satisfy $\nu_{k\ell}\neq 0$.
\end{lemma}

Now for the Hamiltonian \eqref{def:Ham:Diagonal}, the periodic $\TT_j$ has
become the critical point $(p,q,c)=(0,0,0)$ which is of mixed type (four
hyperbolic eigenvalues and $2N-6$ elliptic eigenvalues). Thanks to Lemma
\ref{lemma:SeparatrixAsGraph} and the particular form of the Hamiltonian
\eqref{def:Ham:Diagonal} the hyperbolic directions give connections to the
neighboring periodic orbits $\TT_{j\pm 1}$. In the full phase space the
heteroclinic connection between $(0,0,0)$ and $\TT_{j+1}$ can be parameterized
as a graph by
\[
 (p_1,q_1,p_2,q_2,c)=(0,0,p_2,\xi(p_2),0).
\]
Recall that in the cubic case, this connection is just given by
$c_k=q_1=p_1=q_2=0$ (see \cite{CKSTT}).

Following \cite{CKSTT, GuaKal} we look for orbits which shadow this
concatenation of heteroclinic orbits.

\subsection{The iterative argument: almost product
structure}\label{sec:IterativeTheorem}
To prove Theorem \ref{thm:ToyModelOrbit} and shadow the concatenation of
heteroclinic connections, we follow the approach in \cite{GuaKal}. That is,  we 
consider several
co-dimension one sections $\{\Sigma^{\inn}_j\}_{j=1}^N$ and transition maps
$\BB^j$ from one section $\Sigma^{\inn}_j$ to the next one
$\Sigma^{\inn}_{j+1}$. The maps are given by the flow associated to the
Hamiltonian \eqref{def:Ham:Diagonal}.
We consider  sets $\{\VV_j\}_j,\ \VV_j \subset \Sigma_j^{\inn},\ j=1,\dots,N-1$,
of a very particular form, which in \cite{GuaKal} were called sets with  {\it
almost product structure} (see Definition \ref{def:ProductLikeSet} below).
Moreover, we impose that these sets satisfy that
$\VV_{j+1}\subset\BB^{j}\left(\VV_j\right)$ and that none of them is empty.
Each set $\VV_j$ is located close to the stable manifold of the periodic orbit
$\TT_j$. Composing all these maps we will be able to find orbits claimed to
exist
in Theorem \ref{thm:ToyModelOrbit}. Note that these sets will be slightly
different from the ones in \cite{GuaKal} due to the deviation of the
heteroclinic connections given in Lemma \ref{lemma:SeparatrixAsGraph}. 

In this section we keep the superindexes $(j)$ in the variables since it
involves two consecutive adapted system of coordinates.
We start by defining  transversal
sections to the flow. We use  the coordinates adapted to the saddle $j$,
$(p^{(j)},q^{(j)},c^{(j)})$ to define these
sections. In Lemma \ref{lemma:SeparatrixAsGraph} we have seen that the
heteroclinic connections  which connect
$(p^{(j)},q^{(j)},c^{(j)})=(0,0,0)$ with the previous and next saddles are
$\OO(n^{-1})$ close to $(p_1^{(j)},p_2^{(j)},q_2^{(j)},c^{(j)})=(0,0,0,0)$ and
$(p_1^{(j)},q_1^{(j)},q_2^{(j)},c^{(j)})=(0,0,0,0)$ respectively. Thus,
we define the map $\BB^j$ from the section
\begin{equation}\label{def:Section1Saddle}
\Sigma_j^{\inn}=\left\{q^{(j)}_1=\sigma\right\}
\end{equation}
to the  section
\[
\Sigma_{j+1}^{\inn}=\left\{q^{(j+1)}_1=\sigma\right\}.
\]
Here $\sigma>0$ is a small parameter that will be determined later on. We do not
define the map $\BB^j$ in the whole section but in a set
$\VV_j\subset\Sigma_j^{\inn}$, which lies close to the heteroclinic that
connects the
saddle $j-1$ to the saddle $j$.  Then, 
\[
 \BB^{j}:\VV_j\subset\Sigma_j^{\inn}\rightarrow\Sigma_{j+1}^{\inn}
\]
and we  choose the sets $\VV_j$ recursively in such a way that
\begin{equation}\label{cond:ComposeMaps}
 \VV_{j+1}\subset\BB^{j}\left(\VV_j\right).
\end{equation}
This condition  allows us to compose  all the maps $\BB^j$.

The sets $\VV_j$ will have a product-like structure as is stated in the next
definition, introduced in \cite{GuaKal}.
Before stating it, we fix $j$ and introduce notations adapted to $j$.
\begin{definition}
 We call $b_j$ the {\em   primary mode}, $b_{j\pm 1}$ {\em secondary modes},
$b_{j\pm 2}$ {\em adjacent modes} and all the others 
 {\em peripheral modes}. If $k<j$ we say that $b_k$ is a {\em trailing mode}
while if $k>j$ we say that $b_k$ is a {\em leading mode}.
 Finally we set
\[
\PP^-_j=\{k=1,\ldots, j-3\}\qquad
\PP^+_j=\{k=j+3,\ldots, N\}\,,\quad  \PP_j= \PP^-_j\cup \{j\pm 2\}\cup \PP^+_j.
\]
\end{definition} 

For a point $(p^{(j)},q^{(j)},c^{(j)})\in\Sigma_j^\inn$, we define
$c_-^{(j)}=(c_1^{(j)},\ldots,c_{j-2}^{(j)})$ and
$c_+^{(j)}=(c_{j+2}^{(j)},\ldots,c_N^{(j)})$. We define also the projections
$\pi_\pm(p^{(j)},q^{(j)},c^{(j)})=c_\pm^{(j)}$ and
$\pi_{\hyp,+}=(p^{(j)},q^{(j)},c_+^{(j)})$.

\begin{definition}\label{def:ProductLikeSet}
Fix
positive constants $r\in (0,1)$, $\de$ and $\sigma$ and
consider a multi-parameter set of positive constants
\begin{equation}\label{def:ProductLikeIndex}
 \II_j=\left\{C^{(j)},m_{\el}^{(j)},  M_{\el,\pm}^{(j)}, m_{\adj}^{(j)},
M_{\adj,\pm}^{(j)}, m_{\hyp}^{(j)}, M_{\hyp}^{(j)}\right\}.
\end{equation}
We associate to the set $\II_j$ a smooth function
$g_{j}(p_2,q_2)=g_{\II_j}(p_2,q_2)$, which is 
defined in \eqref{def:Function_g}.

Then, we say that a (non-empty) set $\UU\subset\Sigma_j^\inn$ has an
$\II_j$-product-like structure
if it satisfies the following two conditions:
\begin{description}
 \item[\textbf{C1}]
\[
 \UU\subset \DD_j^1\times\ldots\times\DD_j^{j-2}\times\NNN_{j}^+\times
\DD_j^{j+2}\times\ldots\times\DD_j^{N},
\]
where
\begin{align*}
\DD_j^k&=\left\{\left|c_k^{(j)}\right|\leq  M^{(j)}_{\el,\pm}
\de^{(1-r)/2}\right\}
\,\,\text{ for }k\in \PP^\pm_j\\
\DD_j^{j\pm 2}&\subset\left\{\left|c_{j\pm2}^{(j)}\right|\leq
M^{(j)}_{\adj,\pm} \left(C^{(j)}\de\right)^{1/2}\right\}
\end{align*}
and
\begin{equation}\label{def:Domain:Hyperbolic:sup}
\begin{split}
\NNN_{j}^+= \Big\{& \left(p_1^{(j)},q_1^{(j)},p_2^{(j)},q_2^{(j)}\right)\in
\R^4: \\ \xi(\sigma)-C^{(j)}&\de\left(\ln(1/\de)+M^{(j)}_{\hyp}\right)\leq
p_1^{(j)}\leq \xi(\sigma)
-C^{(j)}\de\left(\ln(1/\de)-M^{(j)}_{\hyp}\right),\\ &q_1^{(j)}=\sigma,\
g_j(p^{(j)}_2,q^{(j)}_2)=0, \ |p_2^{(j)}|,|q_2^{(j)}|\leq
M^{(j)}_{\hyp}\left(C^{(j)}\de\right)^{1/2} \Big\}.
\end{split}
\end{equation}
\item[\textbf{C2}]
\[
\NNN^-_{j}\times \DD_{j,-}^{j+2}\times\ldots\times\DD_{j,-}^{N}
\subset\pi_{\hyp,+}\UU,
\]
where
\begin{align*}
\DD_{j,-}^k&=\left\{\left|c_k^{(j)}\right|\leq  m^{(j)}_{\el}
\de^{(1-r)/2}\right\} \,\,\text{ for }k\in \PP_j^+\\
\DD_{j,-}^{j+ 2}&=\left\{\left|c_{j+2}^{(j)}\right|\leq  m^{(j)}_{\adj}
\left(C^{(j)}\de\right)^{1/2}\right\}
\end{align*}
and
\begin{equation}\label{def:Domain:Hyperbolic:inf}
\begin{split}
\NNN_j^-= \Big\{& \left(p_1^{(j)},q_1^{(j)},p_2^{(j)},q_2^{(j)}\right)\in \R^4:
\\\xi(\sigma) -C^{(j)}&\de\left(\ln(1/\de)+m^{(j)}_{\hyp}\right)\leq
p_1^{(j)}\leq \xi(\sigma)
-C^{(j)}\de\left(\ln(1/\de)-m^{(j)}_{\hyp}\right),\\ &q_1^{(j)}=\sigma,\
g_j(p^{(j)}_2,q^{(j)}_2)=0, \ |p_2^{(j)}|,|q_2^{(j)}|\leq
m^{(j)}_{\hyp}\left(C^{(j)}\de\right)^{1/2}\Big\}.
\end{split}
\end{equation}
\end{description}
where $\xi$ is the function introduced in Lemma \ref{lemma:SeparatrixAsGraph}.
\end{definition}
If one compares this definition to the one in \cite{GuaKal} the only difference
appears in the $p^{(j)}_1$ variable. The reason is the deviation of the
separatrix connection. 
The domains $\VV_j$ of the maps $\BB^j$ will have $\II_j$-product-like structure
as defined in Definition \ref{def:ProductLikeSet}. Thus, we need to obtain
the multi-parameter sets $\II_j$. They will be defined recursively.
In 
Theorem \ref{thm:ToyModelOrbit} we are looking for an  orbit which starts close
to
the periodic orbit $\TT_3$, thus the recursively defined multi-parameter sets
$\II_j$
will start with a set $\II_3$.

\begin{definition}\label{def:RecursivelyDefined}
Fix $\ga>0$, any constants $r,r'\in (0,1)$ satisfying $r<\ln 2/(2\ga)$ and 
$0<r'<\ln 2/\ga-2r$, $K>0$ and small $\de,
\sigma>0$.
We say that  a collection of multi-parameter sets $\{\II_j\}_{j=3,\ldots,N-2}$
defined
in \eqref{def:ProductLikeIndex} is $(\sigma,\de, K)$-recursive if for
$j=3,\ldots, N-2$ the
constants
$C^{(j)}$ satisfy
\[
\begin{split}
C^{(j)}/K\leq C^{(j+1)}\leq K C^{(j)}\\
0<m_{\hyp}^{(j+1)}\leq m_{\hyp}^{(j)}
\end{split}
\]
and all the other parameters should be strictly positive and are defined
recursively as
\[
 \begin{split}
  M_{\el,\pm}^{(j+1)}&=M_{\el,\pm}^{(j)}+ K\de^{r'}\\
  m_{\el}^{(j+1)}&=m_{\el}^{(j)}- K\de^{r'}\\
  M_{\adj,+}^{(j+1)}&= 2M_{\el,+}^{(j)}+ K\de^{r'}\\
  M_{\adj,-}^{(j+1)}&= KM_{\hyp}^{(j)}\\
  m_{\adj}^{(j+1)}&=\frac{1}{2}m_{\el}^{(j)}- K\de^{r'}\\
  M_{\hyp}^{(j+1)}&=K M_{\adj,+}^{(j)}. \\
 \end{split}
\]
\end{definition}
This definition coincides with the definition of \cite{GuaKal}. The only
difference appears in the conditions on the parameters $r$ an $r'$. The reason
is that one needs to take into account the $\OO(n^{-1})$ terms in the
Hamiltonian \eqref{def:Ham:Diagonal}.

The next theorem defines recursively the product-like sets  $\VV_j$, so that
condition \eqref{cond:ComposeMaps} is satisfied.

\begin{theorem}[Iterative Theorem]\label{theorem:iterative}
Fix a large $\ga>0$, a  small $\sigma>0$, two constants $r,r'\in (0,1)$
satisfying $r<\ln 2/(2\ga)$, $0<r'<\ln 2/\ga-2r$ and 
set $\de= e^{-\gamma N}$.  There exist strictly positive constants
$K$ and $C^{(3)}$ independent of $N$ satisfying
\begin{equation}\label{cond:GrowthOfCs}
  \de^{r}\,K^{N-3}\leq C^{(3)}\leq \de^{-r}\,K^{-(N-3)},
\end{equation}
and a multi-parameter set $\II_3$ (as defined in \eqref{def:ProductLikeIndex})
with the following property:
there exists a
$(\sigma,\de, K)$-recursive collection of multi-parameter sets
$\{\II_j\}_{j=3,\ldots,N-2}$ and $\II_j$-product-like sets
$\VV_j\subset\Sigma_{j}^\inn$ such that for each $j=3,\ldots,N-3$ we have
\[
\VV_{j+1}\subset \BB^j(\VV_j).
\]
Moreover, the time spent to reach the section $\Sigma^{\inn}_{j+1}$ can be
bounded
by
\[
 \left|T_{\BB^j}\right|\leq K\ln(1/\de)
\]
for any $(p,q,c)\in \VV_j$ and any $j=3,\ldots,N-3$.
\end{theorem}
The condition
\[
 C^{(j)}/K<C^{(j+1)}<K C^{(j)}
\]
implies
\[
K^{-(j-2)}C^{(3)}\leq  C^{(j+1)}\leq K^{j-2}C^{(3)}.
\]
Namely, at each saddle, the orbits we are studying may lie  further
from the heteroclinic orbit. Nevertheless, since $\de=e^{-\ga N}$ and
\eqref{cond:GrowthOfCs},
these constant does not grow too much. Indeed,
\begin{equation}\label{def:UpperBoundCs}
\de^r\leq C^{(j)} \leq \de^{-r},
\end{equation}
where $r>0$  is taken small.  We  use the bound
\eqref{def:UpperBoundCs} throughout the proof of Theorem
\ref{theorem:iterative}.

Theorem \ref{thm:ToyModelOrbit} is a straightforward consequence of Theorem
\ref{theorem:iterative}.

\begin{proof}[Proof of \ref{thm:ToyModelOrbit}]
It is enough to take as a initial condition $b^0$ a point in the set
$\VV_3\subset\Sigma_3^\inn$ obtained in Theorem \ref{theorem:iterative}.
Then, thanks to this theorem we know that there exists a time $T_0$
satisfying
\[
 T_0\sim N\ln(1/\de),
\]
such that the corresponding orbit satisfies that
$b(T_0)\in\VV_{N-2}\subset\Sigma_{N-2}^\inn$. Note that in this section
there are two components of $b$ with size independent of $\de$.
Nevertheless, from the proof of Theorem \ref{theorem:iterative} in
Section \ref{sec:LocalDynamicsFullToyModel} it can be easily seen that if we 
shift
the time interval $[0,T_0]$ to
$[\rr\ln(1/\de),\rr\ln(1/\de)+T_0]$, for any $\rr<\la$ independent of $n$, there
exists $\nu>0$ such that the orbit $b(t)$ satisfies the statements given in
Theorem \ref{thm:ToyModelOrbit}
\end{proof}

\section{Proof of Theorem \ref{theorem:iterative}: local and global
maps}\label{sec:LocalAndGlobalMaps}

To prove Theorem \ref{theorem:iterative} we proceed as in \cite{GuaKal} and 
split it into two inductive lemmas. The
first part analyzes the evolution of the trajectories close to the saddle $j$
and the second one the travel along the heteroclinic orbit. Thus, we study
$\BB^j$ as a composition of two maps, which we call local and global map.

We consider an intermediate section transversal to the flow
\begin{equation}\label{def:Section2Saddle}
\Sigma_j^{\out}=\left\{p_2^{(j)}=\sigma\right\}.
\end{equation}
Then, we consider the  local map 
\begin{equation}\label{def:SaddleMap}
\BB_\loc^j:\VV_j\subset\Sigma_j^{\inn}\longrightarrow \Sigma_j^{\out},
\end{equation}
and the global map
\begin{equation}\label{def:HeteroMap}
\BB_\glob^j: \UU^j\subset\Sigma_j^{\out}\longrightarrow \Sigma_{j+1}^{\inn}.
\end{equation}
Then, the map $\BB^j$ considered in
Theorem \ref{theorem:iterative} is just $\BB^j=\BB_\glob^j\circ\BB_\loc^j$. To
compose the two maps we need that  the
set $\UU^j$, introduced in \eqref{def:HeteroMap}, has
a modified product-like structure. To define its properties, we
consider the projection
\[
 \wt \pi
\left(c_-^{(j)},p_1^{(j)},q_1^{(j)},p_2^{(j)},q_2^{(j)},c_+^{(j)}\right)=
 \left(p_2^{(j)},q_2^{(j)},c_+^{(j)}\right).
\]
\begin{definition}\label{definition:ModifiedProductLike}
Fix constants $r\in (0,1)$, $\de>0$ and $\sigma>0$ and consider 
a multi-parameter set of positive constants
\[
 \wt\II_j=\left\{\wt C^{(j)},\wt m_{\el}^{(j)},\wt M_{\el,\pm}^{(j)},
\ \wt m_{\adj}^{(j)},\ \wt M_{\adj,\pm}^{(j)},\ \wt m_{\hyp}^{(j)},\
\wt M_{\hyp}^{(j)}\right\}.
\]
Then, we say that a (non-empty) set $\UU\subset\Sigma_j^\out$ has
a $\wt \II_j$-product-like structure provided it satisfies the following
two conditions:
\begin{description}
 \item[\textbf{C1}]
\[
 \UU\subset \wt\DD_j^1\times\ldots\times\wt\DD_j^{j-2}\times\wt\NNN_{j}^+\times
\wt\DD_j^{j+2}\times\ldots\times\wt\DD_j^{N}
\]
where
\begin{align*}
\wt\DD_j^k&=\left\{\left|c_k^{(j)}\right|\leq  \wt M^{(j)}_{\el,\pm}
\de^{(1-r)/2}\right\} \,\,\text{ for }k\in \PP_j^\pm\\
\wt\DD_j^{j\pm 2}&\subset\left\{\left|c_{j\pm2}^{(j)}\right|\leq
\wt M^{(j)}_{\adj,\pm} \left(\wt C^{(j)}\de\right)^{1/2}\right\},
\end{align*}
and
\[
\begin{split}
\wt\NNN_{j}^+= \Big\{& (p_1^{(j)},q_1^{(j)},p_2^{(j)},q_2^{(j)})\in \R^4:
\left|p_1^{(j)}\right|,\left|q_1^{(j)}\right|\leq \wt M_\hyp^{(j)}\left(\wt
C^{(j)}\de\right)^{1/2},\\
&p_2^{(j)}=\sigma, \xi(\sigma) -\wt C^{(j)}\,\de\,\left(\ln(1/\de)+\wt
M_{\hyp}^{(j)}\right)\leq
q_2^{(j)}\leq  \xi(\sigma)-\wt C^{(j)}\,\de\,\left(\ln(1/\de)-\wt
M_{\hyp}^{(j)}\right) \Big\},
\end{split}
\]
\item[\textbf{C2}]
\[
 \{\sigma\}\times\left[ \xi(\sigma)-\wt C^{(j)}\,\de\,\left(\ln(1/\de)+\wt
m_{\hyp}^{(j)}\right), \xi(\sigma) -\wt
C^{(j)}\,\de\,\left(\ln(1/\de)-\wt m_{\hyp}^{(j)}\right)\right]\times
\wt\DD_{j,-}^{j+2}\times\ldots\times\wt\DD_{j,-}^{N} \subset\wt\pi(\UU)
\]
where
\begin{align*}
\wt\DD_{j,-}^k&=\left\{\left|c_k^{(j)}\right|\leq  \wt m^{(j)}_{\el}
\de^{(1-r)/2}\right\} \,\,\text{ for }k\in \PP_j^+\\
\wt\DD_{j,-}^{j+ 2}&=\left\{\left|c_{j+2}^{(j)}\right|\leq
\wt m^{(j)}_{\adj} \left(\wt C^{(j)}\de\right)^{1/2}\right\}.
\end{align*}
\end{description}
\end{definition}
With this definition, we can state the following two lemmas.
Combining these two lemmas we deduce Theorem \ref{theorem:iterative}.

\begin{lemma}\label{lemma:iterative:saddle}
Let $\gamma,\sigma,r,r',\delta$ be as in Theorem \ref{theorem:iterative}. Fix
any natural $j$ with $3\leq j\leq N-3$
and consider any parameter set $\II_j$ with $M_\hyp^{(j)}\geq 1$ and
a $\II_j$-product-like set $\VV_j\subset\Sigma_j^\inn$. Then, for $N$ big
enough, there exists:
\begin{itemize}
\item A constant $K>0$ independent of $N$ and $j$ but which might depend on
$\sigma$.
\item A parameter set $\wt\II_j$ whose constants satisfy
\[
 \begin{split}
  C^{(j)}/2\leq \wt C^{(j)}\leq 2 C^{(j)}\\
0<\wt m_{\hyp}^{(j)}\leq m_{\hyp}^{(j)}
 \end{split}
\]
and
\[
 \begin{split}
\wt M_\hyp^{(j)}&=K\\
\wt M_{\el,\pm}^{(j)}&=M_{\el,\pm}^{(j)}+K\de^{r'}\\
\wt m_{\el}^{(j)}& = m_{\el}^{(j)}-K\de^{r'}\\
\wt M_{\adj,\pm}^{(j)}&=M_{\adj,\pm}^{(j)}(1+4\sigma)\\
\wt m_{\adj}^{(j)}&=m_{\adj}^{(j)}\ (1-4\sigma),
 \end{split}
\]
\item A  $\wt\II_j$-product-like set $\UU_j$ for which  the map $\BB_\loc^j$
satisfies
\begin{equation}\label{cond:ComposeMaps:Localmap}
\UU_j\subset\BB_\loc^j\left(\VV_j\right).
\end{equation}
\end{itemize}
Moreover, the time to reach the section $\Sigma^{\out}_{j}$ can be bounded
as
\[
 \left|T_{\BB_\loc^j}\right|\leq K\ln(1/\de).
\]
\end{lemma}

The proof of this lemma follows the same approach than the proof of Lemma 4.7 in
\cite{GuaKal}. First, in Section \ref{sec:HyperbolicToyModel}, we set the
elliptic modes $c$ to zero,
and we study the saddle map  associated to the corresponding
system. We call  this system \emph{Hyperbolic Toy Model}.  It  has two degrees
of
freedom. As happens in \cite{GuaKal},  the saddle is resonant since both stable
eigenvalues coincide. We used the ideas developed in \cite{GuaKal} to overcome
this problem. They are based on techniques developed by Shilnikov
\cite{Shilnikov67}. Then, in Section \ref{sec:LocalDynamicsFullToyModel} we use
the results obtained for
the Hyperbolic Toy Model to deal with the full system and prove Lemma
\ref{lemma:iterative:saddle}. 

Now we state the iterative lemma for the global maps $\BB^j_\glob$.
\begin{lemma}\label{lemma:iterative:hetero}
Let $\gamma,\sigma,r,r',\delta$ be as in Theorem \ref{theorem:iterative}. Fix
any natural $j$ with $3\leq j\leq N-3$
and consider any parameter set $\wt \II_j$ and a $\wt \II_j$-product-like
set $\UU_j\subset\Sigma_j^\out$. Then, for $N$ large enough, there exists:
\begin{itemize}
\item A constant $\wt K$
depending on $\sigma$, but independent of $N$ and $j$.
\item A parameter set $\II_{j+1}$ whose constants
satisfy
\[
 \begin{split}
  \wt C^{(j)}/\wt K\leq C^{(j+1)}\leq \wt K\wt C^{(j)}\\
0<m_{\hyp}^{(j+1)}\leq \wt m_\hyp^{(j)}
 \end{split}
\]
and
\[ \begin{split}
M_{\el,-}^{(j+1)}&=\max\left\{\wt M_{\el,-}^{(j)}+\wt K\de^{r'},\wt
K\wt M_{\adj,-}^{(j)}\right\}\\
M_{\el,+}^{(j+1)}&=\wt M_{\el,+}^{(j)}+\wt K\de^{r'}\\
m_{\el}^{(j+1)}&=\wt m_{\el}^{(j)}-\wt K\de^{r'}\\
M_{\adj,+}^{(j+1)}&=\wt M_{\el,+}^{(j)}+\wt K\de^{r'}\\
M_{\adj,-}^{(j+1)}&=\wt K\wt M_{\hyp}^{(j)}\\
m_{\adj}^{(j+1)}&=\wt m_{\el}^{(j)}+\wt K\de^{r'}\\
M_{\hyp}^{(j+1)}&=\max\left\{\wt K \wt M_{\adj,+}^{(j)},\wt K\right\}
 \end{split}
\]
\item A  $\II_{j+1}$-product-like set
$\VV_{j+1}\subset\Sigma_{j+1}^\inn$  for which the map $\BB_\glob^j$ satisfies
\begin{equation}\label{cond:ComposeMaps:Heteromap}
\VV_{j+1}\subset\BB_\glob^j\left(\UU_j\right).
\end{equation}
\end{itemize}
Moreover, the time spent to reach the section $\Sigma^{\inn}_{j+1}$
can be bounded as
\[
 \left|T_{\BB_\glob^j}\right|\leq \wt K.
\]
\end{lemma}
The proofs of this lemma is postponed to Section \ref{sec:ProofHeteroMap}.

Now it only remains to deduce from Lemmas \ref{lemma:iterative:saddle}
and \ref{lemma:iterative:hetero}  the Iterative Theorem \ref{theorem:iterative}.
\begin{proof}[Proof of Theorem \ref{theorem:iterative}]
We choose the multi-index $\II_3$ so that we can apply iteratively the
Lemmas \ref{lemma:iterative:saddle} and \ref{lemma:iterative:hetero}. Indeed,
from the recursive formulas in Lemma \ref{lemma:iterative:saddle} and
\ref{lemma:iterative:hetero} it is clear that it is enough to choose a parameter
set $\II_3$ satisfying
\[
 1<M_{\el,+}^{(3)}\ll M_{\adj,+}^{(3)}\ll
M_{\hyp}^{(3)}\ll M_{\adj,-}^{(3)}\ll M_{\el,-}^{(3)}
\]
and
\[
 0<m_\el^{(3)}<3 m_\adj^{(3)}.
\]
From the choice of the constants in $\II_3$ and
the recursion formulas in Lemmas \ref{lemma:iterative:saddle}
and \ref{lemma:iterative:hetero}, we have that $M_\hyp^{(j)}\geq 1$ for any
$j=3,\ldots N-2$.
This fact
along with conditions \eqref{cond:ComposeMaps:Localmap} and
\eqref{cond:ComposeMaps:Heteromap}, allow us to apply
Lemmas \ref{lemma:iterative:saddle} and \ref{lemma:iterative:hetero}
iteratively so that we obtain the
$(\de,\sigma,K)$-recursive collection
of multi-parameter sets $\{\II_j\}_{j=3,\ldots,N-2}$ and
the $\II_j$-product-like sets $\VV_j\subset\Sigma_j^\inn$.
In particular, note that the recursion formulas stated in
Theorem \ref{theorem:iterative} can be easily deduced from
the recursion formulas given in Lemmas \ref{lemma:iterative:saddle}
and \ref{lemma:iterative:hetero} and the choice of $\II_3$.

Finally, we bound the time
\[
  \left|T_{\BB^j}\right|\leq  \left|T_{\BB_\loc^j}\right|+
  \left|T_{\BB_\glob^j}\right|\leq K\ln(1/\de)+\wt K.
\]
This completes the proof of Theorem \ref{theorem:iterative}.
\end{proof}

\subsection{Straightening the heteroclinic connections}
To prove Lemmas \ref{lemma:iterative:saddle} and \ref{lemma:iterative:hetero}
the first step is to straighten the heteroclinic connections which connect with
the future and past saddles. They have been analyzed in Lemma
\ref{lemma:SeparatrixAsGraph}.

We perform the change of coordinates $(P_1,Q_1,P_2,Q_2)=\Xi (p_1,q_1,p_2,q_2)$
defined as 
\begin{equation}\label{def:Change:Heteroclinic}
\begin{split}
P_1&=p_1-\xi(q_1)\\
Q_1&=q_1\\
P_2&=p_2\\
Q_2&=q_2-\xi(p_2),
\end{split} 
\end{equation}
 which straightens the heteroclinic connections. This change is symplectic. 

\begin{lemma}
If one performs the change of coordinates \eqref{def:Change:Heteroclinic}, one
obtains a new Hamiltonian system of the form
\begin{equation}\label{def:Ham:AfterSeparatrix}
   H^{(j)}(P,Q,c) =   H^{(j)}_2(P,Q,c)+   H^{(j)}_4(P,Q,c)
\end{equation}
with
\[
   H^{(j)}_2(P,Q,c)=a_n\sum_{k\in
\PP_j}|c_k|^2+\la_n\left(P_1Q_1+P_2Q_2\right)
\]
and
\[
   H^{(j)}_4(P,Q,c)=   H^{(j)}_{\hyp}\left(P,Q\right)+  H^{(j)}_{\el}(c)+ 
H^{(j)}_{\mix}\left(P,Q,c\right)
\]
where
\begin{align}
 H^{(j)}_\hyp(P,Q)=&-2P_1Q_1 (P_1^2+Q_1^2-P_1Q_1)-2P_2Q_2
(P_2^2+Q_2^2-P_2Q_2)\notag\\
&+\sum_{k,\ell=0}^2\nu_{k\ell}
P_1^kQ_1^{2-k}P_2^\ell
Q_2^{2-\ell}+\OO\left(\frac{1}{n}(P_1+Q_1)^2(P_2+Q_2)^2\right)\notag\\
&+\OO\left(\frac{1}{n}(P_1Q_1(P_1+Q_1)^2+P_2Q_2(P_2+Q_2)^2)\right)\notag\\
  H^{(j)}_{\el}\left(c\right)=&-\frac{1}{4}\sum_{k\in\PP_j}|c_k|^4-\frac{1}{4}
\left(\sum_{k\in\PP_j}|c_k|^2\right)^2 \notag\\
&+\sum_{k\in
\PP_j\setminus\{j+2\}}\Re\big(c_k^2\ol{c_{k-1}}^2\big)+\OO\left(\frac{1}{n}\sum_
{k,k'\in \PP_j}|c_k|^2|c_{k'}|^2\right)\notag\\
  H^{(j)}_{\mix}(p,q,c)=&-{\sqrt{3}}\sum_{k\in
\PP_j}|c_k|^2\left(Q_1P_1+Q_2P_2\right)\notag\\
&+\frac{2\sqrt 3}{3}\Re\left(\left(\omega_0 P_1+\ol\omega_0 Q_1\right)^2 c_{j-
2}^2\right)+\frac{2\sqrt 3}{3}\Re\left(\left(\omega_0 P_2+\ol\omega_0
Q_2\right)^2 c_{j+ 2}^2\right)\notag\\
&+\OO\left(\frac{1}{n}\sum_{k\in \PP_j}|c_k|^2(P_1+Q_1+P_2+Q_2)^2\right)\notag
\end{align}
\end{lemma}
To fix notation, we define the vector field associated to Hamiltonian
\eqref{def:Ham:AfterSeparatrix},
\begin{eqnarray}\label{def:VF:AfterSeparatrix}
 \begin{aligned}
\dot P_1
& = & \la_n P_1+\ZZZ_{\hyp, P_1}+\ZZZ_{\mix, P_1}
& = & \la_n P_1+\partial_{Q_1}  H^{(j)}_{\hyp}+\partial_{Q_1}  H^{(j)}_{\mix}\\
\dot Q_1
& = & -\la_n Q_1+\ZZZ_{\hyp, Q_1}+\ZZZ_{\mix, Q_1}
& = & -\la_n Q_1-\partial_{P_1}  H^{(j)}_{\hyp}-\partial_{P_1}  H^{(j)}_{\mix}\\
\dot P_2
& = & \la_n P_2+\ZZZ_{\hyp, P_2}+\ZZZ_{\mix, P_2}
& = & \la_n P_2+\partial_{Q_2}  H^{(j)}_{\hyp}+\partial_{Q_2}  H^{(j)}_{\mix}\\
\dot Q_2
& = & -\la_n Q_2+\ZZZ_{\hyp, Q_2}+\ZZZ_{\mix, Q_2}
& = & -\la_n Q_2-\partial_{P_2}  H^{(j)}_{\hyp}-\partial_{P_2}  H^{(j)}_{\mix}\\
\dot c_k
& = & 2\ii a_nc_k  + \ZZZ_{\el, c_k} + \ZZZ_{\mix, c_k}
& = & 2\ii a_nc_k - 2\ii \partial_{c_k}  H^{(j)}_{\el}-2\ii \partial_{c_k} 
H^{(j)}_{\mix}.
\end{aligned}
\end{eqnarray}

\section{The local dynamics of the hyperbolic toy
model}\label{sec:HyperbolicToyModel}
If we set to zero the elliptic modes in the Hamiltonian obtained in Lemma
\ref{lemma:SeparatrixAsGraph}, we obtain the Hamiltonian
\begin{equation}\label{def:HypToyModel}
 H(P,Q)= \la_n (P_1Q_1+P_2Q_2)+ H^{(j)}_\hyp(P,Q)
\end{equation}
Therefore, the associated vector field is
\begin{equation}\label{def:HypToyModelVectorField}
\begin{split}
\dot P_1&=\la_n P_1+\ZZZ_{\hyp, P_1}\\
\dot Q_1&=-\la_n Q_1+\ZZZ_{\hyp, Q_1}\\
\dot P_2&=\la_n P_2+\ZZZ_{\hyp, P_2}\\
\dot Q_2&=-\la_n Q_2+\ZZZ_{\hyp, Q_2},
\end{split}
\end{equation}
where $\ZZZ_{\hyp, P_i}=\pa_{Q_i} H^{(j)}_\hyp(P,Q)$ and  $\ZZZ_{\hyp,
Q_i}=-\pa_{P_i}H^{(j)}_\hyp(P,Q)$.

As we have explained, $(P,Q)=(0,0)$ is a hyperbolic critical point for the
hyperbolic toy model. We want to study the local dynamics. The first step is to
perform a $\CCC^k$ resonant normal form to remove the nonresonant terms, as is
done in \cite{GuaKal}. Note that here we encounter the same type of resonace. We
use a result by Bronstein and Kopanskii  \cite{BronsteinK92}, see Theorem 6 of
\cite{GuaKal},  which implies the following.

\begin{lemma}\label{lemma:HypToyModel:NormalForm}
There exists a $\CCC^2$ change of coordinates
\[
 (P_1,Q_1,P_2,Q_2)=\Psi_\hyp(x_1,y_1,x_2,y_2)=(x_1,y_1,x_2,y_2)+
 \wt\Psi_\hyp(x_1
,y_1,x_2,y_2)
\]
which transforms the vector field \eqref{def:HypToyModelVectorField} into the
vector
field
\begin{equation}\label{def:HypHam:Transformed}
\XX_\hyp(z)=D z+R_\hyp,
\end{equation}
where $z$ denotes $z=(x_1,y_1,x_2,y_2)$, $D$ is the diagonal matrix
$D=\mathrm{diag}(\la_n,-\la_n,\la_n,-\la_n)$ and $R_\hyp$ is a
polynomial, which only contains resonant monomials. It can be split as
\begin{equation}\label{def:HypHam:Hyp}
R_\hyp=R_\hyp^0+R_\hyp^1,
\end{equation}
where $R_\hyp^0$ is the first order, which is given by
\[
R_\hyp^0(z)=\left(\begin{array}{c}R_{\hyp,x_1}^0(z)\\R_{\hyp,y_1}^0(z)\\R_{\hyp,
x_2}^0(z)\\R_{\hyp,y_2}^0(z)\end{array}\right)=\left(\begin{array}{c}
4x_1^2y_1+2\nu_{02}y_1x_2^2+\nu_{11}x_1x_2y_2 \\
-4x_1y_1^2-2\nu_{20}x_1y_2^2-\nu_{11}y_1x_2y_2\\4y_2x_2^2+2\nu_{20}
x_1^2y_2
+\nu_{11}x_1y_1x_2\\-4x_2y_2^2-\nu_{02}y_1^2x_2-\nu_{11}
x_1y_1y_2\end{array}\right),
\]
and $R_\hyp^1$ is the remainder and satisfies
 $R_{\hyp,x_i}^1=\OO\left(x^3y^2\right)$  and
 $R_{\hyp,y_i}^1=\OO\left(x^2y^3\right)$.

Moreover, the function $\wt\Psi_\hyp=( \wt\Psi_{\hyp,x_1}, \wt\Psi_{\hyp,y_1},
\wt\Psi_{\hyp,x_2}, \wt\Psi_{\hyp,y_2})$ satisfies
\[
\begin{split}
 \wt\Psi_{\hyp,x_1}(z)&=\OO\left(x_1^3,x_1y_1,x_1(x_2^2+y_2^2),
y_1y_2(x_2+y_2)\right)\\
\wt\Psi_{\hyp,y_1}(z)&=\OO\left(y_1^3,x_1y_1,y_1(x_2^2+y_2^2),
x_1x_2(x_2+y_2)\right)\\
 \wt\Psi_{\hyp,x_2}(z)&=\OO\left(x_2^3,x_2y_2,x_2(x_1^2+y_1^2),
y_1y_2(x_1+y_1)\right)\\
\wt\Psi_{\hyp,y_2}(z)&=\OO\left(y_2^3,x_2y_2,y_2(x_1^2+y_1^2),
x_1x_2(x_1+y_1)\right).
\end{split}
\]
\end{lemma}

\begin{remark}
 All functions involved in this lemma, and also all functions involved in the
forthcoming sections depend on the parameter $n$. We omit this dependence to
simplify the notation. Note that when we use the notation $f=\OO(g)$ we mean
that there exists a constant $C>0$ independent of $n$, $\de$ and $\sigma$ such
that $|f|\leq C|g|$.
\end{remark}

We analyze the local dynamics for the vector field
\eqref{def:HypHam:Transformed} and then we will deduce the dynamics in the
original variables.  Note that after normal form, the vector field
\eqref{def:HypHam:Transformed} is of the same form as the corresponding vector
field in \cite{GuaKal}. Therefore, we can use the results from that paper. As we
have said in Section \ref{sec:IterativeTheorem}, we follow the notation of
multiparameter sets from that paper.

In Section \ref{sec:IterativeTheorem}, we have considered the sets
$\NNN_j^-\subset \NNN_j^+$ to define the almost product structure. Since in this
section we have set the elliptic modes to zero, that is, $c=0$, we consider a
set $\NNN_j'$ satisfying
\[
 \NNN_j^-\cap\{c=0\}\subset\NNN_j'\subset \NNN_j^+\cap\{c=0\}
\]
Now, we need to express it in the new coordinates $(x,y)$. We denote the inverse
of the change $\Psi_\hyp$, obtained in Lemma \ref{lemma:HypToyModel:NormalForm},
by
$\Upsilon=\Id+\wt\Upsilon=\Id+\left(\wt\Upsilon_{x_1},\wt\Upsilon_{y_1},
\wt\Upsilon_{x_2},\wt\Upsilon_{y_2}\right)$ 
and we define
\[
\wh C^{(j)} =\wt C^{(j)}\left(1+\wt\Upsilon_{x_1}(0,\sigma,0,0)\right).
\]
Note that $\wh C^{(j)} =\wt C^{(j)}(1+\OO(\sigma))$. We also define 
$f_1(\sigma)=\Upsilon_{y_1}(0,\sigma,0,0)$. 
This correspond to the first order of shift in the Poincar\'e section due to the
normal form. That is, the section $y_1=f_1(\sigma)$ approximates the section
$\Upsilon(\Sigma_j^{\inn})$ (recall that the change
\eqref{def:Change:Heteroclinic} has not moved the transversal section). We
define the set of points in the normal form variables $(x,y)$ whose dynamics we
want to analyze by
\begin{equation}\label{def:Hyp:ModifiedDomain}
\begin{split}
\wh\NNN_j=\Big\{& |x_1+\wh
C^{(j)}\,\de\,(\ln(1/\de)|\le \wh C^{(j)}\,\de\,K_\sigma, \quad
\left|x_2-x_2^\ast\right|\leq 2\, M_{\hyp}^{(j)}\frac{\left(\wh
C^{(j)}\de\right)^{1/2}}{\ln(1/\de)},\\
 & |y_1- f_1(\sigma)|\le
K_\sigma\wh C^{(j)}\de\ln(1/\de),\quad \qquad |y_2|\leq
2\,M_{\hyp}^{(j)}\left(\wh
C^{(j)}\de\right)^{1/2}\Big\},
\end{split}
\end{equation}
The constant $x_2^*$ will be choosen later analogously to \cite{GuaKal}. The
choice will allow us to obtain a cancellation which avoids deviation from the
invariant manifolds.

 The outcoming section gets also slightly modified by the normal form. To this
end, we need to define the function $f_2(\sigma)$ as
$f_2(\sigma)= \Upsilon_{x_2}(0,0,\sigma,0)$.
As will be seen, the coordinate $x_2$ behaves almost linearly as $x_2\sim x_2^0
e^{\la \tau}$ (recall that we have rescaled time by \eqref{def:timerescaling} so
now the time variable is $\tau$). Therefore, the time needed to reach the
section $x_2=f_2(\sigma)$ is given approximately by
\begin{equation}\label{def:FixedTimeSaddle}
 T_j(x_2^0)=\frac{1}{\la}\ln \left(\frac{f_2(\sigma)}{x_2^0}\right).
\end{equation}

In order to analyze the action of $\Phi_\tau^\hyp $, i.e. the flow associated to
\eqref{def:HypHam:Transformed}, on points $\widehat \NNN_j$ we proceed as in
\cite{GuaKal}. We choose $x_2^\ast$ as the unique
positive solution of
\begin{equation}\label{def:ChoiceEtaHat}
\left(x_2^\ast\right)^2 T_j(x_2^\ast)=\frac{\wh
C^{(j)}\,\de\,\ln(1/\de)}{2\,\nu_{02}\,f_1(\sigma)}.
\end{equation}
 We perform the change of coordinates
$x_i=e^{\la \tau}u_i,\,\,y_i=e^{-\la \tau}v_i$
and thus we obtain the integral equations
\begin{equation}\label{def:Hyp:FixedPoint}
\begin{split}
u_i&=x_i^0+\int_0^T e^{-\la \tau}R_{\hyp,x_i}\left(u e^{\la \tau},v
e^{-\la \tau}\right)d\tau\\
v_i&=y_i^0+\int_0^T e^{\la \tau}R_{\hyp,y_i}\left(u e^{\la \tau},v
e^{-\la \tau}\right)d\tau.
\end{split}
\end{equation}

In the linear case $u_i$'s and $v_i$'s are constant. We use these variables to
find a fixed point argument. We define the contractive operator in
two steps. This approach is inspired by Shilnikov \cite{Shilnikov67}. First we
define an  auxiliary (non-contractive) operator 
\[
\FF_\hyp=(\FF_{\hyp,u_1},\FF_{\hyp,v_1},\FF_{\hyp,u_2},\FF_{\hyp,v_2})
\]
 as
\begin{equation}\label{def:Hyp:Operator1}
\begin{split}
 \FF_{\hyp,u_i}(u,v)&=x_i^0+\int_0^T e^{-\la \tau}R_{\hyp,x_i}\left(u
e^{\la \tau},v e^{-\la \tau}\right)d\tau\\
 \FF_{\hyp,v_i}(u,v)&= y_i^0+\int_0^T e^{\la \tau}R_{\hyp,y_i}\left(u
e^{\la \tau},v e^{-\la \tau}\right)d\tau.
\end{split}
\end{equation}
As happens in \cite{GuaKal}, for the $u_1$ and $v_2$ components the main terms
are not given by the initial condition but by the integral terms. In other
words,  the dynamics near the saddle is not well approximated by the linearized 
dynamics and the operator is not contractive.
Following ideas from Shilnikov \cite{Shilnikov67}, to have a contractive
operator, we modify slightly two of the components of $\FF_\hyp$ by considering
\[
\wt\FF_\hyp=(\wt\FF_{\hyp,u_1},\wt\FF_{\hyp,v_1},\wt\FF_{\hyp,u_2},
\wt\FF_{\hyp, v_2})
\]
as
\[
\begin{split}
\wt
\FF_{\hyp,u_1}(u,v)&=\FF_{\hyp,u_1}(u_1,\FF_{\hyp,v_1}(u,
v),\FF_{\hyp,u_2}(u,v),v_2)\\
\wt \FF_{\hyp,v_1}(u,v)&=\FF_{\hyp,v_1}(u,v)\\
\wt \FF_{\hyp,u_2}(u,v)&=\FF_{\hyp,u_2}(u,v)\\
\wt
\FF_{\hyp,v_2}(u,v)&=\FF_{\hyp,v_2}(u_1,\FF_{\hyp,v_1}(u,v),\FF_{\hyp,u_2}(u,v),
v_2)
\end{split}
\]
The fixed points of these operators are exactly the same as the fixed
points of  $\FF_\hyp$ and, then, are  solutions of equation
\eqref{def:Hyp:FixedPoint}.

The operator $\wt\FF_\hyp$ is contractive in a suitable
Banach space. We define the following 
weighted norms. To fix notation, we denote by $\|\cdot\|_\infty$
the standard supremum norm. Then define
\begin{equation}\label{def:Hyp:Norms}
\begin{split}
\|h\|_{\hyp, u_1}=&\sup_{\tau\in [0,T_j]}\left|\left(-\wh
C^{(j)}\de\ln(1/\de)+2\nu_{02}f_1(\sigma)\left(x_2^\ast\right)^2 \tau+\wh
C^{(j)}\de\right)^{-1}h(\tau)\right| \\
\|h\|_{\hyp, v_1}=&\,f_1(\sigma)\iii\|h\|_\infty \\
\|h\|_{\hyp, u_2}=&\left(x_2^\ast\right)\iii\|h\|_\infty \\
\|h\|_{\hyp, v_2}=&\left(\left(y_1^0\right)^2x_2^0 T_j\right)\iii\|h\|_\infty
\end{split}
\end{equation}
and the norm
\begin{equation}\label{def:Hyp:FullNorm}
\|(u,v)\|_\ast=\sup_{i=1,2}\left\{\|u_i\|_{\hyp,u_i},\|v_i\|_{\hyp,v_i}\right\}.
\end{equation}
This gives rise to the following Banach space
\[
\YY_\hyp=\left\{(u,v):[0,T]\rightarrow \R^4;  \|(u,v)\|_\ast<\infty\right\}.
\]

The contractivity of $\wt\FF_\hyp$ is a consequence of the following two
auxiliary propositions, whose proofs are given in \cite{GuaKal}.

\begin{proposition}\label{lemma:Hyp:FirstIteration}
Assume \eqref{def:ChoiceEtaHat}, then there exists a constant
$\kk_0>0$ independent of $\sigma$, $\de$ and $j$ such that for $\de$ and
$\sigma$ small enough, the operator $\wt \FF_\hyp$ satisfies
\[
 \|\wt \FF (0)\|_\ast\leq \kk_0.
\]
\end{proposition}

\begin{proposition}\label{lemma:Hyp:Contractive}
Consider $w, w'\in B(2\kk_0)\subset\YY_\hyp$ and let us assume
\eqref{def:ChoiceEtaHat}, then taking $\de\ll\sigma$, the operator $\wt
\FF_\hyp$ satisfies
\[
 \|\wt \FF_\hyp(w)-\wt \FF_\hyp(w')\|_\ast\leq K_\sigma\left(\wh
C^{(j)}\de\right)^{1/2}\ln^2(1/\de) \|w-w'\|_\ast.
\]
\end{proposition}

These two propositions show that $\wt\FF_\hyp$ is contractive
from $B(2\kk_0)\subset\YY_\hyp$ to itself. Therefore, it has a unique
fixed point in $B(2\kk_0)\subset\YY_\hyp$ which we denote by $w^*$.
This fixed point argument gives precise estimates for the local dynamics of the
Hyperbolic Toy Model \eqref{def:HypToyModel}. We use these estimates in order to
study the behaviour of the full Toy Model \eqref{def:Ham:AfterSeparatrix} in
Section \ref{sec:LocalDynamicsFullToyModel}.

\section{The local dynamics for the toy
model}\label{sec:LocalDynamicsFullToyModel}

We study the dynamics of the local map and we prove Lemma
\ref{lemma:iterative:saddle}. We rely on the previous analysis of the hyperbolic
toy model \eqref{def:HypToyModel}  done in Section \ref{sec:IterativeTheorem}.
In this section, we consider the Hamiltonian \eqref{def:Ham:AfterSeparatrix},
that is, we incorporate the elliptic modes.

Our goal is to study the map $\BB_\loc^j$.  We adapt its study from
\cite{GuaKal}. As in \cite{GuaKal}, the key point of this study
is that  the elliptic modes  remain almost constant through the saddle map,
which implies that they  
do not make much influence on the hyperbolic ones. In comparison to
\cite{GuaKal},  the vector field \eqref{def:VF:AfterSeparatrix} has some extra
terms. Even if they are small, one needs to treat them carefully since the local
map involves a rather long time.

As a first step we perform the change obtained  in Lemma
\ref{lemma:HypToyModel:NormalForm} to the vector field
\eqref{def:VF:AfterSeparatrix}. 

\begin{lemma}\label{lemma:FullModel:NormalForm}
Let $\Psi_\hyp$ be the map defined in Lemma
\ref{lemma:HypToyModel:NormalForm}. Then, if one performs the  change of
coordinates
\begin{equation}\label{def:Full:NFChange}
 (P_1,Q_1,P_2,Q_2,c)=\left(\Psi_\hyp(x_1,y_1,x_2,y_2),c\right),
\end{equation}
to the vector field \eqref{def:VF:AfterSeparatrix}, obtains a vector field of
the form
\[
\begin{split}
\dot z&= Dz+R_{\hyp}(z)+R_{\mix,z}(z,c)\\
\dot c_k&= 2\ii a_nc_k+\ZZZ_{\el,c_k}(c)+R_{\mix,c}(z,c),
\end{split}
\]
where $z$ denotes $z=(z_1,z_2)=(x_1,y_1,x_2,y_2)$,
$D=\mathrm{diag}(\la_n ,-\la_n ,\la_n ,-\la_n)$,  $R_{\hyp}$ has been
given in Lemma \ref{lemma:HypToyModel:NormalForm}, $\ZZZ_{\el,c_k}$ is defined
in
\eqref{def:VF:AfterSeparatrix},
and $R_{\mix,z}$ and $R_{\mix,c_k}$ are defined as
\begin{align}
R_{\mix,x_1}&=A_{x_1}(z)\ol{c_{j-2}}^2+\ol{A_{x_1}(z)}{c_{j-2}}^2-\sqrt{3}\sum_{
k\in\PP_j}|c_k|^2\Psi_{x_1}(z)+\frac{1}{n}D_{x_1}(z,c)\notag\\
R_{\mix,y_1}&=A_{y_1}(z)\ol{c_{j-2}}^2+\ol{A_{y_1}(z)}{c_{j-2}}^2+\sqrt{3}\sum_{
k\in\PP_j}|c_k|^2\Psi_{y_1}(z)+\frac{1}{n}D_{y_1}(z,c)\notag\\
R_{\mix,x_2}&=A_{x_2}(z)\ol{c_{j+2}}^2+\ol{A_{x_2}(z)}{c_{j+2}}^2-\sqrt{3}\sum_{
k\in\PP_j}|c_k|^2\Psi_{x_2}(z)+\frac{1}{n}D_{x_2}(z,c)\notag\\
R_{\mix,y_2}&=A_{y_2}(z)\ol{c_{j+2}}^2+\ol{A_{y_2}(z)}{c_{j+2}}^2+\sqrt{3}\sum_{
k\in\PP_j}|c_k|^2\Psi_{y_2}(z)+\frac{1}{n}D_{y_2}(z,c)\notag\\
 R_{\mix,c_k}&=\ii c_k P(z)+\frac{1}{n}D_{c_k}(z,c)\,\,\,\text{ for }k\neq j\pm
2\notag\\
R_{\mix,c_{j\pm 2}}&=\ii c_{j\pm 2} P(z)-\ii \ol c_{j\pm
2}Q_\pm(z)+\frac{1}{n}D_{c_{j\pm 2}}(z,c)\notag
\end{align}
where $\Psi_{\hyp,z}$ are the functions defined in  Lemma
 \ref{lemma:HypToyModel:NormalForm},  $A_{z}$ satisfy
\[
 A_{x_i}=\OO(x_i, y_i)\,\,\,\text{ and }\,\,\,A_{y_i}=\OO(x_i, y_i),
\]
the functions $D_z$ satisfy
\[
 D_z=\OO\left(\sum_{k\in\PP_j}|c_k|^2(x_i+y_i)\right),\,\,D_{c_k}(z,
c)=\OO\left(\sum_{k\in\PP_j}|c_k|(x_i+y_i)^2\right)
\]
and $P$ and $Q_\pm$ satisfy
\[
P(z)=\OO\left(xy\right),\,\,\,Q_-(z)=\OO\left(x_1,
y_1\right)  \,\,\,\text{ and }\,\,\,Q_+(z)=\OO\left(x_2,y_2\right).
\]
\end{lemma}
The proof
of this lemma is straightforward taking into account the form of
the vector field \eqref{def:VF:AfterSeparatrix} and the properties
of $\Psi_\hyp$ given in Lemma \ref{lemma:HypToyModel:NormalForm}.

As happens in \cite{GuaKal},  there is a rather strong interaction between the
hyperbolic and the elliptic modes due to the terms $
R_{\mix,x_i}$ and $ R_{\mix,y_i}$. As explained in \cite{GuaKal}, the importance
of these terms can be seen as follows. The manifold $\{x=0,y=0\}$ is normally
hyperbolic \cite{Fenichel74, Fenichel77, HirschPS77}  for the linear truncation
of the vector field obtained in Lemma \ref{lemma:FullModel:NormalForm} and its
stable and unstable manifolds  are defined as $\{x=0\}$ and $\{y=0\}$. For the
full vector field, the manifold $\{x=0,y=0\}$ is persistent. Moreover it is
still normally hyperbolic thanks to \cite{Fenichel74, Fenichel77, HirschPS77}.
Nevertheless, the associated invariant manifolds deviate from $\{x=0\}$ and
$\{y=0\}$ due to the terms $R_{\mix,x_i}$ and $ R_{\mix,y_i}$. To overcome this
problem, we slightly modify  the change \eqref{def:Full:NFChange} to straighten
these invariant manifolds completely.

\begin{lemma}\label{lemma:StraightenInvManifolds}
There exist a change of coordinates of the form
\begin{equation}\label{def:Change:FullSystem:Hyp}
(P_1,Q_1,P_2,Q_2,c)=\left(\Psi(x_1,y_1,x_2,y_2,c),c\right)=\left(x_1,y_1,x_2,y_2
,c\right)+\left(\wt\Psi (x_1,y_1,x_2,y_2,c),0\right)
\end{equation}
which transforms the vector field \eqref{def:VF:AfterSeparatrix} into a
vector field of the form
\begin{equation}\label{def:VF:Straight}
\begin{split}
\dot z&= Dz+R_{\hyp}(z)+\wt R_{\mix,z}(z,c)\\
\dot c_k&= 2\ii  a_nc_k+\ZZZ_{\el,c_k}(c)+ \wt R_{\mix,c_k}(z,c),
\end{split}
\end{equation}
where  $R_\hyp$ and $\ZZZ_\el$ are the functions defined in
\eqref{def:HypHam:Hyp} and \eqref{def:VF:AfterSeparatrix} respectively, and
\begin{align}
\wt
R_{\mix,x_1}&=B_{x_1}(z,c)\ol{c_{j-2}}^2+\ol{B_{x_1}(z,c)}{c_{j-2}}^2+\sum_{
k\in\PP_j}|c_k|^2C_{x_1}(z,c)+\frac{1}{n}F_{x_1}(z,c)\notag
\\
\wt
R_{\mix,y_1}&=B_{y_1}(z,c)\ol{c_{j-2}}^2+\ol{B_{y_1}(z,c)}{c_{j-2}}^2+\sum_{
k\in\PP_j}|c_k|^2C_{y_1}(z,c)+\frac{1}{n}F_{y_1}(z,c)\notag
\\
\wt
R_{\mix,x_2}&=B_{x_2}(z,c)\ol{c_{j+2}}^2+\ol{B_{x_2}(z,c)}{c_{j+2}}^2+\sum_{
k\in\PP_j}|c_k|^2C_{x_2}(z,c)+\frac{1}{n}F_{x_2}(z,c)\notag
\\
\wt
R_{\mix,y_2}&=B_{y_2}(z,c)\ol{c_{j+2}}^2+\ol{B_{y_2}(z,c)}{c_{j+2}}^2+\sum_{
k\in\PP_j}|c_k|^2C_{y_2}(z,c)+\frac{1}{n}F_{y_2}(z,c)\notag
\\
\wt R_{\mix,c_k}&=\ii c_k \wt P(z,c)+\frac{1}{n}F_{c_k}(z,c)\,\,\,\text{ for
}k\neq j\pm
2\notag\\
\wt R_{\mix,c_{j\pm 2}}&=\ii c_{j\pm 2} \wt P(z,c)-\ii \ol c_{j\pm 2}\wt
Q_\pm(z,c)+\frac{1}{n}F_{j\pm 2}(z,c),\notag
\end{align}
where the functions $B_z$ and $C_z$ satisfy
\begin{align*}
B_{x_1}(z,c)=\OO\left(x_1+y_1x_2z_2\right)&&B_{x_2}(z,
c)=\OO\left(x_2+y_2x_1z_1\right)\\
B_{y_1}(z,c)=\OO\left(y_1+x_1y_2z_2\right)&&B_{y_2}(z,
c)=\OO\left(y_2+x_2y_1z_1\right)\\
C_{x_1}(z,c)=\OO\left(x_1+y_1x_2z_2\right)&&C_{x_2}(z,
c)=\OO\left(x_2+y_2x_1z_1\right)\\
C_{y_1}(z,c)=\OO\left(y_1+x_1y_2z_2\right)&&C_{y_2}(z,
c)=\OO\left(y_2+x_2y_1z_1\right)
\end{align*}
the functions $D_z$ and $D_c$ satisfy
\[
 F_{x_i}=\OO\left(\sum_{k\in\PP_j}|c_k|^2(x_1+x_2)\right),\,\,F_{y_i}
=\OO\left(\sum_{k\in\PP_j}|c_k|^2(y_1+y_2)\right),\,\,F_{c_k}(z,
c)=\OO\left(\sum_{k\in\PP_j}|c_k|(x_i+y_i)^2\right)
\]
and $\wt P$ and $\wt Q_\pm$ satisfy
\[
\wt P(z,c)=\OO\left(xy\right),\,\,\,\wt
Q_-(z,c)=\OO\left(x_1,y_1\right)  \,\,\,\text{ and }\,\,\,\wt
Q_+(z)=\OO\left(x_2,y_2\right).
\]
Moreover, the function $\wt\Psi$ satisfies
\[
\begin{split}
\wt\Psi_{x_1}&=\OO\left(x_1^3,x_1y_1,x_1(x_2^2+y_2^2),
y_1y_2(x_2+y_2),c_{j-2}^2y_1,\sum_{k\in\PP}|c_k|^2y_1y_2^2,
\frac{1}{n}\sum_{k\in\PP_j}|c_k|^2y_i\right)\\
\wt\Psi_{y_1}&=\OO\left(y_1^3,x_1y_1,y_1(x_2^2+y_2^2),
x_1x_2(x_2+y_2),c_{j-2}^2x_1,\sum_{k\in\PP}|c_k|^2x_1x_2^2,
\frac{1}{n}\sum_{k\in\PP_j}|c_k|^2x_i\right)\\
\wt\Psi_{x_2}&=\OO\left(x_2^3,x_2y_2,x_2(x_1^2+y_1^2),
y_1y_2(x_1+y_1),c_{j+2}^2y_1,\sum_{k\in\PP}|c_k|^2y_2y_1^2,
\frac{1}{n}\sum_{k\in\PP_j}|c_k|^2y_i\right)\\
\wt\Psi_{y_2}&=\OO\left(y_2^3,x_2y_2,y_2(x_1^2+y_1^2),
x_1x_2(x_1+y_1),c_{j+2}^2x_1,\sum_{k\in\PP}|c_k|^2x_2x_1^2,
\frac{1}{n}\sum_{k\in\PP_j}|c_k|^2x_i\right).
\end{split}
\]
\end{lemma}
\begin{proof}
It is enough to compose two change of coordinates. The first change is the
change \eqref{def:Change:FullSystem:Hyp} considered in Lemma
\ref{lemma:FullModel:NormalForm}. The second one is the one which straightens
the invariant manifolds of a normally hyperbolic invariant manifold
\cite{Fenichel74, Fenichel77, HirschPS77}. Then, to obtain the required
estimates,
it suffices to combine Lemmas \ref{lemma:HypToyModel:NormalForm}
and \ref{lemma:FullModel:NormalForm} with the standard results about normally
hyperbolic invariant  manifolds.
\end{proof}

The change obtained in Lemma \ref{lemma:StraightenInvManifolds} straightens  the
stable and unstable invariant
manifolds of $\{x=0,y=0\}$. This allows us to perform the detailed study of the
transition map close to the saddle that we need. 
As in \cite{GuaKal},  we define a set $\wh \VV_j$ such that
\begin{equation}\label{Cond:SaddleMap:Inclusion}
 \Upsilon\circ\Xi\left(\VV_j\right)\subset \wh \VV_j,
\end{equation}
where  $\VV_j$ is the set defined in Lemma \ref{lemma:iterative:saddle},
$\Upsilon$ is the inverse of the
coordinate change $\Psi$ obtained in Lemma
\ref{lemma:StraightenInvManifolds} and $\Xi$ is the change of coordinates
defined in \eqref{def:Change:Heteroclinic}. Then, we apply the flow $\wh
\Phi^\tau$
associated to the vector field \eqref{def:VF:Straight} to points in $\wh\VV_j$.
To obtain the inclusion \eqref{Cond:SaddleMap:Inclusion} we define the function
$g_j(p_2,q_2)$ involved in the definition of $\VV_j$.

Define the set $ \wh\VV_j=\DD_1^{1}\times\ldots\times
\DD_j^{j-2}\times\wh\NNN_j\times
\DD_j^{j+2}\times\ldots\times \DD_j^{N}$, where  $\wh\NNN_j$ is the set defined
in \eqref{def:Hyp:ModifiedDomain} and
$\DD_j^{k}$ are defined as
\[
 \begin{split}
  \DD_j^{k}&=\left\{\left|c_{k}\right|\leq
M_{\el,\pm}\de^{(1-r)/2}\right\}\,\,\,\text{ for }k\in\PP_j^\pm\\
 \DD_j^{j\pm 2}&=\left\{\left|c_{j\pm 2}\right|\leq M_{\adj,\pm}\left(\wh
C^{(j)}\de\right)^{1/2}\right\}.
 \end{split}
\]
Define the function 
$g_j(p_2,q_2)$ involved in the definition of
the set $\VV_j$ as
\begin{equation}\label{def:Function_g}
 g_j(p_2,q_2)=p_2+a_p(\sigma) p_2+a_q(\sigma)q_2-x^\ast_2
\end{equation}
where $x_2^\ast$ is the constant defined in \eqref{def:ChoiceEtaHat} and
\begin{equation*}
\begin{split}
a_p(\sigma)&=\pa_{p_2}\wt \Upsilon_{p_2}(0,\sigma,0,0,0)\\
a_q(\sigma)&= \pa_{q_2}\wt \Upsilon_{p_2}(0,\sigma,0,0,0),
\end{split}
\end{equation*}
where  $\Upsilon=\mathrm{Id}+\wt\Upsilon$.

\begin{lemma}\label{lemma:FullModel:TiltedSection} With the above
notations for $\de$ small enough condition \eqref{Cond:SaddleMap:Inclusion} is
satisfied.
\end{lemma}
\begin{proof}
 It is a straightforward consequence of Lemmas
\ref{lemma:HypToyModel:NormalForm}
and \ref{lemma:StraightenInvManifolds}. 
\end{proof}

After straightening the invariant manifold, next lemma studies the saddle map in
the transformed variables for points belonging to $\VV_j$.

\begin{lemma}\label{lemma:FullModel:SaddleMap}
Let us consider the flow $\wh\Phi_\tau$ associated to \eqref{def:VF:Straight}
and a
point $(z^0,c^0)\in\wh\VV_j$. Then  for $\de$ and $\sigma$ small enough, the
point $(z^f,c^f)=\wh \Phi_{T_j}(z^0,c^0)$,
where $T_j=T_j(x_2^0)$ is the time defined in \eqref{def:FixedTimeSaddle},
satisfies
\[
\begin{split}
 |x_1^f|, |y_1^f|  \qquad  \qquad \leq &\ K_\sigma\left(\wh
C^{(j)}\de\right)^{1/2}\\
 |x_2^f-f_2(\sigma)| \qquad \leq &\ K_\sigma\de^{r'}\\
\left| y_2^f + \frac{f_1(\sigma)}{f_2(\sigma)}
\wh C^{(j)} \de \ln(1/\de) \right|
\leq & \  \frac{f_1(\sigma)}{f_2(\sigma)}\de.
\end{split}
\]
and
\[
\begin{split}
\left| c_k^f-c_k^0 e^{2\ii a_n T_j}\right|\leq &
K_\sigma\de^{(1-r)/2+r'}\,\,\,\text{
for }k\in\PP_j^\pm\\
\left| c_{j\pm 2}^f-c_{j\pm 2}^0 e^{2\ii a_n T_j}\right|\leq & 2M_{\adj,\pm}
\sigma
\left(\wh C^{(j)}\de\right)^{1/2}.
\end{split}
\]
\end{lemma}
The proof of this lemma follows the same lines as the analogous result in
\cite{GuaKal}. It is explained in  Section \ref{sec:ProofLemmaFixedTimeMap}

Now, to complete the proof of Lemma \ref{lemma:iterative:saddle} we need two
final steps. First we undo the change of coordinates
performed in Lemma \ref{lemma:StraightenInvManifolds} to express
the estimates of the saddle map in the original variables. The second step is to
adjust the time so that the image belongs to
the section $\Sigma_{j}^\out$. These two final steps are done in
the next two following lemmas.

They proofs follow the same lines as the proofs of Lemmas 6.5 and 6.6 in
\cite{GuaKal}. One only needs to take into account the extra terms appearing in
the change of coordinates $\Psi$, given in Lemma
\ref{lemma:StraightenInvManifolds}.

\begin{lemma}\label{lemma:FullModel:SaddleMap:Original:FixedTime}
Let us consider the flow $\Phi_\tau$ associated to
\eqref{def:VF:AfterSeparatrix} and a point $(P^0,Q^0,c^0)\in\Xi(\wh\VV_j)$,
where  $\Xi$ the change in \eqref{def:Change:Heteroclinic} and $\VV_j$ is  the
set considered in Theorem \ref{theorem:iterative}.
Then for $\de$ and $\sigma$ small enough, the point
$(P^f,Q^f,c^f)=\Phi_{T_j}(P^0,Q^0,c^0)$, where $T_j$ is the time defined in
\eqref{def:FixedTimeSaddle},  satisfies
\[
\begin{split}
 |P_1^f|, |Q_1^f| \qquad \leq &\ K_\sigma\left(\wh C^{(j)}\de\right)^{1/2}\\
|P_2^f- \sigma| \qquad\leq &\ K_\sigma\de^{r'}\\
| Q_2^f+ \wt
C^{(j)}\de \ln(1/\de)| \leq & \ \wt C^{(j)}\ \de\ K_\sigma.
\end{split}
\]
for certain constant $\wt C^{(j)}$ satisfying $C^{(j)}/2\leq \wt C^{(j)}\leq
2C^{(j)}$  and
\[
\begin{split}
\left| c_k^f-c_k^0 e^{2\ii a_n T_j}\right|\leq &
K_\sigma\de^{(1-r)/2+r'}\,\,\,\text{ for
}k\in\PP^{\pm}_j\\
\left| c_{j\pm 2}^f-c_{j\pm 2}^0 e^{2\ii a_n T_j}\right|\leq & 2M_{\adj,\pm}
\sigma
\left(\wh C^{(j)}\de\right)^{1/2}.
\end{split}
\]
\end{lemma}

Once we have obtained good estimates for the  approximate time map in the
original variables, we adjust it to obtain image points belonging to the section
$\Sigma_j^\out$.
\begin{lemma}\label{lemma:Saddle:AdjustingSection}
 Let us consider a point $ \left(P^f,Q^f,c^f\right)\in
\Phi^{T_j}\circ\Xi(\VV_j)$,
 where $\Phi^\tau$ is the flow of \eqref{def:VF:AfterSeparatrix}, $T_j$
 is the time defined in \eqref{def:FixedTimeSaddle}, $\Xi$ the change in
\eqref{def:Change:Heteroclinic} and $\VV_j$ is
 the set considered in Theorem \ref{theorem:iterative}.

Then, there exists a time $T'$, which depends on the point
$\left(P^f,Q^f,c^f\right)$, such that
\[
\left(P^\ast,Q^\ast,c^\ast\right)=
\Phi^{T'}\left(P^f,Q^f,c^f\right)\in\Sigma_{j}^\out.
\]
Moreover, there exists a constant $K_\sigma$ such that
\[
|T'|\leq K_\sigma\de^r
\]
and
\[
\begin{split}
\left|c_k^\ast-c_k^f\right|&\leq K_\sigma \de^{1-r}\,\,\,\text{ for }k\in\PP_j\\
\left|P_1^\ast-P_1^f\right|&\leq K_\sigma
\left(C^{(j)}\de\right)^{1/2}\de^{1-r}\\
\left|Q_1^\ast-Q_1^f\right|&\leq K_\sigma
\left(C^{(j)}\de\right)^{1/2}\de^{1-r}\\
P_2&=\sigma\\
\left|Q_2^\ast-Q_2^f\right|&\leq K_\sigma C^{(j)}\de^{2-r}\ln(1/\de).
\end{split}
\]
\end{lemma}

To finish the proof of Lemma \ref{lemma:iterative:saddle}, it  is enough to 
undo the change \eqref{def:Change:Heteroclinic} and to proceed as in
\cite{GuaKal}. Recall that the change \eqref{def:Change:Heteroclinic} only
alters two coordinates.

\subsection{Proof of Lemma
\ref{lemma:FullModel:SaddleMap}}\label{sec:ProofLemmaFixedTimeMap}
It follows the same lines as the proof of Lemma 6.4 in \cite{GuaKal}. We only
need to check that the additional terms are small enough so that the fixed point
argument goes through. We make the  variation of
constants change of coordinates
\begin{equation}\label{def:VariationOfConstants}
 x_i=e^{\la_n \tau}u_i,\,\,y_i=e^{-\la_n \tau}v_i,\,\,c_k=e^{2\ii  a_n \tau}s_k
\end{equation}
to obtain the integral equation
\begin{equation}\label{def:Full:ConstantVariation}
\begin{split}
u_i&=x_i^0+\int_0^{T_j} e^{-\la_n \tau}\left( R_{\hyp,x_i}\left(u e^{\la_n
\tau},v
e^{-\la_n \tau}\right)+\wt R_{\mix,x_i}\left(u e^{\la \tau},v
e^{-\la_n \tau},se^{2\ii  a_n \tau}\right)\right)d\tau\\
v_i&=y_i^0+\int_0^{T_j} e^{\la_n \tau}\left( R_{\hyp,y_i}\left(u e^{\la_n
\tau},v
e^{-\la_n \tau}\right)+\wt R_{\mix,y_i}\left(u e^{\la_n \tau},v
e^{-\la_n \tau},se^{2\ii  a_n \tau}\right)\right)d\tau\\
s_k&=c_k^0+\int_0^{T_j} e^{-2\ii  a_n \tau}\left( \ZZZ_{\el,c_k}\left(se^{2\ii 
a_n \tau}\right)+\wt
R_{\mix,c_k}\left(u e^{\la_n \tau},v e^{-\la_n \tau},se^{2\ii  a_n
\tau}\right)\right)d\tau.
\end{split}
\end{equation}
The terms $ R_{\hyp,z}$ are the ones considered in Section
\ref{sec:HyperbolicToyModel}. So we use the properties of these
functions obtained in that section. We use the  integration time $T_j$
introduced in 
\eqref{def:FixedTimeSaddle}.

We use \eqref{def:Full:ConstantVariation} to set up a fixed point
argument in two steps. First we define $\GG=(\GG_\hyp,\GG_\el)$ as
\[
\begin{split}
 \GG_{\hyp,u_i}(u,v,s)&=x_i^0+\int_0^{T_j}
e^{-\la_n \tau}\left(R_{\hyp,x_i}\left(u e^{\la_n \tau},v
e^{-\la_n \tau}\right)+\wt R_{\mix,x_i}\left(u e^{\la_n \tau},v
e^{-\la_n \tau},se^{2\ii a_n \tau}\right)\right)d\tau\\
&=\FF_{\hyp,u_i}(u,v)+\int_0^{T_j} e^{-\la_n \tau}\wt R_{\mix,x_i}\left(u
e^{\la_n \tau},v e^{-\la_n \tau},se^{2\ii a_n \tau}\right)d\tau\\
 \GG_{\hyp,v_i}(u,v,s)&=
y_i^0+\int_0^{T_j}e^{\la_n \tau}\left(R_{\hyp,y_i}\left(u e^{\la_n \tau},v
e^{-\la_n \tau}\right)+\wt  R_{\mix,x_i}\left(u e^{\la_n \tau},
e^{-\la_n \tau},se^{2\ii a_n \tau}\right)\right)d\tau\\
&= \FF_{\hyp,v_i}(u,v)+\int_0^{T_j} e^{\la_n \tau}\wt R_{\mix,x_i}\left(u
e^{\la_n \tau},v e^{-\la_n \tau},se^{2\ii a_n \tau}\right)d\tau,\\
\end{split}
\]
where  $\FF_\hyp$ is the operator defined in \eqref{def:Hyp:Operator1}, and
\[
 \GG_{\el,c_k}(u,v,s)=c_k^0+\int_0^{T_j} e^{-2\ii a_n \tau}\left(
\ZZZ_{\el,c_k}\left(se^{2\ii a_n \tau}\right)+\wt R_{\mix,c_k}\left(u e^{\la_n
\tau},v
e^{-\la_n \tau},se^{2\ii a_n \tau}\right)\right)d\tau.
\]
We proceed as  in Section
\ref{sec:HyperbolicToyModel}, and we modify it by defining
\[
\begin{split}
\wt
\GG_{\hyp,u_1}(u,v,s)&=\GG_{\hyp,u_1}(u_1,\GG_{\hyp,v_1}(u,v,s),\GG_{\hyp,u_2}(u
,v,s),v_2,s)\\
\wt
\GG_{\hyp,v_1}(u,v,s)&=\GG_{\hyp,v_1}(u,v,s)\\
\wt
\GG_{\hyp,u_2}(u,v,s)&=\GG_{\hyp,u_2}(u,v,s)\\
\wt
\GG_{\hyp,v_2}(u,v,s)&=\GG_{\hyp,v_2}(u_1,\GG_{\hyp,v_1}(u,v,s),\GG_{\hyp,u_2}(u
,v,s),v_2,s)\\
\wt\GG_\el(u,v,s)&=\GG_\el(u,v,s)
\end{split}
\]
which will be contractive. We denote the new operator by
\begin{equation}\label{def:FullSystem:OperatorModified}
 \wt \GG=\left(\wt \GG_{\hyp,u_1},\wt  \GG_{\hyp,u_2},\wt \GG_{\hyp,v_1},\wt
\GG_{\hyp,v_2},\wt \GG_\el\right),
\end{equation}
whose fixed points coincide with those of $\GG$.

We extend the norm defined in \eqref{def:Hyp:Norms}, as in \cite{GuaKal}, by
defining
\[
\begin{split}
\|h\|_{\el,\pm}&=\left(M_{\el,\pm}\de^{(1-r)/2}\right)^{-1}\|h\|_\infty \\
\|h\|_{\adj,\pm}&=M_{\adj,\pm}\iii\left(\wh
C^{(j)}\de\right)^{-1/2}\|h\|_\infty
\end{split}
\]
and
\[
\|(u,v,s)\|_\ast=\sup_{\substack{k\in
\PP_j^\pm\\i=1,2}}\Big\{\|u_i\|_{\hyp,u_i},\|v_i\|_{\hyp,v_i},\|s_k\|_{\el,\pm},
\|s_{j\pm 2}\|_{\adj,\pm}\Big\}
\]
which, abusing notation, is denoted as the norm in \eqref{def:Hyp:FullNorm}. We
also define the Banach space
\[
\YY=\left\{(u,v,s):[0,T]\rightarrow \CC^{N-3}\times\R^4;
\|(u,v,s)\|_\ast<\infty\right\}.
\]
We state the two following propositions, which imply the contractivity of $\wt
\GG$. The proof of the first one is straightforward
taking into account the definition of $\wt\GG$ and Lemma
\ref{lemma:Hyp:FirstIteration}. The proof of the second one is deferred to
end of the section.

\begin{proposition}\label{lemma:Full:FirstIteration}
Let us consider the operator $\wt \GG$ defined in
\eqref{def:FullSystem:OperatorModified}. Then, the components of $\wt\GG(0)$ are
given by
\[
 \begin{split}
  \wt\GG_{\hyp,u_1}(0)&=\wt\FF_{\hyp,u_i}(0)\\
 \wt\GG_{\hyp,v_1}(0)&=y_1^0\\
 \wt\GG_{\hyp,u_2}(0)&=x_2^0\\
 \wt\GG_{\hyp,v_2}(0)&=\wt\FF_{\hyp,v_2}(0)\\
 \wt\GG_{\el,c_k}(0)&=c_k^0.
\end{split}
\]
Thus, there exists a constant $\kk_1>0$  independent of
$\sigma$, $\de$ and $j$ such that the operator $\wt \GG$ satisfies
\[
\left \|\wt \GG (0)\right\|_\ast\leq \kk_1.
\]
\end{proposition}

\begin{proposition}\label{lemma:Full:Contractive}
Let us consider $w_1, w_2\in B(2\kk_1)\subset\YY$, a constant $r'$
satisfying $0<r'<\ln 2/\ga-2r$ and  $\delta$ as defined in
Theorem \ref{thm:ToyModelOrbit}. Then taking  $\sigma$ small enough
and $N$ big enough such that $0<\de=e^{-\ga N}\ll 1$,   there exist
a constant $K_\sigma>0$ which is independent of $j$ and $N$, but
might depend on $\sigma$, and a constant $K$ independent of $j$, $N$
and $\sigma$, such that the operator $\wt \GG$ satisfies
\[
 \begin{split}
&\left\|\wt\GG_{\hyp,u_i}(u,v,s)-\wt\GG_{\hyp,u_i}(u',v',s')\right\|_{\hyp,u_i,
v_i}\leq \\
&\qquad\qquad\qquad\leq K_\sigma\de^{r'}\left\|(u,v,s)-(u',v',s')\right\|_\ast\\
&\left\|\wt\GG_{\hyp,v_i}(u,v,s)-\wt\GG_{\hyp,v_i}(u',v',s')\right\|_{\hyp,u_i,
v_i}\leq \\
&\qquad\qquad\qquad\leq K_\sigma\de^{r'}\left\|(u,v,s)-(u',v',s')\right\|_\ast\\
&\left\|\wt\GG_{\el,c_k}(u,v,s)-\wt\GG_{\el,c_k}(u',v',s')\right\|_{\el,\pm}\leq
\\
&\qquad\qquad\qquad\leq K_\sigma\de^{r'}
\left\|(u,v,s)-(u',v',s')\right\|_\ast,\,\,\,\,\,\text{ for }k\in\PP_j^\pm\\
&\left\|\wt\GG_{\adj,\pm}(u,v,s)-\wt\GG_{\adj,
\pm}(u',v',s')\right\|_{\adj,\pm}\leq \\
&\qquad\qquad\qquad\leq  K\sigma\left\|(u,v,s)-(u',v',s')\right\|_\ast.
\end{split}
\]
Thus, since $0<\de\ll \sigma$,
\[
\left \|\wt \GG(w_2)-\wt \GG(w_1)\right\|_\ast\leq 2K\sigma\|w_2-w_1\|_\ast
\]
and therefore, for   $\sigma$ small enough, it is contractive.
\end{proposition}

The previous two propositions show that the operator $\wt\GG$ is contractive.
Let us denote by $(u^*,v^*,s^*)$ its unique fixed point in the ball
$B(2\kk_1)\subset\YY$. Now,  it only remains to obtain
the estimates stated in Lemma \ref{lemma:FullModel:SaddleMap}.
The estimates for the hyperbolic variables are obtained as in \cite{GuaKal}: it
is enough to undo the change of coordinates \eqref{def:VariationOfConstants} and
to recall the definition of the norm \ref{def:Hyp:Norms}. For the elliptic ones
it is enough
to take into account that
\[
\begin{split}
 c_k^f&=c_k(T_j)=s_k(T_j) e^{2\ii a_n T_j}\\
&=\GG_{\el,c_k}(0)(T_j)e^{2\ii a_n
T_j}+\left(\GG_{\el,c_k}(u^*,v^*,s^*)(T_j)-\GG_{\el,c_k}
(0)(T_j)\right)e^{2\ii a_n T_j}\\
&=c_k^0e^{2\ii a_n
T_j}+\left(\GG_{\el,c_k}(u^*,v^*,s^*)(T_j)-\GG_{\el,c_k}(0)(T_j)\right)e^{2\ii
a_n T_j}
\end{split}
\]
and bound the second term using the Lipschitz constant obtained in Proposition
\ref{lemma:Full:Contractive}.

We finish the section by proving Proposition \ref{lemma:Full:Contractive},
which 
completes the proof of Lemma \ref{lemma:FullModel:SaddleMap}.

\begin{proof}[Proof of Proposition \ref{lemma:Full:Contractive}]
As we have done in the proof of Proposition \ref{lemma:Hyp:Contractive},
first, we stablish bounds for any $(u,v,s)\in B(2\kk_1)\subset\YY$ in
the supremmum norm, which will be used to bound the Lipschitz constant of
each component of $\wt\GG$. Indeed, if $(u,v,s)\in B(2\kk_1)\subset\YY$,
it satisfies 
\[
 \begin{split}
 |u_1|&\leq K_\sigma\wh C^{(j)}\de\ln(1/\de)\\
 |v_1|&\leq K\sigma\\
|u_2|&\leq K_\sigma \left(\wh C^{(j)}\de\right)^{1/2}\\
|v_2|&\leq K_\sigma\left(\wh C^{(j)}\de\right)^{1/2}\ln(1/\de),
\end{split}
\]
where $K>0$ is a constant independent of $\sigma$,  and
\[
 \begin{split}
  |s_k|&\leq K_\sigma\de^{(1-r)/2}\,\,\,\,\text{ for }\,\,\,k\in\PP_j^\pm\\
  |s_{j\pm 2}|&\leq K_\sigma\left(\wh C^{(j)}\de\right)^{1/2}\leq
K_\sigma\de^{(1-r)/2}.
 \end{split}
\]
We bound the Lipschitz constant for each component of $\wt\GG_\el$. We split
each component of the operator between the elliptic, hyperbolic and mixed part.
For the elliptic part the additional terms with respect to the toy model in
\cite{GuaKal} are of the same type as the terms in \cite{GuaKal} (plus an extra
$n^{-1}$). Therefore, they can be bounded as done in \cite{GuaKal} to obtain 
\[
\left\|\int_0^{T_j} e^{-2\ii a_n \tau} \left(\ZZZ_{\el,c_k}
\left(s_ke^{2\ii a_n \tau}\right)-\ZZZ_{\el,c_k}  \left(s'e^{2\ii a_n
\tau}\right)\right)
dt\right\|_{\el,\pm}\leq K_\sigma\de^{1-r} N T_j\|(u,v,s)-(u',v',s')\|_\ast.
\]
and
\[
\left\|\int_0^{T_j} e^{-2\ii a_n \tau} \left(\ZZZ_{\el,c_{j\pm 2}}
\left(se^{2\ii a_n \tau}\right)-\ZZZ_{\el,c_{j\pm 2}}  \left(s'e^{2\ii a_n
\tau}\right)\right)
d \tau\right\|_{\adj,\pm}\leq K_\sigma\de^{1-r} N
T_j\|(u,v,s)-(u',v',s')\|_\ast.
\]
Now we bound the mixed terms. We can write them as $\wt R_{\mix,c_k}=\wt
R_{\mix,c_k}^0+\wt R_{\mix,c_k}^1$, where $\wt R_{\mix,c_k}^0$ is the order in
first in $n\iii$, that is, it is the term considered in \cite{GuaKal}, and $\wt
R_{\mix,c_k}^1$ contains the rest. In \cite{GuaKal} it is seen that
\[
\begin{split}
\left\|\int_0^{T_j} e^{-2\ii a_n \tau}\left(\wt R^0_{\mix,c_k}
(ue^{\la_n \tau},ve^{-\la_n \tau},se^{2\ii a_n \tau})-\wt R^0_{\mix,c_k}
(u'e^{\la_n \tau},v'e^{-\la_n \tau},s'e^{2\ii a_n \tau})\right)d
\tau\right\|_{\el,\pm}\\
\quad\quad\quad\quad\quad\quad\quad\quad\quad\leq K_\sigma \wh
C^{(j)}\de\ln^3(1/\de)\|(u,v,s)-(u',v',s')\|_\ast.
\end{split}
\]
for non adjacent modes and 
\[
\begin{split}
\left\|\int_0^{T_j} e^{-2\ii a_n \tau}\left(\wt R^0_{\mix,c_{j\pm2}}
\left(ue^{\la_n \tau},ve^{-\la_n \tau},se^{2\ii a_n \tau}\right)-\wt
R^0_{\mix,c_{j\pm2}}
\left(u'e^{\la_n \tau},v'e^{-\la_n \tau},s'e^{2\ii a_n \tau}\right)\right)d
\tau\right\|_{\adj,-}
\\
\quad\quad\quad\quad\quad\quad\quad\quad\quad\leq
K\sigma\|(u,v,s)-(u',v',s')\|_\ast,
\end{split}
\]
where $K>0$ is a constant independent of $\sigma$.

Now we bound the term  $\wt R^1_{\mix,c_k}$ stated in Lemma
\ref{lemma:StraightenInvManifolds}. We can see that for either for non-adjacent
elliptic modes,
\[
 \begin{split}
  \left\|\wt R^1_{\mix,c_k} \right.&\left.
\left(ue^{\la_n \tau},ve^{-\la_n \tau},se^{2\ii a_n \tau}\right)-\wt
R^1_{\mix,c_k}
\left(u'e^{\la_n \tau},v'e^{-\la_n \tau},s'e^{2\ii a_n
\tau}\right)\right\|_{\el,\pm}\\
\leq &K_\sigma n\iii\sum_{i=1,2}\left(\|u_i-u_i'\|_{\hyp,u_i}+\|v_i-v_i'\|_{
\hyp,v_i}\right)\\
&+K_\sigma n\iii
\sum_{\ell\in\PP_j^\pm}\left\|s_\ell-s_\ell'\right\|_{\el,\pm}\\
&\leq K_\sigma N n\iii\left\|(u,v,s)-(u',v',s')\right\|_\ast.
 \end{split}
\]
For the adjacent modes, recalling the bounds for $C^{(j)}$ in
\eqref{cond:GrowthOfCs}, we have that 
\[
 \begin{split}
  \left\|\wt R^1_{\mix,c_{j\pm 2}} \right.&\left.
\left(ue^{\la_n \tau},ve^{-\la_n \tau},se^{2\ii a_n \tau}\right)-\wt
R^1_{\mix,c_{j\pm 2}}
\left(u'e^{\la_n \tau},v'e^{-\la_n \tau},s'e^{2\ii a_n
\tau}\right)\right\|_{\el,\pm}\\
&\leq K_\sigma N \de^{-2r} n\iii\left\|(u,v,s)-(u',v',s')\right\|_\ast.
 \end{split}
\]
Therefore, using that $\de=e^{-\ga N}$ and \eqref{def:FixedTimeSaddle}, we have
that for $k\in \PP^{\pm}_j$,
\[
\begin{split}
\left\|\int_0^{T_j} e^{-2\ii a_n \tau}\left(\wt R_{\mix,c_k}
(ue^{\la_n \tau},ve^{-\la_n \tau},se^{2\ii a_n \tau})-\wt R_{\mix,c_k}
(u'e^{\la_n \tau},v'e^{-\la_n \tau},s'e^{2\ii a_n \tau})\right)d
\tau\right\|_{\el,\pm}\\
\quad\quad\quad\quad\quad\quad\quad\quad\quad\leq K_\sigma n\iii
\ln^2(1/\de)\|(u,v,s)-(u',v',s')\|_\ast.
\end{split}
\]
and for the adjacent modes
\[
\begin{split}
\left\|\int_0^{T_j} e^{-2\ii a_n \tau}\left(\wt R_{\mix,c_k}
(ue^{\la_n \tau},ve^{-\la_n \tau},se^{2\ii a_n \tau})-\wt R_{\mix,c_k}
(u'e^{\la_n \tau},v'e^{-\la_n \tau},s'e^{2\ii a_n \tau})\right)d
\tau\right\|_{\el,\pm}\\
\quad\quad\quad\quad\quad\quad\quad\quad\quad\leq K_\sigma
n\iii\de^{-2r}\ln^2(1/\de)\|(u,v,s)-(u',v',s')\|_\ast.
\end{split}
\]
So, using the definition the properties of $r$ and $r'$ stated in Lemma
\ref{lemma:iterative:saddle}, we can conclude that either for $k\in\PP_j^\pm$,
\[
\left\|\GG_{\el,c_k}(u,v,s)-\GG_{\el,c_k}(u',v',s')\right\|_{\el,\pm}\leq
K_\sigma\de^{r'}\|(u,v,s)-(u',v',s')\|_\ast.
\]
and for the adjacent modes
\[
\left\|\GG_{\el,c_{j-2}}(u,v,s)-\GG_{\el,c_{j-2}}(u',v',s')\right\|_{\adj,-}\leq
K\sigma\|(u,v,s)-(u',v',s')\|_\ast.
\]
Now we bound the Lipschitz constant for the hyperbolic components of the
operator. Note that we only need to bound the terms involving $\wt R_{\mix,z}$
since the other terms of the operator have been bounded in Proposition
\ref{lemma:Hyp:Contractive}. As for the elliptic modes we split them as  $\wt
R_{\mix,z}= \wt R^0_{\mix,z}+ \wt R^1_{\mix,z}$, where $\wt R^0_{\mix,z}$ is the
term in \cite{GuaKal} and $\wt R^1_{\mix,z}$ is the remainder which contains the
terms of order $\OO(n\iii)$.

 We start with the Lipschitz constants of
$\GG_{\hyp,v_i}$. In \cite{GuaKal} it is shown that 
\[
\begin{split}
& \left\|\int_0^{T_j} e^{\la_n \tau}\left(\wt R^0_{\mix, y_1}
\left(ue^{\la_n \tau},ve^{-\la_n \tau},se^{2\ii a_n \tau}\right)-\wt R^0_{\mix,
y_1}
\left(ue^{\la_n \tau},ve^{-\la_n \tau},se^{2\ii a_n \tau}\right)\right)d
\tau\right\|_{\hyp, v_1}\\
&\qquad\qquad\leq K_\sigma \de^{1-r} \ln^2(1/\de) \|(u,v,s)-(u',v',s')\|_\ast\\
& \left\|\int_0^{T_j} e^{\la_n \tau}\left(\wt R^0_{\mix, y_2}
\left(ue^{\la_n \tau},ve^{-\la_n \tau},se^{2\ii a_n \tau}\right)-\wt R^0_{\mix,
y_2}
\left(ue^{\la_n \tau},ve^{-\la_n \tau},se^{2\ii a_n \tau}\right)\right)d
\tau\right\|_{\hyp, v_2}\\
&\qquad\qquad\leq K_\sigma \de^{1/2-2r} \ln(1/\de) \|(u,v,s)-(u',v',s')\|_\ast.
\end{split}
\]
Now, using the bounds on $F_{y_i}$ given in Lemma
\ref{lemma:StraightenInvManifolds}, the definition of $T_j$ in
\eqref{def:FixedTimeSaddle} and the upper and lower bounds for $C^{(j)}$ in
\eqref{cond:GrowthOfCs}, we bound the $\wt R^1_{\mix, y_i}$ terms as
\[
\begin{split}
& \left|\int_0^{T_j} e^{\la_n \tau}\left(\wt R^1_{\mix, y_i}
(ue^{\la_n \tau},ve^{-\la_n \tau},se^{2\ii a_n \tau})-\wt R^1_{\mix, y_i}
(ue^{\la_n \tau},ve^{-\la_n \tau},se^{2\ii a_n \tau})\right)d \tau\right|\\
&\qquad\qquad\leq K_\sigma  N n\iii \de^{1/2-3r/2} \|(u,v,s)-(u',v',s')\|_\ast.
\end{split}
\]
Therefore, applying norms and using condition on $\de$ from
Theorem \ref{thm:ToyModelOrbit} and the condition on $r'$ in Lemma
\ref{lemma:iterative:saddle}, we obtain
\[
\begin{split}
& \left\|\int_0^{T_j} e^{\la_n \tau}\left(\wt R_{\mix, y_i}
\left(ue^{\la_n \tau},ve^{-\la_n \tau},se^{2\ii a_n \tau}\right)-\wt R_{\mix,
y_i}
\left(ue^{\la_n \tau},ve^{-\la_n \tau},se^{2\ii a_n \tau}\right)\right)d
\tau\right\|_{\hyp, v_1}\\
&\qquad\qquad\leq K_\sigma  N n\iii \de^{1/2-3r/2} \|(u,v,s)-(u',v',s')\|_\ast\\
&\qquad\qquad\leq K_\sigma \de^{r'}\|(u,v,s)-(u',v',s')\|_\ast\\
& \left\|\int_0^{T_j} e^{\la_n \tau}\left(\wt R_{\mix, y_i}
\left(ue^{\la_n \tau},ve^{-\la_n \tau},se^{2\ii a_n \tau}\right)-\wt R_{\mix,
y_i}
\left(ue^{\la_n \tau},ve^{-\la_n \tau},se^{2\ii a_n \tau}\right)\right)d
\tau\right\|_{\hyp, v_2}\\
&\qquad\qquad\leq K_\sigma \de^{-2r}n\iii\|(u,v,s)-(u',v',s')\|_\ast\\
&\qquad\qquad\leq K_\sigma \de^{r'} \|(u,v,s)-(u',v',s')\|_\ast.
\end{split}
\]
Then,  taking into account the results of Lemma \ref{lemma:Hyp:Contractive}, one
can conclude that
\[
\begin{split}
&\left\|\wt\GG_{\hyp,v_1}(u,v,s)-\wt\GG_{\hyp,v_1}(u',v',s')\right\|_{\hyp,v_1}
\leq\\&\qquad\qquad\qquad\leq K_\sigma
\de^{r'}\left\|(u,v,s)-(u',v',s')\right\|_\ast\\
&\left\|\wt\GG_{\hyp,v_2}(u,v,s)-\wt\GG_{\hyp,v_2}(u',v',s')\right\|_{\hyp,v_2}
\leq\\&\qquad\qquad\qquad\leq
K_\sigma\de^{r'}\left\|(u,v,s)-(u',v',s')\right\|_\ast.
\end{split}
\]
Proceeding in the same way, one can obtain that
\[
\begin{split}
&\left\|\wt\GG_{\hyp,u_1}(u,v,s)-\wt\GG_{\hyp,u_1}(u',v',s')\right\|_{\hyp,u_1}
\leq\\
&\qquad\qquad\qquad\leq  K_\sigma
\de^{r'}\left\|(u,v,s)-(u',v',s')\right\|_\ast\\
&\left\|\wt\GG_{\hyp,u_2}(u,v,s)-\wt\GG_{\hyp,u_2}(u',v',s')\right\|_{\hyp,u_2}
\leq \\
&\qquad\qquad\qquad\leq K_\sigma\de^{r'}\left\|(u,v,s)-(u',v',s')\right\|_\ast.
\end{split}
\]
This completes the proof.
\end{proof}

\section{Study of the global map: proof of Lemma
\ref{lemma:iterative:hetero}}\label{sec:ProofHeteroMap}

We devote this section to prove Lemma \ref{lemma:iterative:hetero}. It follows
the same lines as the proof of Lemma 4.8 in \cite{GuaKal}. The main difference
is that now the heteroclinic connection is not straightened in the original
coordinates and therefore we use the coordinates $(P,Q)$, obtained in Lemma
\ref{lemma:SeparatrixAsGraph}, to prove Lemma \ref{lemma:iterative:hetero}. 
Recall that the initial
section $\Sigma^{\out}_{j}$, defined in \eqref{def:Section2Saddle},
is expressed in the variables
adapted to the $j^{th}$ saddle, that is $(p_1^{(j)}, q_1^{(j)},  p_2^{(j)},
q_2^{(j)},
c^{(j)})$, whereas the final section $\Sigma^{\inn}_{j+1}$,
defined in \eqref{def:Section1Saddle}, is expressed in the variables adapted to
the  $(j+1)^{st}$ saddle, that is   $(p_1^{(j+1)}, q_1^{(j+1)},  p_2^{(j+1)},
q_2^{(j+1)},
c^{(j+1)})$. The change of variables between these two system of coordinates is
given in \cite{GuaKal}, and stated in the next lemma. To simplify
notation we define
\[
( p_1, q_1, p_2, q_2, c)=
\left(p_1^{(j)},q_1^{(j)}, p_2^{(j)}, q_2^{(j)}, c^{(j)}\right)
\]  and
\[
\left(\wt p_1,\wt q_1, \wt p_2, \wt q_2, \wt c\right)=
\left(p_1^{(j+1)},q_1^{(j+1)}, p_2^{(j+1)},  q_2^{(j+1)},
c^{(j+1)}\right)
\]
and we
denote by $\Theta^j$ the change of coordinates that relates them, namely
\[
 (\wt p_1,\wt q_1, \wt p_2, \wt q_2, \wt c)=\Theta^j( p_1, q_1, p_2, q_2, c).
\]

\begin{lemma}\label{lemma:ChangeOfSaddle}
The change of coordinates $\Theta^j$ is given by
\begin{align*}
\Theta_{\wt c_k}^j( p_1, q_1, p_2, q_2, c)&=\frac{\ol \omega q_2+\omega p_2}{\wt
r\sqrt{2\Im(\om^2)}}c_k &\text{ for }k\in \PP_{j+1}^\pm \cup \{j+3\}\\
\Theta_{\wt c_{j-1}}^j( p_1, q_1, p_2, q_2, c)&=\frac{\ol \omega q_2+\omega
p_2}{\wt r\Im(\om^2)}\left(\omega q_1+\ol\omega p_1\right)\\
\Theta_{\wt p_1}^j( p_1, q_1, p_2, q_2, c)&= \frac{r}{\wt r}q_2\\
\Theta_{\wt q_1}^j( p_1, q_1, p_2, q_2, c)&= \frac{r}{\wt r}p_2\\
\Theta_{\wt p_2}^j( p_1, q_1, p_2, q_2, c)&= \frac{1}{2}\left(\frac{\Re z}{\Re
\om}+\frac{\Im z}{\Im\om}\right)\\
\Theta_{\wt q_2}^j( p_1, q_1, p_2, q_2, c)&= \frac{1}{2}\left(\frac{\Re z}{\Re
\om}-\frac{\Im z}{\Im\om}\right),
\end{align*}
where  $\omega$ has been defined in \eqref{def:anANDomega} and
\begin{align}
r^2=&1-\sum_{k\neq j-1,j,j+1}
|c_k|^2-\frac{1}{\Im(\om^2)}(p_1^2+q_1^2+2\Re(\om^2)p_1q_1)\notag\\
&-\frac{1}{\Im(\om^2)}(p_2^2+q_2^2+2\Re(\om^2)p_2q_2)\notag\\
\wt r^2=&\frac{1}{\Im(\om^2)}\left(p_2^2+q_2^2+2\Re(\om^2)p_2q_2\right)\notag\\
z=&\frac{c_{j+2}}{\wt r}\left(\ol \omega q_2+\omega p_2\right)\notag.
\end{align}
\end{lemma}

To prove Lemma \ref{lemma:iterative:hetero}, we want to use the system of
coordinates given in Lemma \ref{lemma:SeparatrixAsGraph} for the old variables.
We define
\[
( P_1, Q_1, P_2, Q_2)=
\left(P_1^{(j)},Q_1^{(j)}, P_2^{(j)}, Q_2^{(j)}\right)
\]  
For the new ones, we want to stick with $ (\wt p_1,\wt q_1, \wt p_2, \wt q_2,
\wt c)$ since those are the ones used to state Lemma
\ref{lemma:iterative:hetero}. We denote by $\wt\Theta^j$ the change of
coordinates that relates them, namely
\[
 (\wt p_1,\wt q_1, \wt p_2, \wt q_2, \wt c)=\Theta^j( P_1, Q_1, P_2, Q_2, c).
\]

\begin{corollary}\label{coro:ChangeOfSaddle}
The change of coordinates $\wt\Theta^j$ is given by
\begin{align*}
\wt\Theta_{\wt c_k}^j( P_1, Q_1, P_2, Q_2, c)&=\frac{\ol\omega
(Q_2+\xi(P_2))+\omega P_2}{\wt r\sqrt{\Im(\om^2)}}c_k \qquad\qquad\text{ for
}k\in \PP_{j+1}^\pm \cup \{j+3\}\\
\wt\Theta_{\wt c_{j-1}}^j( P_1, Q_1, P_2, q_2, c)&=\frac{\ol \omega
(Q_2+\xi(P_2))+\omega P_2}{\wt r\Im(\om^2)}\left( \omega Q_1+\ol \omega
(P_1+\xi(Q_1))\right)\\
\wt\Theta_{\wt p_1}^j( P_1, Q_1, P_2, Q_2, c)&= \frac{r}{\wt r}(Q_2+\xi(P_2))\\
\wt\Theta_{\wt q_1}^j( P_1, Q_1, P_2, Q_2, c)&= \frac{r}{\wt r}P_2\\
\wt\Theta_{\wt P_2}^j( P_1, Q_1, P_2, Q_2, c)&= \frac{1}{2}\left(\frac{\Re
z}{\Re \om}+\frac{\Im z}{\Im\om}\right)\\
\wt\Theta_{\wt Q_2}^j( P_1, Q_1, P_2, Q_2, c)&= \frac{1}{2}\left(\frac{\Re
z}{\Re \om}-\frac{\Im z}{\Im\om}\right)\notag,
\end{align*}
where  
\begin{align*}
r^2=&1-\sum_{k\neq j-1,j,j+1}
|c_k|^2-\frac{1}{\Im(\om^2)}
((P_1+\xi(Q_1))^2+Q_1^2+2\Re(\om^2)(P_1+\xi(Q_1))Q_1)\notag\\
&-\frac{1}{\Im(\om^2)}
(P_2^2+(Q_2+\xi(P_2))^2+2\Re(\om^2)P_2(Q_2+\xi(P_2)))\notag\\
\wt
r^2=&\frac{1}{\Im(\om^2)}
\left(p_2^2+(Q_2+\xi(P_2))^2+2\Re(\om^2)P_2(Q_2+\xi(P_2))\right)\notag\\
z=&\frac{c_{j+2}}{\wt r}\left(\ol\omega (Q_2+\xi(P_2))+\omega P_2\right)\notag.
\end{align*}
\end{corollary}

Note that in the new variables, we will need to check that the sets we obtain in
the final section are close to the separatrix defined in Lemma
\ref{lemma:SeparatrixAsGraph}. This will be a consequence of the next lemma.
\begin{lemma}\label{lemma:Global:PropertyXi}
The function $\xi$ introduced in Lemma \ref{lemma:SeparatrixAsGraph} satisfies
\[
 \xi(\wt q_1)=\frac{r_0}{\wt r_0}\xi\left(\frac{\wt r_0}{r_0}\wt q_1\right)
\]
where $r_0$ and $\wt r_0$ are defined by the following equations
\begin{align*}
\wt r_0^2=&\frac{1}{\Im(\om^2)}\left(\left(\frac{\wt r_0}{r_0}\wt
q_1\right)^2+\xi^2(\frac{\wt r_0}{r_0})+2\Re(\om^2)\frac{\wt r_0}{r_0}\wt
q_1\xi\left(\frac{\wt r_0}{r_0}\right)\right)\notag\\
r_0^2=&1-\wt r_0^2\notag\\
\end{align*}
\end{lemma}
\begin{proof}
Note that the separatrix we are traveling close to is defined by $q_2=\xi(p_2)$.
Applying the change obtained in Lemma \ref{lemma:ChangeOfSaddle}, we obtain that
in the new variables it must satisfy
\[
 \wt p_1=\frac{r_0}{\wt r_0}\xi\left(\frac{\wt r_0}{r_0}\wt q_1\right)
\]
where $r_0$ and $\wt r_0$ are just the functions $r$ and $\wt r$ introduced in
Lemma \ref{lemma:ChangeOfSaddle}  evaluated over the separatrix. Moreover, using
that the hyperbolic toy model at each saddle is the same, we know that in the
new variables the separatrix can be parameterized as a graph as $\wt p_1=\xi(\wt
q_1)$. Since the graph parameterization is unique, we obtain the formula stated
in the lemma.
\end{proof}

Now, we express the section $\Sigma_{j+1}^{\inn}$ in the variables $(P_1,Q_1,
P_2,  Q_2, c)$
using the change $\wt\Theta^j$ obtained in Lemma \ref{coro:ChangeOfSaddle}.

\begin{corollary}\label{coro:SectionOldVariables}
Fix $\sigma>0$ and define the set
\[
 \wt \Sigma_{j+1}^{\inn}=\left(\wt\Theta^j\right)^{-1}
\left(\Sigma_{j+1}^{\inn}\cap\WW_{j+1}\right),
\]
where $\Sigma^{\inn}_{j+1}$ is the section defined in
\eqref{def:Section1Saddle} and
\[
 \WW_{j+1}=\left\{|P_1|\leq \eta, |Q_1|\leq \eta,
|Q_2|\leq \eta, |c_k|\leq \eta\,\,\text{for }k\in\PP_j^\pm
\,\text{ and }k=j\pm 2\right\}.
\]
Then, for $\eta>0$ small enough, $ \WW_{j+1}$
can be expressed  as a graph as
\[
P_2=w(P_1,Q_1,  Q_2, c).
\]
Moreover, there exist constants $\kk',\kk''$
independent of $\eta$   satisfying
\[
 0<\kk'<\sqrt{\Im(\om^2)-\sigma^2}<\kk''<1
\]
such that, for any  $(P_1,Q_1,Q_2,c)\in\WW_{j+1}$,
the function $w$ satisfies
\[
\kk'< w(P_1,Q_1,  Q_2, c)<\kk''.
\]
\end{corollary}

Once we have defined the section $\wt\Sigma_{j+1}^{\inn}$,
we can define the map
\[
\begin{array}{cccc}
  \wt\BB_\glob^j:&\Xi(\UU_j)\subset\Sigma_j^{\out}&\longrightarrow
&\wt\Sigma_{j+1}^{\inn}\\
&(P_1,Q_1,Q_2,c)&\mapsto&\wt\BB_\glob^j(P_1,Q_1,Q_2,c)
\end{array}
 \]
induced by the flow \eqref{def:VF:AfterSeparatrix}. Thanks to Corollary
\ref{coro:SectionOldVariables}, one can easily deduce
that the time $T_{\wt \BB_\glob^j}=T_{\wt \BB_\glob^j}(Q_1,P_1,P_2,c)$ spent by
the map $\wt\BB_\glob^j$ for any point
$(Q_1,P_1,P_2,c)\in\Xi(\UU_j)\subset\Sigma_j^\out$ is  independent of $\de$, $j$
and $N$. Since the
difference between $\wt\BB_\glob^j$ and $\BB_\glob^j$ is just a change of
coordinates, we have that the time  spent  by $\BB_\glob^j$ is the same
 $T_{\wt \BB_\glob^j}$.

Now we study the behavior of the map $\wt \BB_\glob^j$. 
\begin{proposition}\label{prop:HeteroMap}
Let us consider a parameter set $\wt \II_j$ (as defined in Definition
\ref{definition:ModifiedProductLike}) and a $\wt \II_j$-product-like set
$\UU_j$. Then, there exists a constant $\wt K_\sigma$ independent of $j$,  $N$
and $\de$ and a constant $D^{(j)}$ satisfying
\[
\wt C^{(j)} /\wt K_\sigma\leq D^{(j)}\leq \wt K_\sigma \wt C^{(j)},
\]
such that the set $\wt \BB_\glob^j\circ\Xi(\UU_j)\subset \wt\Sigma_{j+1}^\inn$,
where $\Xi$ is the change defined in \eqref{def:Change:Heteroclinic}, satisfies
the following conditions:
\begin{description}
 \item[\textbf{C1}]
\[
\wt \BB_\glob^j\circ\Xi(\UU_j)\subset
\wh\DD_j^1\times\ldots\times\wh\DD_j^{j-2}\times \SSS_j\times
\wh\DD_j^{j+2}\times\ldots\times\wh\DD_j^{N}
\]
where
\begin{align*}
\wh\DD_j^k&=\left\{\left|c_k\right|\leq  \left(\wt M^{(j)}_{\el,\pm}+\wt
K_\sigma\de^{r'}\right)
\de^{(1-r)/2}\right\} \,\,\text{ for }k\in \PP_j^\pm\\
\wh\DD_j^{j\pm 2}&\subset\left\{\left|c_{j\pm2}\right|\leq
\wt K_\sigma\wt M^{(j)}_{\adj,\pm} \left(\wt C^{(j)}\de\right)^{1/2}\right\},
\end{align*}
and
\[
\begin{split}
\SSS_{j}= \Big\{& (P_1,Q_1,P_2,Q_2)\in \R^4:
|P_1|,|Q_1|\leq \wt K_\sigma\wt M_\hyp^{(j)}\left(\wt C^{(j)}\de\right)^{1/2},\\
&P_2=w(P_1,Q_1,  Q_2, c),  - D^{(j)}\,\de\,\left(\ln(1/\de)-\wt
K_\sigma\right)\leq
Q_2^{(j)}\leq -D^{(j)}\,\de\,\left(\ln(1/\de)+\wt K_\sigma\right) \Big\},
\end{split}
\]
\item[\textbf{C2}] Let us define the projection $\wt
\pi(P,Q,c)=(P_2,Q_2,c_{j+2},\ldots,c_N)$. Then,
\[
\begin{split}
\left[- D^{(j)}\,\de\,(\ln(1/\de)-1/\wt K_\sigma), -
D^{(j)}\,\de\,(\ln(1/\de)+1/\wt K_\sigma)\right]\times  \{P_2=w(P_1,Q_1,  Q_2,
c)\}\times
\DD_{j,-}^{j+2}\times\ldots\times\DD_{j,-}^{N} \\
\qquad\subset\wt\pi\left(\wt \BB_\glob^j\circ\Xi(\UU_j)\right)
\end{split}
\]
where
\begin{align*}
\DD_{j,-}^k&=\left\{\left|c_k^{(j)}\right|\leq  \left(\wt m^{(j)}_{\el}-\wt
K_\sigma \de^{r'}\right)
\de^{(1-r)/2}\right\} \,\,\text{ for }k\in \PP_j^+\\
\DD_{j,-}^{j+ 2}&=\left\{\left|c_{j+2}^{(j)}\right|\leq
\wt m^{(j)}_{\adj} \left(C^{(j)}\de\right)^{1/2}/\wt K_\sigma\right\}.
\end{align*}
\end{description}
\end{proposition}
This proposition is proved for the cubic nonlinear Schr\"odinger equation toy
model in \cite{GuaKal}. One can easily check that the prove is also valid for
the vector field \eqref{def:VF:AfterSeparatrix}. Therefore, the prove in
\cite{GuaKal} also applies to our setting.

Now we complete the proof of Lemma \ref{lemma:iterative:hetero}. We need to show
that the set $\BB_\glob^j(\UU_j)\subset \Sigma_{j+1}^\inn$ satisfies similar
properties to the ones of the set $\wt \BB_\glob^j\circ\Xi(\UU_j)$ and also to
obtain a parameter set $\II_{j+1}$ and $\II_{j+1}$-product like   set
$\VV_j\subset \Sigma_{j+1}^\inn$ which satisfies condition
\eqref{cond:ComposeMaps:Heteromap}. These two last steps are summarized in the
next lemma.

\begin{lemma}\label{lemma:HeteroMapOriginalVars}
Let us consider a parameter set  $\II_{j+1}$ whose constants
satisfy
\[
 \begin{split}
  D^{(j)}/2\leq C^{(j+1)}\leq 2 D^{(j)}\\
0<m_{\hyp}^{(j+1)}\leq \wt m_\hyp^{(j)}
 \end{split}
\]
and
\[
 \begin{split}
M_{\el,-}^{(j+1)}&=\max\left\{\wt M_{\el,-}^{(j)}+\wt K_\sigma\de^{r'},\wt
K_\sigma\wt M_{\adj,-}^{(j)}\right\}\\
M_{\el,+}^{(j+1)}&=\wt M_{\el,+}^{(j)}+\wt K_\sigma\de^{r'}\\
m_{\el}^{(j+1)}&=\wt m_{\el}^{(j)}-\wt K_\sigma\de^{r'}\\
m_{\adj,+}^{(j+1)}&=\wt m_{\el,+}^{(j)}+\wt K_\sigma\de^{r'}\\
M_{\adj,-}^{(j+1)}&=\wt K_\sigma\wt M_{\hyp}^{(j)}\\
m_{\adj}^{(j+1)}&=\wt m_{\el}^{(j)}+\wt K_\sigma\de^{r'}\\
M_{\hyp}^{(j+1)}&=\max\left\{\wt K_\sigma \wt M_{\adj,+}^{(j)},\wt
K_\sigma\right\}.
 \end{split}
\]
Then, the set
\[
 \VV_{j+1}=\BB_\glob^j(\UU_j)\cap\left\{g_j(p_2^{(j+1)},q_2^{(j+1)})=0\right\}
\cap \left\{\left|c_{j+3}^{(j+1)}\right|\leq
M_{\adj,+}^{(j+1)}\left(C^{(j+1)}\de\right)^{1/2}\right\},
\]
where $g_j$ is the function defined in \eqref{def:Function_g}, is a
$\II_{j+1}$-product-like set and satisfies condition
\eqref{cond:ComposeMaps:Heteromap}
\end{lemma}

\begin{proof}
It is enough to apply the change of coordinates $\Theta^j$ given in Lemma
\ref{lemma:ChangeOfSaddle} and take into account  that  $\xi$ satisfies the
equality given by Lemma \ref{lemma:Global:PropertyXi} and $|\xi|=\OO(n^{-1})$.
\end{proof}

\section{The  approximation argument: proof of  Theorem \ref{thm:approximation}}
 \label{sec:Approximation}

Write the equation  associated to Hamiltonian \eqref{def:HamRotating} as
\begin{equation}  \label{eq:AfterNFInRotating}
- \ii \dot r_n = \EE_n(r)+ \wt\RRR_n(r),
\end{equation}
where $\EE$ is the function defined in \eqref{eq:HRes} and $\wt\RRR$ is the
vector field associated to the Hamiltonian $\RRR'$ defined in
\eqref{def:RemainderInRotating}. We want to study the closeness of the orbit $\mathtt r^
\rrr(t)$ obtained in \eqref{def:SolutionTruncatedSystem}, which is a solution of
$-\ii\dot{\mathtt r}^\rrr=\EE(\mathtt r^\rrr)$, with an orbit $\wt r(t)$ of  equation
\eqref{eq:AfterNFInRotating} which satisfies $\|\wt r(0)-\mathtt r^
\rrr(0)\|_{\ell^1}\leq \rrr^{-5/2}$. Define the function $\xi$ as
\begin{equation}  \label{def:Approx:Error}
 \xi=\wt r-\mathtt r^\rrr,
\end{equation}
which satisfies $\|\xi(0)\|_{\ell^1}\leq \rrr^{-5/2}$. We proceed as in
\cite{GuaKal} and we apply Gronwall-like estimates to bound the $\ell^1$ norm of
$\xi(t)$.

The equation for $\xi$ can be written as
$\dot \xi =\ZZZ^0(t)+\ZZZ^1(t)\xi+\ZZZ^2(\xi,t)$
with
\begin{align}
 \ZZZ^0(t)=&\wt\RRR\left(\mathtt r^\rrr\right)  \label{def:Approx:FirstIteration}\\
\ZZZ^1(t)=&D\EE\left(\mathtt r^\rrr\right)  \label{def:Approx:Linear}\\
\ZZZ^2(\xi,t)=&\EE\left(\mathtt r^\rrr+\xi\right)-\EE\left(\mathtt r^\rrr\right)-D\EE\left(r^
\rrr\right)\xi+
\wt\RRR\left(\mathtt r^\rrr+\xi\right)-\wt\RRR\left(\mathtt r^\rrr\right) 
\label{def:Approx:Higher}
\end{align}
Applying the $\ell^1$ norm to this equation, we obtain
\begin{equation}  \label{eq:Approx:Difference}
 \frac{d}{dt}\|\xi\|_{\ell^1}\leq \left\|\ZZZ^0(t)\right\|_{\ell^1}+\left\|
\ZZZ^1(t)\xi\right\|_{\ell^1}+\left\| \ZZZ^2(\xi,t)\right\|_{\ell^1}.
\end{equation}
The next three lemmas give estimates for each term in the right hand side of
this equation. 
\begin{lemma}  \label{lemma:Approx:BoundFirstIteration}
The function $\ZZZ^0$
 defined in \eqref{def:Approx:FirstIteration}
 satisfies
$ \left\|\ZZZ^0\right\|_{\ell^1}\leq C \rrr^{-(2d+1)} 2^{(2d+1)N}$.
\end{lemma}
The proof of this lemma is analogous to the proof of Lemma B.1 in \cite{GuaKal}.

\begin{lemma}  \label{lemma:Approx:BoundLinear}
The linear operator $\ZZZ^1(t)$
satisfies
$\left\| \ZZZ^1(t)\xi\right\|_{\ell^1}\leq C \rrr^{-(2d-2)}2^{N(2d-2)}
\|\xi\|_{\ell^1}$
\end{lemma}
\begin{proof}
 Taking into account the definition of $\EE$ in \eqref{eq:HRes}, we have that 
\[
 \left\| \ZZZ^1(t)\xi\right\|_{\ell^1}\leq \|\mathtt r^\rrr
\|_{\ell^1}^{2d-2}\|\xi\|_{\ell^1}
\]
For each $t\in [0,T]$, we have that there exists $j^*$  such that, for any $k\in
\SSS_{j^*}$,   $|\mathtt r^\rrr_{k}|\leq \rrr$. For any other $j$ and $k\in\SSS_{j^*}$,
 $|\mathtt r^\rrr_{k}|\leq \rrr\de^\nu$. Recall that $\mathtt r^\rrr_{k}=0$ for all
$k\not\in\SSS$. Then, since $|\SSS_j|\leq 2^{N-1}$, we have that 
 $\|\mathtt r^\rrr \|_{\ell^1}\lesssim  \rrr^{-1} 2^{N-1}$,
which implies  $\left\| \ZZZ^1(t)\xi\right\|_{\ell^1}\lesssim C
\rrr^{-(2d-2)}2^{N(2d-2)}\|\xi\|_{\ell^1}$.
\end{proof}

To obtain estimates for  $\ZZZ^2(\xi,t)$ defined in  \eqref{def:Approx:Higher},
we apply a bootstrap argument as done in  \cite{CKSTT}.  Assume that for
$0<t<T^*$ we have
\begin{equation}  \label{cond:Bootstrap}
 \|\xi(t)\|_{\ell^1}\leq C \rrr^{-3/2}2^{N}.
\end{equation}
For $t=0$ we know that it is already satisfied since $\|\xi(0)\|_{\ell^1}\leq
\rrr^{-5/2}$. \emph{A posteriori} we will show that the time $T$ in
\eqref{def:Time:Rescaled} satisfies $0<T<T^*$ and therefore
the bootstrap assumption holds.
\begin{lemma}  \label{lemma:Approx:BoundHigher}
Assume that condition \eqref{cond:Bootstrap} is satisfied. Then
the  operator $\ZZZ^2(\xi,t)$
satisfies
\[
 \left\| \ZZZ^2(\xi,t)\right\|_{\ell^1}\leq  C
\rrr^{-(2d-2)-1/2}2^{N(2d-2)}\|\xi\|_{\ell^1}.
\]
\end{lemma}
\begin{proof}
The proof of this lemma follows the same lines as the proof of Lemma B.3 in
\cite{GuaKal}. We split $\ZZZ^2$ in \eqref{def:Approx:Higher} as
$\ZZZ^2=\ZZZ^{21}+\ZZZ^{22}$ with
\[
\begin{split}
\ZZZ^{21}(\xi,t)=&\EE\left(\mathtt r^\rrr+\xi\right)-\EE\left(r^
\rrr\right)-D\EE\left(\mathtt r^\rrr\right)\xi\\
\ZZZ^{22}(\xi,t)=&\wt\RRR\left(\mathtt r^\rrr+\xi\right)-\wt\RRR\left(\mathtt r^\rrr\right).
\end{split}
\]
By the definition of $\EE$ in  \eqref{eq:HRes}, we have that
\[
 \|\ZZZ^{21}\|_{\ell^1}\leq C\sum^{2d-1}_{j=2}\|r^
\rrr\|_{\ell^1}^{2d-1-j}\|\xi\|_{\ell^1}^j.
\]
In the proof of Lemma \ref{lemma:Approx:BoundLinear}, we have seen that $\|r^
\rrr\|_{\ell^1}\leq  \rrr^{-1} 2^{N-1}$. Using this estimate and the bootstrap
assumption \eqref{cond:Bootstrap} we obtain
\[
 \|\ZZZ^{21}\|_{\ell^1}\lesssim  \rrr^{-(2d-2)-1/2}2^{N(2d-2)}\|\xi\|_{\ell^1}.
\]
Proceeding analogously one can see that  $\|\ZZZ^{22}\|_{\ell^1}\lesssim
\rrr^{-2d}2^{2Nd}\|\xi\|_{\ell^1}$.
Since we assume that $ \rrr^{-2d}2^{2N}\ll 1$, these two estimates imply the
statement of the lemma.
\end{proof}

We apply the estimates obtained in these three lemmas and the bootstrap
assumption \eqref{cond:Bootstrap}  to equation \eqref{eq:Approx:Difference}. We
obtain
\[
 \frac{d}{dt}\|\xi\|_{\ell^1}\leq C \rrr^{-(2d+1)}2^{N(2d+1)}+C
\rrr^{-(2d-2)}2^{N(2d-2)} \|\xi\|_{\ell^1}.
\]
We apply Gronwall estimates. We take $\|\xi\|_{\ell^1}=\zeta e^{C
\rrr^{-(2d-2)}2^{N(2d-2)}t}$ and therefore
\[
 \dot\zeta\leq  \dot\zeta  e^{C \rrr^{-(2d-2)}2^{N(2d-2)}t}\leq C
\rrr^{-(2d+1)}2^{N(2d+1)}.
\]
Integrating and taking into account the estimates for $T$ in
\eqref{def:Time:Rescaled} and that
$\|\zeta(0)\|_{\ell^1}=\|\xi(0)\|_{\ell^1}\leq C \rrr^{-5/2}$, we have that for
$t\in [0,T]$,
\[
 \|\zeta(t)\|_{\ell^1}\leq \|\zeta(0)\|_{\ell^1}+C \rrr^{-(2d+1)}2^{N(2d+1)}T
\leq C \rrr^{-5/2}+C \rrr^{-3}2^{N(d+3)}N^2\leq  \rrr^{-5/2}.
\]
Then, using again the estimate for $T$ in \eqref{def:Time:Rescaled}, for $t\in
[0,T]$,
\[
  \|\xi(t)\|_{\ell^1}\leq  \rrr^{-5/2}e^{C \rrr^{-(2d-2)}2^{N(2d-2)}T}\leq
\rrr^{-5/2}e^{C2^{dN}N^2}.
\]
Since we have assumed that $\rrr\geq\rrr_0=e^{C 2^{dN}N^2}$, we obtain that 
 $\|\xi(t)\|_{\ell^1}\leq  \rrr^{-3/2}$ for all $t\in [0,T]$. This completes the
proof of Theorem \ref{thm:approximation}.

\bibliography{references}
\bibliographystyle{alpha}
\end{document}